\documentclass[12pt,reqno,a4paper]{amsart}

\usepackage[utf8]{inputenc}
\usepackage[english]{babel}
\usepackage{amsmath,amsthm,amsfonts,amssymb,latexsym}
\usepackage[bookmarksopen,hidelinks, 
pdftitle={Character degrees in principal blocks for distinct primes}, 
pdfauthor={Annika Bartelt},
pdfstartview={FitH}, 
pdfpagemode=FullScreen,pdfkeywords={Character degrees, principal blocks, two primes}
]{hyperref}
\usepackage[shortlabels]{enumitem} %
\usepackage{cases}
\usepackage{booktabs} %
\usepackage{multicol}
\usepackage{multirow}
\usepackage{diagbox}
\usepackage{makecell}
\usepackage{array}
\newcolumntype{M}[1]{>{\centering\arraybackslash}m{#1}}
\usepackage{csquotes} 
\usepackage{graphicx}
\usepackage{mathcomp}
\usepackage{mathtools}
\usepackage{verbatim} %
\usepackage{bbm} %
\usepackage{thmtools, thm-restate} %

\usepackage{pdflscape} %

\usepackage[colorinlistoftodos, textsize=scriptsize, color=red!60]{todonotes} %

\theoremstyle{plain} %

\newtheorem{theorem}{Theorem}[section]

\newtheorem{lemma}[theorem]{Lemma}
 
\newtheorem{proposition}[theorem]{Proposition}
\newtheorem{corollary}[theorem]{Corollary}

\numberwithin{figure}{theorem}

\newtheorem*{conjNRS}{Conjecture (NRS)}

\theoremstyle{plain}
\newtheorem{thmalph}{Theorem}
\newtheorem{coroalph}[thmalph]{Corollary}

\newcommand{\blocktheorem}[1]{%
  \csletcs{old#1}{#1}%
  \csletcs{endold#1}{end#1}%
  \RenewDocumentEnvironment{#1}{o}
    {\par\addvspace{1.5ex}
     \noindent\begin{minipage}{\textwidth}
     \IfNoValueTF{##1}
       {\csuse{old#1}}
       {\csuse{old#1}[##1]}}
    {\csuse{endold#1}
     \end{minipage}
     \par\addvspace{1.5ex}}
}

\blocktheorem{thmalph}
\blocktheorem{coroalph}

\theoremstyle{definition}
\newtheorem{definition}[theorem]{Definition}

\newtheorem{example}[theorem]{Example}
 
\newtheorem*{notation}{Notation}
\newtheorem*{convention}{Convention}

\theoremstyle{remark}
\newtheorem{remark}[theorem]{Remark}
\newtheorem*{remark*}{Remark}

\newcommand\proofsymbol{\frame{\rule[0pt]{0pt}{7pt}\rule[0pt]{7pt}{0pt}}}

\newcommand{\R}{\mathbb R}

\newcommand{\Q}{\mathbb Q}

\newcommand{\Z}{\mathbb Z}
\newcommand{\F}{\mathbb F}

\newcommand{\eps}{\varepsilon}

\newcommand{\GL}[2]{{ \operatorname{GL} }_{ #1 } ( #2 )} 
 
\newcommand{\SL}[2]{{ \operatorname{SL} }_{ #1 } ( #2 )}
\newcommand{\SU}[2]{{ \operatorname{SU} }_{ #1 } ( #2 )} 
\newcommand{\PSL}[2]{{ \operatorname{PSL} }_{ #1 } ( #2 )}
\newcommand{\PGL}[2]{{ \operatorname{PGL} }_{ #1 } ( #2 )}
\newcommand{\PSU}[2]{{ \operatorname{PSU} }_{ #1 } ( #2 )}

\newcommand{\alt}[1]{\mathfrak{A}_{#1}}
\newcommand{\sym}[1]{\mathfrak{S}_{#1}}

\newcommand{\irr}[1]{\mathrm{Irr}(#1)} %

\def\irrprime#1#2{{\mathrm{Irr}}_{#1'}(#2)} %

\def\blprinc#1#2{ {B_{ #1 } ( #2 )}}

\newcommand{\GG}{\mathbf G}

\newcommand{\PP}{\mathbf P}

\DeclareMathOperator{\Uch}{Uch}
\newcommand{\tw}[1]{{}^#1\!}
\newcommand\type[1]{\operatorname{#1}} %
\renewcommand{\PP}{P} %
\newcommand{\QQ}{Q} %

\newcommand{\ord}{\mathrm{ord}}

\def\zent#1{{\bf Z}(#1)}

\newcommand{\aut}[1]{\operatorname{Aut}( #1 )} %

\newcommand{\res}[2]{{#1}_{\mkern 1mu \vrule height 1.7ex\mkern2mu #2}}

\newcommand{\sgn}{\mathrm{sgn}}

\newcommand{\floor}[1]{\left\lfloor #1 \right\rfloor}
\newcommand{\abs}[1]{\vert#1\vert}

\newcommand\reallywidehat[1]{%
	\savestack{\tmpbox}{\stretchto{%
			\scaleto{%
				\scalerel*[\widthof{\ensuremath{#1}}]{\kern-.6pt\bigwedge\kern-.6pt}%
				{\rule[-\textheight/2]{1ex}{\textheight}}%
			}{\textheight}%
		}{0.5ex}}%
	\stackon[1pt]{#1}{\tmpbox}%
}

\author{Annika Bartelt}
\address{FB Mathematik, RPTU Kaiserslautern--Landau, Postfach 3049, 67653 Kaisers\-lautern, Germany}
\email{bartelt.a@rptu.de}%
\title[Character degrees in principal blocks for distinct primes]{Character degrees in principal blocks for distinct primes}%
\date{\today}

\begin{document}
\begin{abstract}
	Let $G$ be a finite group of order divisible by two distinct primes $p$ and $q$. We show that $G$ possesses a non-trivial irreducible character of degree not divisible by~$p$ nor $q$ lying in both the principal $p$- and $q$-block whenever $G$ is one of the following: an alternating group $\alt n$, $n\geq 4$, a symmetric group $\sym n$, $n\geq 3$,
	or a finite simple classical group of type $\type{A}$, $\type{B}$, or $\type C$, defined in characteristic distinct from $p$ and $q$.
	This extends earlier results of Navarro--Rizo--Schaeffer Fry for $2\in\{p,q\}$, and in particular completes the proof of an instance of a conjecture of the same authors, e.g., in the case of symmetric and alternating groups.
\end{abstract}

\maketitle

\section*{Introduction}
Let $G$ be a finite group and $p$ be a prime. A common question in character theory is the following: Which properties of $G$ are encoded in the irreducible characters of $G$? Many group-theoretic properties of $G$ can indeed be read off its character table, for example, one can detect a normal abelian Sylow $p$-subgroup, by the famous It\^{o}–Michler Theorem from 1986 \cite[Theorem 2.3]{Michler}.
Another conjectural instance of this phenomenon is the following conjecture of Navarro, Rizo and Schaeffer Fry from 2022. They relate the absence of certain non-trivial characters of $G$ to the existence of a certain subgroup of $G$.\\
Let $\irrprime p G$ denote the set of irreducible characters of $G$ of degree not divisible by~$p$ and $\blprinc p G$ the principal $p$-block of $G$. Note that the trivial character~$\mathbbm{1}_G$ of $G$ lies in $\irrprime p {\blprinc{p}{G}}\coloneqq \irrprime p G \cap \irr{\blprinc{p}{G}}$, and we have $\irrprime p {\blprinc{p}{G}}\neq \{\mathbbm{1}_G\}$ if and only if $p$ divides $\abs G$, by \cite[Lemma 2.1]{NRS}. In the same article, the authors propose the following:
\begin{conjNRS} \cite[Conjecture A]{NRS}
	Let $G$ be a finite group and $p$ and $q$ distinct primes. If $\irrprime p {\blprinc p G} \cap \irrprime q {\blprinc q G}=\{\mathbbm{1}_G\}$, then there exists a Sylow $p$-subgroup $P$ of $G$ and a Sylow $q$-subgroup $Q$ of $G$ such that $[P,Q]=\{1\}$.
\end{conjNRS}
\noindent Thus, they predict an existence criterion for nilpotent Hall $\{p,q\}$-subgroups of~$G$.
For a purely group theoretical criterion, which is still encoded in the character table, see \cite[Theorem B]{nilpqHall}.
Conjecture (NRS) is related to a conjecture of Liu et al. who conjecture the same conclusion given the stronger assumption $\irr {\blprinc p G} \cap \irr {\blprinc q G}=\{\mathbbm{1}_G\}$ \cite[Conjecture 1.3]{LWXZ}. Thus, the above conjecture implies Liu et al.'s conjecture. Note that the converse of both conjectures is false as the dihedral group $D_{2pq}$ of order $2pq$ with odd primes $p$, $q$ provides a counterexample.\newline
Using the classification of finite simple groups (CFSG), both Liu et al. and \linebreak Navarro--Rizo--Schaeffer~Fry reduce their conjectures to the same problem for almost simple groups $A$ (with additional structural properties), see \cite[Theorem~1.4]{LWXZ} and \cite[Theorem D (a)]{NRS}. Again using the CFSG, they also prove the reduced conjectures in many cases, in particular if $2\in\{p,q\}$, see \cite[Theorem 1.5]{LWXZ} and \cite[Theorem~D~(b)]{NRS}. However, there are still some  minor issues in the case of symmetric and alternating groups in Conjecture (NRS) which are addressed later in this article. In total, both conjectures remain open if the socle of $A$ is a finite simple classical group and the primes $p$ and $q$ are odd.

In this paper we show the following main theorem, thus contributing to the solution of Conjecture~(NRS).

\begin{thmalph}\label{thm:A}
	Let $G$ be a finite group, $p$ and $q$ distinct primes such that $pq$ divides~$\abs G$. Then we have $\irrprime p {\blprinc p G} \cap \irrprime q {\blprinc q G}\neq \{\mathbbm{1}_G\}$ if $G$ is one of the following:
	\begin{enumerate}[\rm(a)]
		\item a symmetric group $\sym n$ of degree $n\geq 3$, %
		\item an alternating group $\alt n$ of degree $n\geq 4$, %
		\item a finite simple group of Lie type $\type{A}_{n-1}$ or $\tw{2}{\type A_{n-1}}$, with $n\geq 2$, type $\type{B}_n$, with $n\geq 2$, or type $\type{C}_n$, with $n\geq 3$, defined in a characteristic distinct from $p$ and $q$.
	\end{enumerate}
	In particular, Conjecture {\rm{\textup(NRS\textup)}} holds for these groups.
\end{thmalph}
\noindent Thus, we reprove (sometimes in a slightly different, more generic, way) and extend earlier results of Navarro--Rizo--Schaeffer~Fry in the case $2\in\{p,q\}$, by generalizing many of their ideas. 
Theorem~\ref{thm:A} was first announced at a conference in Florence in June 2024 and also presented in Aachen in December 2024. Furthermore, statements (a) and (b) were already presented in Aachen in December 2022.

We also derive the following Corollary. %
\begin{coroalph} \label{coro:B}
	Let $n\geq 5$. Conjecture {\rm{(NRS)}} holds true for all almost simple groups with socle \textup(isomorphic to\textup) $\alt n$.
\end{coroalph}

 An auxiliary result  for symmetric groups is Theorem~\ref{thm:C}, the characterization whenever there are only linear characters (i.e., the trivial and the sign character $\sgn$) in $\irrprime p {\blprinc{p}{\sym n}} \cap \irrprime q {\blprinc{q}{\sym n}}$. 

\begin{restatable*}{thmalph}{thmpqlinAnSn} \label{thm:C}
	Let $n\geq 4$ and $p$ and $q$ be two distinct primes such that $q<p\leq n$. Then, we have $\irrprime p {\blprinc{p}{\sym n}} \cap \irrprime q {\blprinc{q}{\sym n}}=\{\mathbbm{1}_{\sym n},\sgn\}$ if and only if $q=2$ and one of the following holds:
	\begin{itemize}
		\item $n=9$ and $p=3$, or
		\item $n=p$ is a Fermat prime, or
		\item $n-1=p$ is a Mersenne prime.
	\end{itemize}
	In these cases, we have $\irrprime p {\blprinc{p}{\sym n}} \cap \irrprime 2 {\blprinc{2}{\sym n}}=\irrprime{\{2,p\}}{\sym n}$.
	Moreover, there exist distinct irreducible characters $\psi_{+}, \psi_{-}\in \irr {\alt n}$ of~$\alt n$ inducing the same irreducible character of~$\sym n$ such that $\irrprime p {\blprinc{p}{\alt n}} \cap \irrprime 2 {\blprinc{2}{\alt n}}=\{\mathbbm{1}_{\alt n},\psi_+, \psi_-\}$.%
\end{restatable*}
\noindent Theorem~\ref{thm:C} is similar to \cite[Theorem 2.8]{GSV} where the blocks are not prescribed, and the two results turn out to be equivalent.

Let us come back to Theorem~A and Conjecture (NRS). Notice that, for an almost simple group $A$ the strategy to show Conjecture~(NRS) is almost always\footnote{For the sporadic groups $A\in\{J_1, J_4\}$ with $\{p,q\}=\{3,5\}$ for $J_1$, or $\{p,q\}=\{5,7\}$ for $J_4$, we have $\irr{\blprinc p A} \cap \irr {\blprinc q A}=\{\mathbbm{1}_A\}$ by \cite[Proposition 3.5]{BesZha}. This can easily be checked with \cite{GAP}.}, like here, a proof by contraposition. Thus, we have to find a non-trivial irreducible character $\chi$ lying in $\irrprime p {\blprinc p A} \cap \irrprime q {\blprinc q A}\setminus \{\mathbbm{1}_A\}$. This is now a two-sided problem, one has to look for an irreducible character $\chi$ lying in the principal blocks for both primes $p$ and $q$ (block conditions) which at the same time has degree not divisible by $p$ nor $q$ (degree conditions).\newline
In their paper, Giannelli--Schaeffer Fry--Vallejo examine the set $\irrprime {\{p,q\}} G$ for any finite group $G$, where $\irrprime {\{p,q\}} G\coloneqq \irrprime p G \cap \irrprime q G$. They show, using the CFSG, that $\irrprime {\{p,q\}} G \neq \{\mathbbm{1}_G\}$ if and only if $G\neq \{1\}$ is non-trivial \cite[Theorem A]{GSV}. This isolates the degree  conditions. \newline
The other part of our problem, the block problem, is addressed, e.g., for symmetric and alternating groups in \cite[Proposition 2.1]{BMO} and \cite[Proposition 3.2]{BesZha}, respectively. Here, the authors explicitly construct a non-trivial character lying in both the principal \linebreak $p$-~and $q$-block.
Moreover, for any non-abelian finite simple group $G$, except for some sporadic groups\footnote{Here, the exceptions are again $G\in\{J_1, J_4\}$ with $\{p,q\}=\{3,5\}$ for $J_1$, and $\{p,q\}=\{5,7\}$ for $J_4$.}, Brough--Liu--Paolini show that the intersection $\irr {\blprinc{p}{G}}\cap \irr {\blprinc{q}{G}} $ is non-trivial \cite[Lemma 3.3]{BLP}. However, for many families of finite simple groups of Lie type, including the classical groups, their proof is not constructive, so there is no control on the degree of the chosen non-trivial unipotent character in $\irr {\blprinc{p}{G}}\cap \irr {\blprinc{q}{G}}$.
Note that both the results of Giannelli--Schaeffer Fry--Vallejo and Brough--Liu--Paolini make use of the CFSG. \newline
In general, bringing the block and degree conditions together requires more work.

Now, the strategy to prove Theorem~\ref{thm:A} is to explicitly construct a character of $G$ with the desired properties, going type-by-type and case-by-case. In Section~\ref{sec:1}, we deal with the situation $G\in\{\sym n, \alt n\}$ and then move on to classical $G$ in Section~\ref{sec:2}. Note that each type of group requires a separate case distinction.
In each part, we first develop the number-theoretical framework used in the proofs (extending ideas of NRS), and examine degrees and blocks of certain (unipotent) characters. For classical groups, the main ingredient in our proof is Lusztig’s combinatorial parametrization of unipotent characters together with known results on character degrees and blocks of such characters. We also develop number-theoretic tools for the analysis of unipotent character degrees. %

In this article we often use boxes to highlight both important assumptions or definitions.

\section*{Acknowledgments}
I dedicate this work to my late teacher Christine Bessenrodt who introduced me to the world of representation theory, with a particular emphasis on symmetric and alternating groups.
Moreover, I thank my PhD advisor Gunter Malle for careful reading of this manuscript and useful comments which lead to improvements of this article. This work was financially supported by the DFG (German Research Foundation), Project-ID~286237555.

\section{Symmetric and alternating groups}\label{sec:1}
In this part, we prove Theorem~\ref{thm:A} for symmetric and alternating groups. The first three sections are of preparatory nature, then we address the symmetric groups in Section~\ref{sec:sym} followed by the alternating groups in Section~\ref{sec:alt}, which also includes the proof of Corollary~\ref{coro:B}. The auxiliary Theorem~\ref{thm:C} is proven in Section~\ref{sec:smallintersection}.

\subsection{Background}\label{sec:background}
Here, we gather standard facts on the representation theory of symmetric groups.
The results recalled in this section can be found, e.g., in \cite{JK} or \cite{OlssonCRT}.

Let $n\geq 2$ be a positive integer.
A partition is a sequence $\lambda=(\lambda_1,\dots, \lambda_l)$ of positive integers $\lambda_i$, called \emph{parts} of~$\lambda$, such that $\lambda_1\geq \lambda_2\geq \dots\geq \lambda_l\geq 1$. We say that $\lambda$ is a partition of (size) $n$ if $\sum_{i=1}^l \lambda_i=n$, and we write $\lambda\vdash n$. The \emph{trivial partition} of~$n$ is $(n)$.

\begin{notation} We use the \emph{exponential notation} for a partition $\lambda=(1^{a_1}, 2^{a_2},\dots, n^{a_n})$ of size~$n$ where $a_i$ denotes the multiplicity of~$i$ appearing as a part of~$\lambda$, for all $1\leq i\leq n$. Exponents equal to~$1$ as well as parts $i$ with zero multiplicity are frequently (but not always) omitted. For notational convenience, we might allow a partition to have a (or multiple) part(s) equal to $0$. If we are given a partition, say $(1^{a_1},\dots,k^{a_k},\dots)$, in terms of a parameter $k$ with $1\leq k\leq n$, and we specialize, say $k=1$, the partition is understood to be $(1^{a_1+a_k}, \dots)$.
	Finally, note that in exponential notation parts of~$\lambda$ appear in increasing instead of decreasing order.
\end{notation}

Recall that the set of irreducible characters $\irr{\sym n}=\{\chi^\lambda\mid  \lambda \vdash n\}$ of the symmetric group $\sym  n$ is labeled naturally by partitions of~$n$ in the following canonical way: \newline
The two linear characters of~$\sym n$, the trivial character $\mathbbm{1}_{\sym n}$ and the sign character, denoted~$\sgn$, correspond to the partitions $(n)$ and $(1^n)$, respectively, i.e.,  we have $\mathbbm{1}_{\sym n}=\chi^{(n)}$ and $\sgn= \chi^{(1^n)}$.
Moreover, given a partition $\lambda \vdash n$, the degree $d^\lambda\coloneqq\chi^\lambda(1)$ of the corresponding character $\chi^\lambda$ of~$\sym n$ can be computed with a simple formula, the so-called \emph{hook formula} \cite[2.3.21~Theorem]{JK}, and is determined by the \emph{hook lengths} of~$\lambda$.
Then we say that the partition $\lambda$ itself has \emph{degree} $d^\lambda$.

Given a prime $p$ and partitions $\lambda,\mu\vdash n$, the classical ``Nakayama Conjecture'' (which is indeed a theorem) %
states that the characters $\chi^\lambda$ and $\chi^\mu$ lie in the same $p$-block of~$\sym n$ if and only if the \emph{$p$-cores} of the partitions $\lambda$ and $\mu$ agree \cite[6.1.21]{JK}. Thus, $p$-cores of partitions of $n$ label the $p$-blocks of~$\sym n$.
The \emph{$p$-core} of a partition is again a partition %
obtained from $\lambda$ by removing all its \emph{$p$-hooks}, see \cite[Section 2.7]{JK}.\newline
In particular, the character $\chi^\lambda$ belongs to the principal $p$-block of~$\sym n$ (i.e., the block of the trivial character $\mathbbm{1}_{\sym n}$) if and only if the $p$-core of~$\lambda$ equals the $p$-core of the partition~$(n)$. Note that the $p$-core of~$(n)$ is the partition $(s)$, if $s$ denotes the residue of~$n$ modulo $p$.
Also notice that the $p$-core of the partition $(1^{n})$ is $(1^s)$, hence the sign character lies in the block labeled by~$(1^s)$.

\begin{definition} \label{def:pprincSn} Let $n\geq 2$ and $p$ a prime. %
	We say that a partition $\lambda\vdash n$ is \emph{$p$-principal} for $\sym n$
	if the corresponding character $\chi^\lambda$ lies in the principal $p$-block of~$\sym n$.\newline
	If $q$ is a prime distinct from $p$, we say that $\lambda$ is \emph{$\{p,q\}$-principal} for $\sym n$ if $\lambda$ is $p$-principal as well as $q$-principal for $\sym n$.
\end{definition}

\subsection{Almost hook partitions}\label{sec:almhooks}
In this section, we look at partitions of a specific (almost hook like) shape, re-express their degree obtained from the hook formula, and in special situations determine whether
they are $p$-principal for a prime number~$p$. %

\begin{definition} \label{def:almhooks} Let $n\geq 2$ be a positive integer.
	For the purpose of this article, a partition $\lambda\vdash n$ of the form $ \lambda=(1^{n-k-\ell-1},k+1,\ell)$ with $0\leq k< \ell< n-k\leq n$,
	or $\lambda=(n)$, is called
	an \emph{almost hook partition} (of size $n$).\newline
	If $\lambda$ is an almost hook partition of size $n$ with $k=0$, or $\lambda=(n)$, it is called a \emph{hook partition} (of size $n$). %
	Note that for $k=0$ and $\ell\in\{0,1\}$, we obtain the same hook partition $(1^n)$.
\end{definition}

We apply the hook formula to express the degree of almost hook partitions in two different ways.
\begin{lemma}\label{lem:degalmosthookpartition}
	Let $n\geq 2$.
	The degree $d^\lambda$ of the almost hook partition
	$\lambda=(1^{n-k-\ell-1},k+1,\ell)$, where $0\leq k< \ell< n-k\leq n$, is
	$$d^\lambda=\frac{\prod\limits_{i=1}^{k} (n-k+i) \cdot \prod\limits_{\substack{i=1\\ i\neq \ell-k}}^{\ell} (n-k-i) }{\prod\limits_{i=1}^{k} i \cdot \prod\limits_{\substack{i=1 \\ i\neq \ell-k}}^\ell i }= \frac{\prod\limits_{\substack{i=1\\i\neq \ell-k}}^{\ell} (n-\ell+i) \cdot \prod\limits_{i=1}^{k} (n-\ell-i) }{\prod\limits_{\substack{i=1 \\ i\neq \ell-k}}^\ell i \cdot \prod\limits_{i=1}^{k} i  }.$$
	In particular, if $k=0$, we have $d^\lambda=\frac{\prod_{i=1}^{\ell-1} (n-i)}{\prod_{i=1}^{\ell-1}i}=\frac{\prod_{i=1}^{\ell-1} (n-\ell+i)}{\prod_{i=1}^{\ell-1}i}$.
\end{lemma}

\begin{proof}
	By the hook formula, a partition $\mu\vdash n$ has degree $
		d^\mu=\frac{\prod_{i=1}^n i}{\prod_{h} {h}},
	$ where $h$ runs through the multiset of hook lengths of~$\mu$ \cite[2.3.21 Theorem]{JK}.
	Here, the hook lengths %
	of~$\lambda=(1^{n-k-\ell-1},k+1,\ell)$ are as follows:
	$1,\dots,n-k-\ell-1$ and 
	$1,\dots,k$ and %
	$1,\dots,\widehat{\ell-k},\dots \ell$ and $n-\ell$ and $n-k$.
	By $\widehat{\ell-k}$ we mean that the number $\ell-k$ is omitted.
	Inserting the hook lengths into the previous formula, we get the degree
	\begin{align*}\tag{a} \label{eq:degalmhook}
		d^\lambda
		=\frac{\prod\limits_{i=1}^n i}{ (n-\ell)\cdot(n-k)\cdot \prod\limits_{i=1}^{n-k-\ell} i \cdot \prod\limits_{i=1}^{k} i \cdot \prod\limits_{\substack{i=1 \\ i\neq \ell-k}}^\ell i}
		= \frac{\prod\limits_{\substack{i=n-k-\ell                                                                                          \\i\neq n-\ell,n-k}}^{n} i }{\prod\limits_{i=1}^{k} i \cdot \prod\limits_{\substack{i=1 \\ i\neq \ell-k}}^\ell i }.
	\end{align*}
	Splitting the product in the numerator at the omitted number $n-k$ and suitable index shifts yield the expression
	$$
		d^\lambda
		=\frac{\prod\limits_{i=n-k+1}^{n} i \cdot \prod\limits_{\substack{i=n-k-\ell \\ i\neq n-\ell}}^{n-k-1} i }{\prod\limits_{i=1}^{k} i \cdot \prod\limits_{\substack{i=1 \\ i\neq \ell-k}}^\ell i }
		= \frac{\prod\limits_{i=1}^{k} (n-k+i) \cdot \prod\limits_{\substack{i=1\\ i\neq \ell-k}}^{\ell} (n-k-i) }{\prod\limits_{i=1}^{k} i \cdot \prod\limits_{\substack{i=1 \\ i\neq \ell-k}}^\ell i }.
	$$
	Analogous manipulations of the last expression in Equation~\eqref{eq:degalmhook} with the role of~$n-k$ played by~$n-\ell$ (split the product in the numerator at~$n-\ell$ instead of~$n-k$) yield:
	$$d^\lambda
		= \frac{\prod\limits_{\substack{i=n-\ell+1\\i\neq n-k}}^{n} i \cdot \prod\limits_{i=n-k-\ell}^{n-\ell-1} i }{\prod\limits_{\substack{i=1 \\ i\neq \ell-k}}^\ell i \cdot \prod\limits_{i=1}^{k} i  }
		= \frac{\prod\limits_{\substack{i=1\\i\neq \ell-k}}^{\ell} (n-\ell+i) \cdot \prod\limits_{i=1}^{k} (n-\ell-i) }{\prod\limits_{\substack{i=1 \\ i\neq \ell-k}}^\ell i \cdot \prod\limits_{i=1}^{k} i  }.
	$$
	This proves the two formulas for $d^\lambda$. The second statement follows immediately by inserting $k=0$.
\end{proof}

Now, we examine whether an (almost) hook partition is $p$-principal for $\sym n$, i.e.,  corresponds to an irreducible character of~$\sym n$ in the principal $p$-block, given a prime $p$. We make use of the next two lemmas, the first one of
which is easily proven using the $e$-abacus configuration associated to a partition and a positive integer $e$.\newline
An $e$-abacus is a tool to determine the $e$-core of a partition $\lambda$, see \cite[Chapter 2.7]{JK}. Here, one would use the abacus with $e$ runners whose initial bead configuration is given by the first column hook lengths of~$\lambda$. %

\begin{lemma} \label{lem:partitionecoreinsteps}
	Let $\lambda$ be a partition and $m,e\in\Z_{\geq 1}$. Let $\mu$ be another partition obtained from $\lambda$ by removing an $me$-hook. Then, the $e$-cores of~$\mu$ and $\lambda$ agree.
\end{lemma}

The next lemma is immediate if one considers the Young diagram of~$\lambda$ to determine its $p$-core and uses the previous lemma with $p$ playing the role of~$e$.

\begin{lemma} \label{lem:almosthookpartitionpprincipal}Let $n\geq 2$ and $p$ a prime. %
	The almost hook partition $(1^{n-k-\ell-1},k+1,\ell)$, where $0\leq k<  \ell< n-k\leq n$, is $p$-principal for $\sym n$ if $n-\ell$ or $n-k$ are divisible by~$p$.
\end{lemma}

In the next section we develop the number-theoretical framework used in the proof of Theorem A for symmetric, and later also for alternating as well as projective special linear or unitary, groups.
\subsection{Chopped \texorpdfstring{$p$-adic}{p-adic} expansion} \label{padicsetup}
In this section, we define a couple of quantities attached to a positive integer~$n\geq 3$ and two distinct primes $p$ and $q$ smaller or equal to~$n$.
We define and investigate a decomposition of the $p$-adic expansion of~$n$ with respect to the second prime $q\leq n$, and a variation thereof.%

The following expansion of~$n$ in terms of~$p$ is well-known. Our convention for the \emph{$p$-adic expansion} of~$n$ is as follows.

\begin{convention} \label{def:padicexpansion}
	Let $n\geq 2$ and $p$ a prime. \newline
	The \emph{$p$-adic expansion} of~$n$ is the expansion
	$n=s+\sum_{i=1}^t c_i p^{a_i}$ with non-negative integers~$t\geq 0$ %
	and $0\leq s<p$ and positive integral coefficients $1\leq c_i\leq  p-1$ for all $1\leq i\leq t$, and exponents $1\leq a_1< a_2< \ldots< a_t$. %
	Note that $s$ is just the residue of~$n$ modulo $p$ and that all coefficients $c_i$ are non-zero, and the exponents are in ascending order.
\end{convention}

From now on, fix the following notation.
\begin{notation}Let $n\geq 3$ and $p$ and $q$ distinct primes such that \fbox{$p,q\leq n$}.
	Fix the $p$-adic expansion $n=s+\sum_{i=1}^t c_i p^{a_i}$ of~$n$ according to the above convention, and note that we have $t\geq 1$ since~$p\leq n$.\newline
	Write $n=mq+r$ with integers $m\geq 1$ %
	and $0\leq r<q$. Note that $r$ is the residue of~$n$ modulo $q$.
\end{notation}

Depending on $r$, the residue of~$n$ modulo $q$, we can define a decomposition of the $p$-adic expansion of~$n$, and we then examine properties of the parts of~$n$ obtained in this way, often provided that \fbox{$r\geq s$}.

Recall that any positive integer $k$ can be uniquely decomposed as a product $k=k_p k_{p'}$ of positive integers $k_p$, $k_{p'}\geq 1$ where $k_p$ is a power of~$p$ and $k_{p'}$ is not divisible by~$p$. Then, $k_p$ and $k_{p'}$ are called the \emph{$p$-part} and \emph{$p'$-part} of~$k$, respectively.

\begin{lemma}\label{lem:npqquantities}
	Let $n\geq 3$. %
	Then, the quantities
	\renewcommand{\arraystretch}{1.5}
	\begin{center}
		\begin{tabular}{|ll|}
			\hline
			$t_0$ & $\coloneqq \min\{i\in\{1,\dots, t\}  \mid  r< c_i  p^{a_i}\}$, \\
			$b$   & $\coloneqq s+\sum_{i=1}^{t_0-1} c_i  p^{a_i}$,                 \\
			$T$   & $\coloneqq \sum_{i=t_0}^t c_i  p^{a_i}$                        \\
			\hline
		\end{tabular}
	\end{center}
	\renewcommand{\arraystretch}{1}
	are well-defined, and we have $n=b+T$ with $T\neq 0$.
	Moreover, the following hold: 	%
	\begin{enumerate}[\rm(i)]
		\item If $k\leq b$ is a positive integer, then we have $(T+k)_p=k_p$.
		\item If $k\leq r$  is a positive integer, then we have $(T-k)_p=k_p$.
		\item If \fbox{$r\geq s$ and $r\neq 0$}, then we have $b<2r$.
	\end{enumerate}
\end{lemma}

\begin{proof}
	For $1\leq k\leq t$, let $S_k\coloneqq s + \sum_{i=1}^{k} c_i p^{a_i}$ be the $k$th partial sum of the $p$-adic expansion of~$n$, and set $S_0\coloneqq s$. Note that for $ k\neq t$ the $(k-1)$th partial sum $S_{k-1}$ is always smaller than the following term $c_k p^{a_k}$ in the $p$-adic expansion. Setting $a_0\coloneqq 0$ we indeed have
		\begin{equation*} \tag{a} \label{eq:partialsumestimate}
			\begin{aligned}
	S_{k-1} &= s p^{a_0} + \sum_{i=1}^{k-1} c_i p^{a_i} 
		\leq \sum_{i=0}^{k-1} (p-1)p^{a_i}\\
		& \leq \sum_{i=0}^{a_{k-1}} (p-1)p^i %
		=p^{a_{k-1}+1} - 1
		<p^{a_{k-1}+1} \leq p^{a_{k}},
	\end{aligned}
	\end{equation*}
	using the geometric sum formula.
	In particular, we see that $n-c_t p^{a_t}=S_{t-1}< p^{a_t}\leq c_t p^{a_t}$, thus concluding $\frac{n}{2}< c_t p^{a_{t_0}}$.
	Using $r<q$ we get $r<\frac{r+mq}{2}=\frac{n}{2}< c_t p^{a_t}$ and %
	the number $t_0=\min\{i\in\{1,\dots, t\}  \mid  r< c_i p^{a_i}\}$ is well-defined, and so are $b=s+\sum_{i=1}^{t_0-1} c_i p^{a_i}$ and $T= \sum_{i=t_0}^t c_i p^{a_i}$. Then we have $n=b+T$ and $T\neq 0$ as~$T>r\geq 0$.

	Note that we have $b=S_{t_0-1}< p^{a_{t_0}} $ by Equation~\eqref{eq:partialsumestimate}.
	Further $T=\sum_{i=t_0}^t c_i p^{a_i}$ has $p$-part $T_p=p^{a_{t_0}}$ and $p'$-part $T_{p'}=c_{t_0}+py$ with $y\coloneqq \sum_{i=t_0+1}^t c_i p^{a_i-a_{t_0}-1}\geq 0$.

	For (i) and (ii), %
	let $k \geq 1$, and note that if $k<p^l$ for some $l>0$, we certainly have~$k_p<p^l$.

	(i) Let $k\leq b$. %
	As seen above, we have $b<p^{a_{t_0}}$, and so the $p$-part of~$k$ satisfies \linebreak $k_p<p^{a_{t_0}}=T_p$. This implies
	$k+T=k_p k_{p'} +T_p T_{p'}=k_p (k_{p'} + \frac{T_p}{k_p} T_{p'})$. Noticing that $k_{p'} + \frac{T_p}{k_p} T_{p'}$ is coprime to~$p$ as~$p$ divides~$\frac{T_p}{k_p}$, but not $ k_{p'}$, we obtain $(T+k)_p=k_p$.

	(ii) Now, let $ k\leq r$. We want to show that $(T- k)_p=k_p$.
	Note that we have $r<T$ by definition of~$t_0$, thus $T-k$ is positive, %
	as~$r<c_{t_0} p^{a_{t_0}}<p^{a_{t_0}+1}$ holds we have $k_p\leq p^{a_{t_0}}=T_p$. %

	If $k_p<p^{a_{t_0}}=T_p$, we can argue as in~(i) to see that $T-k=k_p (\frac{T_p}{k_p} T_{p'}-k_{p'})$ has $p$-part $(T-k)_p=k_p$.

	Now, assume that $k_p=T_p=p^{a_{t_0}}$. Then we have $T-k=k_p(T_{p'}-k_{p'})$.
	Observe that we have $p^{a_{t_0}} k_{p'}=k_p k_{p'}=k \leq r < c_{t_0} p^{a_{t_0}}$, hence $0<k_{p'} < c_{t_0} <p$.
	Using that $T_{p'}=c_{t_0}+py$ with $y$ as above, we obtain that $T_{p'}-k_{p'}=T_{p'}-c_{t_0}+c_{t_0}-k_{p'}=p y+c_{t_0}-k_{p'}$ is not divisible by~$p$. We conclude that $T-k=k_p(T_{p'}-k_{p'})$ has $p$-part $k_p$.

	(iii) Assume that $r\geq s$ and $r\neq 0$.
	We claim that we have $b<2r$ (this is often stronger than $b<p^{a_{t_0}}$).
	For $t_0=1$, we have $b=s\leq r<2r$ by assumption (in particular we use $r\neq 0$).
	For $t_0=2$, we have $c_1p^{a_1}\leq r$ (by definition of~$t_0$), and \linebreak thus $b=s+c_1p^{a_1}<2c_1p^{a_1}\leq 2r$.
	If $t_0\geq 3$, then, using $S_{t_0-2}<p^{a_{t_0-1}}$ (Equation~\eqref{eq:partialsumestimate}) and the definition of~$t_0$, %
	we see that
	$b=S_{t_0-1}=S_{t_0-2}+ c_{t_0-1}p^{a_{t_0-1}}<2c_{t_0-1} p^{a_{t_0-1}}<2r$.
\end{proof}

\begin{remark}
	Whenever we have $r\leq s$, the decomposition $n=b+T$ of the $p$-adic expansion of~$n$ is trivial in the sense that $b=s$, and $s$ was just the residue of~$n$ modulo $p$. %
	Moreover, Lemma~\ref{lem:npqquantities}~(iii) is false in general without the assumption $r>s$. If $r\leq s$, we have $b=s\geq  r$, as said before, but not necessarily $b<2r$.
	(The numbers $(n,p,q)=(11,7,2)$ provide a counterexample with $r<s$ and $b>2r$.) %
	This is why we usually restrict our attention to the case $r>s$.
	Finally, notice that, if $r=s$, the bound $b<2r$ fails if and only if $r=0$.
	That is why we exclude $r=0$ in~(iii).%
\end{remark}

Under an extra condition on $n$ and $q$, and setting $r'\coloneqq q+r$, we define a slightly different decomposition $n=b'+T'$ of~$n$ analogously to the one in the previous lemma. Given $r\geq s$ and an additional assumption, %
we get much better bounds on the quantity~$b'$, see~(iv)--(v), than the naive bound in~(iii).

\begin{lemma}\label{lem:npqquantities2}
	Let $n\geq 3$, and \fbox{$n\geq 2r'$}, and $t_0$, $b$ be as in Lemma~\ref{lem:npqquantities}.
	Then, the quantities
	\renewcommand{\arraystretch}{1.5}
	\begin{center}
		\begin{tabular}{|ll|}
			\hline
			$t_0'$ & $\coloneqq \min\{i\in\{1,\dots, t\}  \mid  r'< c_i p^{a_i}\}$, \\
			$b'$   & $\coloneqq s+\sum_{i=1}^{t_0'-1} c_i  p^{a_i}$,                \\
			$T'$   & $\coloneqq \sum_{i=t_0'}^t c_i  p^{a_i}$                       \\
			\hline
		\end{tabular}
	\end{center}
	\renewcommand{\arraystretch}{1}
	are well-defined, and we have $n=b'+T'$ with $T'\neq 0$.
	Moreover, the following hold:
	\begin{enumerate}[\rm(i)]
		\item If $k\leq b'$ is a positive integer, then we have $(T'+k)_p=k_p$.
		\item If $k\leq r'$ is a positive integer, then we have $(T'- k)_p=k_p$.
		\item If \fbox{$r\geq s$}, then we have $b'<2r'$.
	\end{enumerate}
	If we assume \fbox{$b\geq q$} and \fbox{$r\geq s$}, %
	then we have $ t_0'\in\{t_0, t_0+1\}$
	and $b'\neq r'$. We have $t_0=t_0'$ if and only if $b'<r'$,
	and the following hold:
	\begin{enumerate}
		\item[\rm(iv)]  If $t_0'=t_0$, then we have $q\leq b'=b<r'$. In this case, $T'$ is not divisible by~$q$.
		\item[\rm(v)]  If $t_0'=t_0+1$, then we have $r'<b'<q+3r$.
	\end{enumerate}

\end{lemma}

\begin{proof}
	Combining the assumption $2r'\leq n$ with the inequality $\frac{n}{2}<c_t p^{a_t}$ (see proof of the previous lemma),
	we see that
	$r'\leq \frac{n}{2}<c_t p^{a_t}$. %
	So, the number
	$$t_0'= \min\{i\in\{1,\dots, t\}  \mid  r'< c_i p^{a_i}\}$$ is well-defined,
	as well as $b'=s+\sum_{i=1}^{t_0'-1} c_i p^{a_i}$ and $T'= \sum_{i=t_0'}^t c_i p^{a_i}$. Then we have $n=b'+T'$ with $T'>r'\geq q$.

	(i)--(iii) The proofs of~(i)--(iii) are analogous to the proofs of Lemma~\ref{lem:npqquantities}~(i)--(iii) substituting $t_0,T,b,r$ by~$t_0',T',b',r'$, respectively. (However, do not substitute $r$ by~$r'$ whenever using the assumption $r\geq s$). Note that we use $r'=q+r\neq 0$ only in~(iii) and only if $t_0'=1$. Indeed, if $t_0'=1$, we have $b'=s\leq r<q+r=r'<2r'$.

	For the rest of the proof, we assume that {$b\geq q$} and $r\geq s$.
	We first show $ t_0'\in\{t_0, t_0+1\}$.
	By definition of~$t_0$, we have $t_0\leq t_0'$ (as~$r<q+r=r'<c_{t_0'}p^{a_{t_0'}}$), thus $b'\geq b$.
	Further, we have $r'=q+r\leq b+r< p^{a_{t_0}}+r<(1+c_{t_0})p^{a_{t_0}}\leq p^{a_{t_0+1}}\leq c_{t_0+1} p^{a_{t_0+1}}$
	using $b\geq q$, the estimate $b<p^{a_{t_0}}$ (noted in the proof of Lemma~\ref{lem:npqquantities}) and the definition of~$t_0$.
	We conclude that $t_0'\leq t_0+1$ by definition of~$t_0'$. %

	For~(iv)--(v), note that we automatically have $r\neq 0$ given $b\geq q\geq r\geq s$, otherwise we obtain $b=s=0$ using $r\geq s$, contradicting $b\geq q$. In particular, we have $b<2r$ by Lemma~\ref{lem:npqquantities}~(iii).

	(iv) If $t_0'=t_0$, then we have $b'=b<2r<q+r=r'$.
	In particular, we get $0<r'-b'<q$ using $r'<2q$ and $b'=b\geq q$, and we deduce that  $t'=n-r'+(r'-b')=(m-1)q+(r'-b')$ is not divisible by~$q$.

	(v) if $t_0'=t_0+1$, then we have $b'=b+c_{t_0}p^{a_{t_0}}$, hence $r'=q+r<b+c_{t_0}p^{a_{t_0}}=b'$ using $q\leq b$ and then definition of~$t_0$. Moreover, again by Lemma~\ref{lem:npqquantities}~(iii) and by definition of~$t_0'$, we have
	$b'=b+c_{t_0'-1}p^{a_{t_0'-1}}<2r+r'=q+3r.$

	In both~(iv) and~(v), we have $b'\neq r'$, and we see that $t_0'=t_0$ if and only if $b'<r'$. This concludes the proof.
\end{proof}

Unfortunately, the property that $T'$ is coprime to~$q$ is lost in~(v) as we have \linebreak $T'=(m-1)q-(b-r')$ but the positive integer $b'-r'$ can be divisible by~$q$. The best bound we have is $0<b'-r'=b'-e_q-r<2r<2q$.

\subsection{Proof of \texorpdfstring{Theorem~\ref{thm:A}}{Theorem A} for symmetric groups}\label{sec:sym}
Now, we have all tools at hand to prove Theorem~\ref{thm:A} for symmetric groups.

\begin{proposition} \label{pqSn}
	Let $n\geq 3$ and $p$ and $q$ be distinct primes such that $p,q\leq n$. Then, we have \[\irrprime p {\blprinc{p}{\sym n}} \cap \irrprime q {\blprinc{q}{\sym n}}\neq \{ \mathbbm{1}_{\sym n} \}.\]
\end{proposition}

\begin{proof}
	Write $	n=w p+s =mq +r$	with positive integers $w,m\geq 1$ such that $0\leq s<p$ and $0\leq r<q$.
	We have to find a partition $\lambda \neq (n)$ of~$n$ which is $\{p,q\}$-principal (for $\sym n$) %
	and has degree $d^\lambda$ coprime to both $p$ and $q$.
	By possibly interchanging $p$ and $q$, we can assume $r\geq s$.
	Further, we can assume $r\geq 2$ as
	for $r\in \{0,1\}$ (in particular $s\in\{0,1\}$), %
	the partition $(1^n)$ is $p$- and $q$-principal by what we said in Section~\ref{sec:background}.
	Thus, in this situation, the sign character fulfills all the conditions (having degree $1$).

	Henceforth, we assume that \fbox{$r\geq \max(2,s)$}. In particular, we have $q\geq 3$ (as~$r\geq 2$).
	Consider the $p$-adic expansion of~$n$. Let
	$n=s+\sum_{i=1}^t c_i p^{a_i}$,
	where $t\geq 1$ is a positive integer, and we have $1\leq a_1< a_2< \ldots< a_t$, and $1\leq c_i\leq  p-1$ for all $1\leq i\leq t$.
	In Lemma~\ref{lem:npqquantities}, we defined the following three quantities:
	\renewcommand{\arraystretch}{1.5}
	\begin{center}
		\begin{tabular}{|ll|}
			\hline
			$t_0$ & $\coloneqq \min\{i\in\{1,\dots, t\}  \mid  r< c_i p^{a_i}\}$, \\
			$b$   & $\coloneqq s+\sum\limits_{i=1}^{t_0-1} c_i p^{a_i}$,          \\
			$T$   & $\coloneqq \sum\limits_{i=t_0}^t c_i p^{a_i}$.                \\
			\hline
		\end{tabular}
	\end{center}
	\renewcommand{\arraystretch}{1}
	This yields a decomposition $n=b+T$ with $r<T$ (following from the definition of~$t_0$) and
	$b\equiv s \bmod p$.
	Now, we proceed with a case distinction as follows:%
	\renewcommand{\arraystretch}{1.5}
	\begin{center}
		\begin{tabular}{|p{2cm}l|}
			\hline
			Case I:    & $r=b$,                          \\
			Case IIa:  & $r>b$,                          \\
			Case IIb:  & $r<b<q$,                        \\
			Case IIbb: & $r<q\leq b$ and $q\nmid (m-1)$, \\
			Case III:  & $r<q\leq b$ and $q\mid (m-1)$.  \\
			\hline
		\end{tabular}
	\end{center}
	\renewcommand{\arraystretch}{1}
	Let $R\coloneqq m e_q$ so that $n=R+r$.
	Note that in Case I, we have $T=n-b=n-r=R$. In Case IIa, $T=n-b$ is strictly bigger than $R=n-r$ and in the remaining cases $T$ is strictly smaller than $R$.
	In all cases, we have $b<2r$ %
	by Lemma~\ref{lem:npqquantities}~(iii) using $r\geq s$ and~$r\neq 0$.

	For the moment, turn the attention to Case III. Here, we assume $r<q\leq b$, and that $q$ divides~$m-1$. Let $r'\coloneqq q+r$.
	Note first that $mq+r=R+r=n=b+T>b+r\geq q+r'$ using $r<T$ and the assumption $b\geq q$. Hence, we have $m>1$. Since $q$ divides~$m-1$ and~$q\geq 3$, we even have $m\geq 4$. %
	Thus, the number
	$${t_0'\coloneqq \min\{i\in\{1,\dots, t\}  \mid  r'< c_i p^{a_i}\}}$$
	is well-defined according to Lemma~\ref{lem:npqquantities2}. %
	The same lemma gives a modified decomposition $n=b'+T'$ where
	\renewcommand{\arraystretch}{1.5}
	\begin{center}
		\begin{tabular}{|ll|}
			\hline
			$b'$ & $\coloneqq s+\sum\limits_{i=1}^{t_0'-1} c_i p^{a_i}$, \\
			$T'$ & $\coloneqq \sum\limits_{i=t_0'}^t c_i p^{a_i}$.       \\

			\hline
		\end{tabular}
	\end{center}
	\renewcommand{\arraystretch}{1}
	In particular, we have $r'<T'$ and $b'\equiv s \equiv b \bmod p$.
	Note that, as shown in Lemma~\ref{lem:nepeqquantities2}, we have $b'\neq r'$, so only two more cases arise:
	\renewcommand{\arraystretch}{1.5}
	\begin{center}
		\begin{tabular}{|p{2cm}l|}
			\hline
			Case IIIa: & $r<q\leq b$ and $q\mid (m-1)$ and $r'>b'$, \\
			Case IIIb: & $r<q\leq b$ and $q\mid (m-1)$ and $r'<b'$. \\
			\hline
		\end{tabular}
	\end{center}
	\renewcommand{\arraystretch}{1}
	In the first case, we have $b'=b$ (equivalently $t_0=t_0'$), otherwise $b'<q+3r$ (equivalently $t_0'=t_0+1$) holds, using Lemma~\ref{lem:npqquantities2}~(iv)--(v) and the fact that $q\leq b$ in Case~III.
	Setting $R'\coloneqq (m-1)q$, we have $n=R'+r'$, and so by definition $T'>R'$ in Case~IIIa and $T'<R'$ in Case~IIIb.
	Moreover, by Lemma~\ref{lem:npqquantities2}~(iii), we always have $b'<2r'$.

	Now, we give a partition in each case, see Table \ref{tab:partitionsforSn}, and show that it has the desired properties.
	Remark that $R-b=n-r-b=T-r>0$ (as~$T>r$), thus we have $R>b$. Similarly, we have $R'>b'$. These facts (together with the definition of each case) imply that the partitions are well-defined in each case, and of size $n$.
	\begin{table}[h]
		\begin{center}
			\begin{tabular}{lcl}
				\toprule
				\multirow{2}{*}{Case} & Main condition         & \multirow{2}{*}{Partition}            \\
				                      & in this case           &                                       \\
				\midrule
				I                     & $r=b$                  & $(1^{R},r)$                           \\
				IIa                   & $r>b$                  & $(1^{R-b-1}, b+1,r)$                  \\
				IIb                   & \multirow{2}{*}{$r<b$} & \multirow{2}{*}{$(1^{R-b-1}, r+1,b)$} \\
				IIbb                  &                        &                                       \\
				IIIa                  & $r'>b'$                & $(1^{R'-b'-1},b'+1,r')$               \\
				IIIb                  & $r'<b'$                & $(1^{R'-b'-1}, r'+1, b')$             \\
				\bottomrule
			\end{tabular}
		\end{center}%
		\caption{Partitions for non-trivial characters in \linebreak $\irrprime p {B_p(\sym n)} \cap \irrprime q {B_q(\sym n)}$, %
			assuming $r\geq \max(2, s)$.}
		\label{tab:partitionsforSn}
	\end{table}
	
	Let $\lambda$ be one of the partitions displayed in Table \ref{tab:partitionsforSn}.

	\paragraph*{\textbf{Block analysis}} We first explain why the given partition %
	is $\{p,q\}$-principal.
	By definition of~$T$ and $R$, the prime $p$ divides~$n-b=T$ and the prime $q$ divides~$n-r=R$. Similarly, we have $p$ dividing~$n-b'$ and $q$ dividing~$n-r'$. Thus, Lemma~\ref{lem:almosthookpartitionpprincipal} yields that $\lambda$ is $p$- as well as $q$-principal in each case. In Case I, we additionally use the condition $r=b$.

	\paragraph*{\textbf{Degree analysis}}
	Now, we analyze case-by-case the $p$- and $q$-part of the degree $d^\lambda$ of the partition $\lambda$. In each case, the degree is immediately obtained by the formulas in Lemma~\ref{lem:degalmosthookpartition} as $\lambda$ is an almost hook partition. The results are collected in Table \ref{tab:degSn}.
	\begin{table}[t]
		\begin{center}
			\begin{tabular}{lccc}
				\toprule
				Case    & \multicolumn{3}{c}{Degree $d^\lambda$}                                                                                             \\
				\cmidrule(lr){1-1} \cmidrule(lr){2-4}
				\addlinespace[0.5 em]
				I%
				        & $\frac{\prod\limits_{k=1}^{b-1} (T+k) }{\prod\limits_{k=1}^{b-1} k} $           & $=$ & $\frac{ \prod\limits_{k=1}^{r-1} (R+k) }{ \prod\limits_{k=1}^{r-1} k}$ \\
				\addlinespace[0.5 em]
				IIa     & $\frac{\prod\limits_{k=1}^{b} (T+k) \cdot \prod\limits_{\substack{k=1                                                                     \\ k\neq r-b}}^{r} (T-k)}{\prod\limits_{k=1}^{b} k \cdot \prod\limits_{\substack{k=1 \\ k\neq r-b}}^{r} k}$ &
				$=$     & $ \frac{\prod\limits_{\substack{k=1                                                                                                \\k\neq r-b}}^{r} (R+k) \cdot \prod\limits_{k=1}^{b} (R-k) }{ \prod\limits_{\substack{k=1\\k\neq r-b}}^{r} k \cdot \prod\limits_{k=1}^b k}$\\
				\addlinespace[0.5 em]
				IIb--bb &
				$\frac{\prod\limits_{\substack{k=1                                                                                                           \\ k\neq b-r}}^{b} (T+k) \cdot \prod\limits_{k=1}^{r} (T-k)  }{\prod\limits_{\substack{k=1 \\ k\neq b-r}}^{b} k \cdot \prod\limits_{k=1}^{r} k}$ &
				$=$     & $\frac{\prod\limits_{k=1}^{r} (R+k) \cdot \prod\limits_{\substack{k=1                                                                     \\ k \neq b-r}}^{b} (R-k) }{ \prod\limits_{k=1}^{r} k \cdot \prod\limits_{\substack{k=1 \\ k\neq b-r}}^b k}$ \\
				\addlinespace[0.5 em]
				IIIa    & $\frac{\prod\limits_{k=1}^{b'} (T'+k) \cdot\prod\limits_{\substack{k=1                                                                    \\k\neq r'-b'}}^{r'} (T'-k) }{\prod\limits_{k=1}^{b'} k \cdot \prod\limits_{\substack{k=1 \\k\neq r'-b'}}^{r'} k}$ &
				$=$     & $\frac{\prod\limits_{\substack{k=1                                                                                                 \\k\neq r'-b'}}^{r'} (R'+k) \cdot \prod\limits_{k=1}^{b'} (R'-k) }{\prod\limits_{\substack{k=1 \\k\neq r'-b'}}^{r'} k \cdot \prod\limits_{k=1}^{b'} k }
				$                                                                                                                                            \\
				\addlinespace[0.5 em]
				IIIb    & $\frac{\prod\limits_{\substack{k=1                                                                                                 \\k\neq b'-r'}}^{b'} (T'+k) \cdot \prod\limits_{k=1}^{r'} (T'-k) }{\prod\limits_{\substack{k=1 \\k\neq b'-r'}}^{b'} k \cdot \prod\limits_{k=1}^{r'} k} $ &
				$=$     & $\frac{ \prod\limits_{k=1}^{r'} (R'+k) \cdot \prod\limits_{\substack{k=1                                                                  \\k\neq b'-r'}}^{b'} (R'-k) }{\prod\limits_{k=1}^{r'} k \cdot \prod\limits_{\substack{k=1 \\k\neq b'-r'}}^{b'} k}$ \\
				\addlinespace[0.5 em]
				\bottomrule
			\end{tabular}
		\end{center}
		\caption{Degree $d^\lambda$ of the partition $\lambda$.}
		\label{tab:degSn}
	\end{table}
	To show that $d^\lambda$ is coprime to~$p$ we use the observations in Lemma~\ref{lem:npqquantities} (resp. Lemma~\ref{lem:npqquantities2}) about the $p$-part of~$T$ (resp. $T'$) shifted by an integer. For the analysis of the $q$-part of~$d^\lambda$ one has to take into account the individual conditions on $q$ in each particular case.

	\subparagraph*{\textit{$p$-part of the degree $d^\lambda$}}

	By Lemma~\ref{lem:npqquantities}~(i)--(ii), we have  $(T+k)_p=k_p$, for all $1\leq k\leq b$, and $(T-k)_p=k_p$ for all $1\leq k\leq r$. This implies that the degree $d^\lambda$ is not divisible by~$p$ in Cases~I--IIbb because the $p$-part of the numerator cancels with the $p$-part of the denominator.
	In Cases~IIIa--b, we use analogous observations in Lemma~\ref{lem:npqquantities2}~(i)--(ii) to conclude that the degree $d^\lambda$ is not divisible by~$p$.

	\subparagraph*{\textit{$q$-part of the degree $d^\lambda$}}
	To determine the $q$-part of~$d^\lambda$, look at the second degree expression in Table \ref{tab:degSn}.
	Note that $R=mq$ and $R'=(m-1)q$ are both divisible by~$q$.

	In Case~I--IIb, we assume $b<q$, and we always have $r<q$. Hence, the $q$-parts of~$R+k$ and $k$ agree, i.e., we have $(R+k)_q=k_q$, for all $1\leq k\leq r$, using that $R$ is a multiple of~$q$. Similarly, we get that $(R-k)_q=k_q$, for all $1\leq k\leq b$. We deduce that the degree $d^\lambda$ is coprime to~$q$.

	Note that in Case IIbb, we assume $b>q$, but we have $b<2r<2q$ by Lemma~\ref{lem:npqquantities}~(iii). Moreover, we require that $q$ does not divide~$m-1$ in this case.
	Thus, we have \linebreak $(R-k)_q=(mq-k)_q=k_q$, for all $1\leq k\leq b$. With the same reasoning as in the previous cases, we get $(R+k)_q=k_q$, for all $1\leq k\leq r$. In total, the degree $d^\lambda$ is coprime to~$q$.

	In the remaining cases, $m-1=\frac{R'}{q}$ is assumed to be a multiple of~$q$. Note that we have $r'=q+r<2q$. %
	Thus, we deduce that $(R'+ k)_q=((m-1)q+k)_q=k_q$, for all $1\leq k\leq r'$, since~$q$ does not divide~$m-2$.

	In Case IIIa, we have $b'<r'$. The same argument as before yields that \linebreak $(R'- k)_q=((m-1)q-k)_q=k_q$, for all $1\leq k\leq b'$. %
	So, the degree $d^\lambda$ is again coprime to~$q$.

	In Case IIIb, Lemma~\ref{lem:npqquantities2}~(v) is applicable as~$b\geq q$ and $b'>r'$, and we get $b'{<}q+3r<q^2$. For the very last inequality we make use of the fact that $q\geq 3$. In particular, the $q$-part of any $1\leq k\leq b'$ is at most $q$, but the $q$-part of~$R'$ is at least $q^2$.
	So, we conclude that $(R-k)_q=((m-1)q - k)_q=k_q$ holds. As in the previous cases, we deduce that the degree is coprime to~$q$.

	All in all, we see that the partition $\lambda$ is $p$- as well as $q$-principal for $\sym n$ and has degree not divisible by~$p$ nor $q$. %
	Thus, the character $\chi^\lambda$ lies in the intersection \linebreak $\irrprime p {\blprinc{p}{\sym n}} \cap \irrprime q {\blprinc{q}{\sym n}}$ and is indeed non-trivial, as~$\lambda\neq (n)$, and we are done.
\end{proof}

\begin{remark}%
	In Proposition~\ref{pqSn}, we reprove \cite[Proposition~3.11]{NRS} (corrected proof in \cite{NRSAddendu}, see comment below) where the case $2\in\{p,q\}$ is addressed, and generalize this result to arbitrary primes $p$,~$q$. The methods explained in Section \ref{padicsetup}, together with the assumption $r\geq s$, give the proof a universal framework and allow us to generalize the ideas of Navarro--Rizo--Schaeffer Fry. They use the very same methods in many cases of their proof for $p=2$ (where Lemma~\ref{lem:npqquantities} and Lemma~\ref{lem:npqquantities2} are straightforward) but without assuming~$r\geq s$.
	Note that our proof specialized to~$2\in \{p,q\}$ differs from their proof in some cases (e.g., whenever $r=0=s-1$, where we would use the $q$-adic expansion of~$n$ as a starting point, but they stick to the $2$-adic one), but we do not go into more details here. %
	Previous gaps and a missing case in their proof were closed in an addendum \cite{NRSAddendu}.
\end{remark}

\subsection{Proof of Theorem \texorpdfstring{\ref{thm:A}}{A} for alternating groups}\label{sec:alt}
In this section, we prove the analog of Proposition~\ref{pqSn} for alternating groups after recalling some well-known facts about their character theory, see again \cite{JK}. %

Let $n\geq 3$. The irreducible characters of the alternating group $\alt n$ are labeled by partitions of~$n$ but, in contrast to the ones of the symmetric group $\sym n$, not uniquely. \newline
As the alternating group $\alt n$ is a normal subgroup of~$\sym n$, Clifford theory in its simplest instance (here $|\sym n\,:\,\alt n|=2$) applies. It
allows to deduce the following partial information about the set $\irr {\alt n}$ from the (partial) knowledge (of properties) of the set~$\irr{ \sym n}$:

Let $\lambda\vdash n$ be a partition of~$n$.
The \emph{conjugate} of~$\lambda$ (obtained by mirroring its Young diagram at the diagonal) is denoted by~$\lambda^t$.
Note that we have $\chi^{\lambda^t}=\sgn \cdot \chi^\lambda$ and that the sign character $\sgn$ restricts to the trivial character $\mathbbm{1}_{\alt n}\in\irr{\alt n}$. \newline %
If $\lambda$ is \emph{not~self-conjugate}, i.e., $\lambda\neq \lambda^t$,
the irreducible character $\chi^\lambda\in \irr{\sym n}$ of the symmetric group $\sym n$ restricts irreducibly to its normal subgroup $\alt n$.
Otherwise, the restricted character~$\res{\chi^\lambda}{\alt n}$ %
is the sum of two irreducible characters of $\alt n$, say $\psi^\lambda_{+}$ and $\psi^{\lambda}_{-}$, of degree~$\frac{d^\lambda}{2}$ \cite[2.5.7 Theorem]{JK}.

\begin{definition} \label{def:pprincAn}
	Let $n\geq 3$ and $p$ a prime. %
	We say that a partition $\lambda\vdash n$ is \emph{$p$-principal for $\alt n$}
	if one of the irreducible constituents of the character $\res{\chi^\lambda}{\alt n}$ lies in the principal $p$-block of~$\alt n$.

	If $q$ is a prime distinct from $p$, we say that $\lambda\vdash n$ is \emph{$\{p,q\}$-principal for $\alt n$} if  $\lambda$ is $p$-principal as well as $q$-principal for $\alt n$.
\end{definition}
We first have to explain that the notion of~$p$-principal partitions for $\alt n$ is well-defined.

\begin{remark} Let $\lambda\vdash n$.
	Note that the irreducible constituents of the character $\res{\chi^\lambda}{\alt n}$ all lie in the same block of~$\alt n$ unless $\lambda$ is a self-conjugate $p$-core, by \cite[6.1.46 Theorem]{JK}. %
	In this case, the irreducible constituents $\psi^\lambda_{\pm}$ of~$\res{\chi^\lambda}{\alt n}$ form their own block %
	\cite[6.1.18]{JK} as they are of \emph{defect zero} (i.e., $(d^\lambda)_p={\abs {\alt n}}_p$) using that $\lambda$ is a $p$-core.
	However, the characters~$\psi^\lambda_{\pm}$ do not lie in the principal $p$-block of~$\alt n$,
	otherwise, we would have $\lambda\in\{(n), (1^n)\}$, but these partitions are not self-conjugate. So, there is no ambiguity in Definition \ref{def:pprincAn}.

\end{remark}
The property of~$p$-principality for $\sym n$ implies $p$-principality for $\alt n$, but the converse is not true. This remark is crucial for our strategy to prove Proposition~\ref{pqAn}.

\begin{remark}	\label{rem:pprincSnAn}
	If $\lambda$ is $p$-principal for $\sym n$, it is also $p$-principal for $\alt n$. This follows from the fact that the principal $p$-block of~$\sym n$ \emph{covers} \textsl{only} the principal $p$-block of~$\alt n$ and no other $p$-block (for covering of blocks see \cite[(9.2) Theorem]{NavarroBlocks}). %
	In other words, the irreducible constituents of~$p$-principal irreducible characters of~$\sym n$ lie in the principal $p$-block of~$\alt n$.
	However, the converse is not true. If $\lambda$ is $p$-principal for $\alt n$ it is not necessarily $p$-principal for $\sym n$, but either $\lambda$ or the conjugate partition $\lambda^t$ is.
\end{remark}

As a preparation for examining the set $\irrprime p {\blprinc{p}{\alt n}} \cap \irrprime q {\blprinc{q}{\alt n}}$, we answer the question whether the $p$-part of the degree of an almost hook partition \linebreak $\lambda=(1^{n-k-\ell-1},k+1,\ell)\vdash n$, with $k\in\{0,1\}$, is trivial or $2$,
given a prime $p\leq n$ and specific choices of~$n$.

\begin{lemma} \label{lem:degalmosthookpprimeor2}
	Let $n\geq 2$ and $p\leq n$ a prime. Let $n=cp^a+\eps$ with $\eps\in\{0,1\}$ and $c\geq 1$ coprime to~$p$, and let $1\leq c_1<p$ be the residue of~$c$ modulo $p$.
	\begin{enumerate}[\rm(i)]
		\item  %
		      Let $\lambda=(1^{n-\ell}, \ell)\vdash n$ with $1\leq \ell\leq n-1$. %
		      \begin{enumerate}[\rm(a)]
			      \item If $\eps=0$ and $\ell\leq c_1 p^a$, then the degree $d^\lambda$ is coprime to~$p$.

			      \item If $\eps=1$ and $c=c_1$, %
			            then the degree $d^\lambda$ is coprime to~$p$ if and only if $\ell=1$, or $\ell\geq 2$ and $(\ell-1)_p=p^a$.

			      \item  If $n=2^a+1$, %

			            then we have $(d^\lambda)_2=2$ if and only if $\ell=2^{a-1}+1$.

		      \end{enumerate}
		\item %
		      Let $\lambda=(1^{n-\ell-2},2,\ell)\vdash n$ with $2\leq \ell\leq n-2$. %
		      \begin{enumerate}[\rm(a)]
			      \item If $\eps=0$ and $\ell\leq c_1 p^a$ and $\ell_p=p^a$, %
			            then the degree $d^\lambda$ is coprime to~$p$.

			      \item If $\eps=1$ and $\ell<p^a$, then the degree $d^\lambda$ is coprime to~$p$.

			      \item If $n=2^{a}$ with $a\geq 2$, %
			            then we have $(d^\lambda)_2=2$ if and only if and $\ell=2^{a-1}$.
		      \end{enumerate}
	\end{enumerate}
\end{lemma}

\begin{proof} 
	Let $\lambda=(1^{n-k-\ell-1},k+1,\ell)\vdash n$ with $0\leq k< \ell< n-k\leq n$ and $k\in\{0,1\}$.
	We use Lemma~\ref{lem:degalmosthookpartition} to write out the $p$-part of~$d^\lambda$ as a fraction, and then compare $p$-parts of factors in the denominator with those of factors in the numerator.
	Let $c=c_1+jp$ for a non-negative integer~$j\geq 0$.

	\noindent (i) In this situation, we have $k=0$, and
	$d^\lambda=\frac{\prod_{i=1}^{\ell-1} (n-i)}{\prod_{i=1}^{\ell-1}i}$ by Lemma~\ref{lem:degalmosthookpartition}.

	(a) Note that if $\eps=0$ and $\ell\leq c_1 p^a$, then the $p$-part of~$n-i=cp^a=c_1p^a + j p^{a+1} - i$ is the $p$-part of~$i$, for all $i<\ell$. Thus, the degree is coprime to~$p$.

	In the other cases, we argue similarly.

	(b) Now, assume that $\eps=1$. If $\ell=1$, we have $\lambda=(1^n)$ and $d^\lambda=1$. From now on let $\ell\geq 2$. Then, we have $d^\lambda=\frac{(n-1)\cdot \prod_{i=1}^{\ell-2} (n-1-i)}{(\ell-1) \cdot \prod_{i=1}^{\ell-2} i}$ using the previous formula.

	If $c=c_1$, we have $(n-1)_p=(c_1p^a)_p=p^a$.
	With similar arguments as before, we have $(n-1-i)_p=(c_1p^a-i)_p=i_p$ for all $i$ as~$i<\ell\leq n-1=c_1p^a$. %
	In particular, the degree $d^\lambda$ is coprime to~$p$ if and only if $(\ell-1)_p=(n-1)_p=p^a$.

	(c) Now, let $n=2^a+1$ (i.e., $\eps=1$, $p=2$ and $c=c_1=1$) with $a\geq 2$.
	Then, we have $d^\lambda=\frac{2^a\cdot \prod_{i=1}^{\ell-2} (2^a-i)}{(\ell-1) \cdot \prod_{i=1}^{\ell-2} i}$.
	Note that the $2$-part of~$2^a-i$ is $(2^a-i)_2=i_2$ for any $1\leq i\leq \ell-2<2^a$ since~$i_2\leq 2^{a-1}$. In particular, the degree $d^\lambda$ has $2$-part $(d^\lambda)_2=2$ if and only if $(\ell-1)_2=2^{a-1}$, i.e., if and only if $\ell-1=2^{a-1}$ (using that $\ell\leq n-1=2^a)$.

	\noindent (ii) Here, we have $k=1$ and so by Lemma~\ref{lem:degalmosthookpartition}:
	$d^\lambda= (n-\ell-1) \cdot{\prod_{\substack{i=1\phantom{\ell-}\\i\neq \ell-1}}^{\ell} (n-\ell+i)  }/{\prod_{\substack{i=1\phantom{\ell-} \\ i\neq \ell-1}}^\ell i }.$

	(a) If $\eps=0$, we have $n=cp^a$. %
	Let $\ell=dp^a$ with $1\leq d\leq c_1$.
	Note that $n-\ell=(c-d)p^a$ has $p$-part at least $p^a$, and
	for any positive integer $i\leq \ell$ we have $i_p\leq p^{a}$ as~$\ell<p^{a+1}$.

	If $d=c_1$, then we have $(n-\ell)_p=((c-c_1)p^a)_p=(jp^{a+1})_p\geq p^{a+1}$. This immediately implies $(n-\ell+i)_p=i_p$ for all $i\leq \ell$ (using $i_p\leq p^a$). Hence, the degree $d^\lambda$ is coprime to~$p$.

	Now, let $d<c_1$ and $i\leq \ell$. If $i_p<p^a$, then $(n-\ell+i)_p=i_p$ is trivial.
	If $i_p=p^a$, we have $i_{p'}\leq d$ since~$i\leq \ell=dp^a$. Note that $0\leq d-i_{p'}\leq d<c_1<p$, thus $c_1-(d-i_{p'})$ is not divisible by~$p$. Hence, $$n-\ell+i=p^a(c-d+i_{p'})=p^a(c-c_1+c_1-(d-i_{p'}))=p^a (jp+(c_1-(d-i_{p'}))$$ has $p$-part $p^a=i_p$. Again, we conclude that the degree $d^\lambda$ is coprime to~$p$.

	(b) If $\eps=1$, we have $n-1=c p^a$, and $d^\lambda=n \cdot { \prod_{\substack{i=1\phantom{\ell-}\\ i\neq \ell-1}}^{\ell} (n-1-i) }/{ \prod_{\substack{i=1 \phantom{\ell-} \\ i\neq \ell-1}}^\ell i } $ by Lemma~\ref{lem:degalmosthookpartition}. %
	Assuming $\ell<p^a$, we always have $(n-1-i)_p=(cp^a-i)_p=i_p$ for all positive $i\leq \ell$, and the result follows.

	(c) Now let $n=2^{a}$ (i.e., $\eps=0$, $p=2$, $c=1$) with $a\geq 2$. Then, we have \linebreak
	$d^\lambda=(n-\ell-1) \frac{n \cdot (n-(\ell-1)) \cdot \prod_{i=2}^{\ell-2} (n-i)  }{\ell \cdot \prod_{i-2}^{\ell-2} i }$, i.e., $d^\lambda=\frac{(2^a-\ell-1)\cdot 2^a \cdot (2^a-\ell+1)}{\ell} \cdot \frac{\prod_{i=2}^{\ell-2} (2^a-i)  }{\prod_{i-2}^{\ell-2} i }.$

	As in (i)(c), we have $(2^a-i)_2=i_2$ for any $2\leq i\leq \ell-2<n=2^a$ since~$i_2\leq 2^{a-1}$. Thus, the degree $d^\lambda$ has $2$-part $(d^\lambda)_2=2$ if and only if $\ell_2=2^{a-1}$, i.e, $\ell=2^{a-1}$.
\end{proof}

Notice that it might be challenging to find a non-trivial irreducible character of~$\alt n$ lying in $\irrprime p {\blprinc{p}{\alt n}} \cap \irrprime q {\blprinc{q}{\alt n}}$ if the corresponding set on the level of~$\sym n$ is sufficiently small. The following example illustrates this observation.

\begin{example}[An instance of small intersection] \label{ex:trivialintersection9}
	Let $n=9$ and $\{p,q\}=\{2,3\}$. We determine $\irrprime{2}{\blprinc{2}{\sym 9}}\cap \irrprime{3}{\blprinc{3}{\sym 9}}$ and $\irrprime{2}{\blprinc{2}{\alt 9}}\cap \irrprime{3}{\blprinc{3}{\alt9}}$.

	Using the hook formula, one can show that the partitions of~$9=3^2$ having degree coprime to~$3$ are exactly the hook partitions of size $9$. %
	However, by Lemma~\ref{lem:degalmosthookpprimeor2}~(i)(b), the only non-trivial hook partition of~$9=2^3+1$ of odd degree is $(1^9)$ (which is $\{2,3\}$-principal for $\sym 9$ by Lemma~\ref{lem:almosthookpartitionpprincipal}).
	We conclude that $$\irrprime{2}{\blprinc{2}{\sym 9}}\cap \irrprime{3}{\blprinc{3}{\sym 9}}=\irrprime{\{2,3\}}{\sym 9}=\{\mathbbm{1}_{\sym 9}, \sgn\}.$$
	By Lemma~\ref{lem:degalmosthookpprimeor2}~(i)(c),
	the partition $\lambda=(1^4,5)$ is the only hook partition of size $9$ whose degree has $2$-part $(d^\lambda)_2=2$. Moreover, this partition is $\{2,3\}$-principal (for $\sym 9$) by Lemma~\ref{lem:almosthookpartitionpprincipal}. Note that $\lambda$ is self-conjugate.
	Thus, $\lambda$ is the only non-trivial partition of~$9$ labeling non-trivial irreducible characters of~$\alt 9$ that are $\{2,3\}$-principal and of degree coprime to~$2$ and $3$.
	Then, we have $\irrprime{2}{B_2(\alt 9)}\cap \irrprime{3}{B_3(\alt 9)}=\{\mathbbm{1}_{\alt 9}, \psi^\lambda_+,\psi^\lambda_-\}$.
\end{example}

Now, we prove the analog of Proposition~\ref{pqSn} for alternating groups $\alt n$, $n\geq 4$. Note that we have to impose $n\geq 4$ such that there exist at least two distinct primes $p$ and $q$ dividing the group order $\abs{\alt n}$.

\begin{proposition} \label{pqAn}
	Let $n\geq 4$ and $p,q\leq n$ be distinct primes.
	Then, we have \[\irrprime p {\blprinc{p}{\alt n}} \cap \irrprime q {\blprinc{q}{\alt n}}\neq \{ \mathbbm{1}_{\alt n} \}.\]
\end{proposition}

In many cases, the statement of Proposition~\ref{pqAn} can be deduced from its analog for $\sym n$ (Proposition~\ref{pqSn}) by restricting characters of the symmetric group $\sym n$ to its normal subgroup $\alt n$. %
We use ad-hoc arguments to complete the remaining cases where Proposition~\ref{pqSn} yields the sign character.

\begin{proof}[Proof of Proposition~\ref{pqAn}]
	Let $n\geq 4$ and $p,q\leq n$ be distinct primes. We reuse the notation of the proof of Proposition~\ref{pqSn}. As before, let $0\leq s<p$ and $0\leq r<q$ such that $n\equiv s \bmod p$ and $n\equiv r \bmod q$, and we keep assuming that $r\geq s$. Write $n=s+\sum_{i=1}^t c_i p^{a_i}$ with $1\leq a_1< a_2< \ldots< a_t$, and $1\leq c_i<p$ for all $1\leq i\leq t$.

	Let $\mu\vdash n$ be the partition constructed in the proof of Proposition~\ref{pqSn} parametrizing a non-trivial character in $\irrprime p {\blprinc{p}{\sym n}} \cap \irrprime q {\blprinc{q}{\sym n}}$.
	For $r\geq 2$, the partition $\mu$ is never equal to~$(1^n)$ nor self-conjugate,
	refer to Table \ref{tab:partitionsforSn} for inspection.
	Thus, the corresponding character $\chi^\mu$ restricts non-trivially and irreducibly to~$\alt n$, as said before. Hence, the same partition then also gives rise to a non-trivial irreducible character of~$\alt n$ lying in the intersection of the principal $p$- and $q$-block, and of degree $d^\mu$ coprime to both $p$ and $q$.

	Now, we are left with the case $s\leq r<2$, i.e., we either have $r=s=0$, or $r=s=1$ or $r=1=s+1$. In this situation, we have $\mu=(1^n)$ which corresponds to the trivial character of~$\alt n$.

	\paragraph*{\textbf{Case 1: $s=r=0$}} Consider the non self-conjugate hook partition $\lambda=(1,n-1)$, which is $\{p,q\}$-principal for $\sym n$, and hence for $\alt n$, by Lemma~\ref{lem:almosthookpartitionpprincipal}. Moreover, its degree  $d^\lambda=n-1$ (by Lemma~\ref{lem:degalmosthookpartition}) is coprime to both $p$ and $q$ as~$pq$ divides~$n$. Thus, the irreducible character $\res{\chi^\lambda}{\alt n}\in \irr {\alt n}$ does the job.

	We are left with $r=s=1$ or $r=1=s+1$. Then the strategy is as follows:\newline
	In any case, we look for a partition $\lambda\vdash n$ such that $\lambda\neq \mu=(1^n)$, which is \linebreak $\{p,q\}$-principal for $\sym n$ (and hence for $\alt n$ by Remark \ref{rem:pprincSnAn}).
	If $2\notin \{p,q\}$, we want a partition $\lambda$ of degree $d^\lambda$ coprime to~$p$ and $q$.
Assuming $\lambda$ is not self-conjugate, the corresponding non-trivial character $\res{\chi^\lambda}{\alt n}$ is irreducible and lies in both the principal $p$- and $q$-block of~$\alt n$, and has degree $d^\lambda$ coprime to both $p$ and $q$. If $\lambda$ is self-conjugate,
	any of the two irreducible constituents $\psi^\lambda_{\pm}$ of~$\res{\chi^\lambda}{\alt n}$ lies in both the principal $p$- and $q$-block of~$\alt n$, and has degree $\frac{d^\lambda}{2}$ coprime to both $p$ and $q$.

	In the case $2\in \{p,q\}$ we either give a non-self-conjugate partition and use the same reasoning as before,
	or in some cases consider a self-conjugate partition $\lambda \vdash n$. In the latter case we want
	that the degree $d^\lambda$ has $2$-part $2$ and is coprime to the odd prime $r\in \{p,q\}\setminus\{2\}$. Then, any of the two irreducible constituents $\psi^\lambda_{\pm}$ of the character $\res{\chi^\lambda}{\alt n}$ (having degree~$\frac{d^\lambda}{2}$) does the job.

	\paragraph*{\textbf{Case 2: } $s=0$ and $r=1$}
	In this case, we have $n=\sum_{i=1}^t c_i p^{a_i}=c_1p^{a_1}+\dots$ as~$s=0$. In particular, $p$ divides~$n$.
	Let $n=q^k \tilde m +1$ with positive integers $k,\tilde m\geq 1$ such that $\tilde m$ is coprime to~$q$.

	We distinguish the following cases: either we have $q^k<c_1 p^{a_1}$ or $q^k > c_1 p^{a_1}$.
	Note that if $q^k < c_1 p^{a_1}$ and $\tilde m=1$, we have $n=c_1p^{a_1}=q^k+1$. In this case, which we treat separately, we also distinguish if $q=2$ or not. Table \ref{tab:Anextracases1} lists the partition $\lambda\vdash n$ considered in any of the four resulting cases.

	\begin{table}[h]
		\begin{center}
			\begin{tabular}{clll}
				\toprule
				Case & Defining conditions            & Partition $\lambda$                         & $\lambda$ self-conjugate?        \\
				\midrule
				2.1 & $q^k<c_1 p^{a_1}$ and $\tilde m\neq 1$ & $(1^{n-q^k-1},q^k+1)$                & no unless $\tilde m=2$ \\
				2.2 & $q^k>c_1 p^{a_1}$                      & $(1^{n-c_1p^{a_1}-2}, 2,c_1p^{a_1})$ & no                     \\
				2.3 & $n=c_1 p^{a_1}=q^k+1$, $q\neq 2$       & $(1^{\frac{n}{2}-2},2,\frac{n}{2})$  & yes                    \\
				2.4 & $n=c_1 p^{a_1}=2^k+1$, $q=2$           & $(1^{2^{k-1}},2^{k-1}+1)$            & yes                    \\
				\bottomrule
			\end{tabular}
		\end{center}
		\caption{Partitions $\lambda$ for non-trivial characters in \linebreak $\irrprime p {\blprinc{p}{\alt n}} \cap \irrprime q {\blprinc{q}{\alt n}}$, assuming $n=q^k\tilde m +1$ and $n\equiv 0\bmod p$.}
		\label{tab:Anextracases1}
	\end{table}

	\subparagraph*{\textit{Case 2.1: } $q^k<c_1 p^{a_1}$ and $\tilde m\neq 1$}
	Consider the hook partition $\lambda=(1^{n-q^k-1},q^k+1)$ which is $p$-principal (as~$s=0$) as well as $q$-principal for $\sym n$ by Lemma~\ref{lem:almosthookpartitionpprincipal}.
	By Lemma~\ref{lem:degalmosthookpprimeor2}~(i)(a) (with $\ell=q^k+1$), its degree $d^\lambda$ is coprime to~$p$ using that $\ell\leq c_1 p^{a_1}$ by assumption.
	By~(i)(b) of the same lemma, we deduce that $d^\lambda$ is coprime to~$q$ as the $q$-part of~$\ell-1$ is $q^k$.
	Note that the partition $\lambda$ is self-conjugate if and only if $\tilde m=2$. In this case, the primes $p$ and $q$ must be odd as~$n=\tilde m q^k+1=2q^k+1$ is odd (so $p$ is odd) and $2=\tilde m$ is coprime to~$q$.
	In total, either the character $\res{\chi^\lambda}{\alt n}$ (irreducible if $\tilde m>2$) or one of its irreducible constituents $\psi^\lambda_{\pm}$ does the job. %

	In the next two cases, we consider an almost hook partition $\lambda=(1^{n-\ell-2},2,\ell)\vdash n$ with $2\leq\ell\leq n-2$ divisible by~$p$. Since we have $r=1$ (resp. $s=0$, so $p$ divides~$n-\ell$), the partition $\lambda$ is $q$-principal (resp. $p$-principal) for $\sym n$, by Lemma~\ref{lem:almosthookpartitionpprincipal}.

	Thus, the only thing to show is that the degree $d^\lambda$ is coprime to both $p$ and $q$, depending on the choice of~$\ell$.

	\subparagraph*{\textit{Case 2.2: }$q^k>c_1 p^{a_1}$}
	Let $\ell=c_1 p^{a_1}$, and consider the partition $\lambda =(1^{n-c_1p^{a_1}-2},2,c_1p^{a_1})$ which is non-self-conjugate. %
	Indeed, the degree $d^\lambda$ is coprime to both $p$ and $q$.
	This follows from Lemma~\ref{lem:degalmosthookpprimeor2}~(ii)(a) (for the prime $p$) using $\ell=c_1p^a$, and~(ii)(b) (for the prime $q$) using $\ell=c_1p^a<q^k$.
	All in all, the %
	irreducible character $\res{\chi^\lambda}{\alt n}\in \irr {\alt n}$ does the job.

	\subparagraph*{\textit{Case 2.3: }$n=c_1 p^{a_1}=q^k+1$ and $q\neq 2$}
	Here, $n=q^k+1$ is even as $q$ is odd, so either $c_1$ is even or we have $p=2$.
	This time, we choose $\ell=n-\ell=\frac{n}{2}=\frac{c_1 p^{a_1}}{2}$ such that the partition $\lambda=(1^{\frac{n}{2}-2},2,\frac{n}{2})$ is self-conjugate.
	Again $\ell=\frac{n}{2}<q^k$ holds, so the degree $d^\lambda$ is coprime to~$q$ by the same reasoning as before (Lemma~\ref{lem:degalmosthookpprimeor2}~(ii)(b)).

	For $p\neq 2$, we have $\ell=\frac{c_1}{2} p^{a_1}$, thus $\ell$ has $p$-part $p^{a_1}$ (the same as~$n=c_1p^{a_1}$) and Lemma~\ref{lem:degalmosthookpprimeor2}~(ii)(a) yields that the degree $d^\lambda$ is coprime to~$p$.

	We are left with the case $p=2$. Then, we have $c_1=1$, so $n=2^{a_1}=q^k+1$ and $\ell=2^{a_1-1}$. (One can show that $k=1$ holds, so $q$ is a so-called Mersenne prime, see Lemma~\ref{lem:fermatmersenne}~(ii) below, but this doesn't matter here.)
	We have $(d^\lambda)_2=2$ by Lemma~\ref{lem:degalmosthookpprimeor2}~(ii)(c).

	In total, as $\lambda$ is self-conjugate, one of the irreducible constituents $\psi_{\pm}^\lambda$ of the character~$\res{\chi^\lambda}{\alt n}$ is a character of~$\alt n$ having all the desired properties.

	\subparagraph{\textit{Case 2.4: }$n=c_1 p^{a_1}=2^k+1$}
	In this case, we have $k>1$ as $n\geq 4$. Consider the self-conjugate hook partition $\lambda=(1^{2^{k-1}},2^{k-1}+1)\vdash n$ which is $\{2,p\}$-principal for $\sym n$ by Lemma~\ref{lem:almosthookpartitionpprincipal} (with $\ell=2^{k-1}+1$).
	Its degree $d^\lambda$ is coprime to~$p$ by Lemma~\ref{lem:degalmosthookpprimeor2}~(i)(a)
	and we have~$(d^\lambda)_2=2$ by Lemma~\ref{lem:degalmosthookpprimeor2}~(i)(c).
	Again, any of the irreducible constituents $\psi_{\pm}^\lambda$ of the character~$\res{\chi^\lambda}{\alt n}$ does the job. \newline
	We remark that if $c_1>1$, we could also consider the almost hook partition $(1^{n-p^{a_1}-2},2,p^{a_1})$, see the proof of Theorem~\ref{thm:C}, or its conjugate $(1^{p^{a_1}-2},2,n-p^{a_1})$. %

	Now, we move to the last big case.

	\paragraph*{\textbf{Case 3: $s=r=1$}}
	Note that in this case, as $p$ and $q$ divide~$n-1$, we have $n\geq 13$ unless $n=7=3\cdot 2+1$ (and $\{p,q\}=\{2,3\}$) or $n=11=5\cdot 2+1$ (and $\{p,q\}=\{2,5\}$). %

	In the following, we always look at a non-self-conjugate almost hook partition $\lambda\vdash n$. Then we use Lemma~\ref{lem:degalmosthookpartition} to determine the degree $d^\lambda$ of the partition $\lambda$ considered. In any case, the partition $\lambda$ will always be $\{p,q\}$-principal for $\sym n$ using that $p$ and $q$ divide~$n-1$, the definition of~$\lambda$, and Lemma~\ref{lem:almosthookpartitionpprincipal}. So, we only have to determine the $p$- and $q$-part of the degree $d^\lambda$ of the partition $\lambda$ to conclude. Then, the irreducible character $\res{\chi^\lambda}{\alt n}$ does the job.

	\subparagraph*{\textit{Case 3.1: }$2\notin \{p,q\}$ or $n\not \equiv 3 \bmod 4$}
	Here, consider the %
	partition $\lambda=(2,n-2)$
	of degree $d^\lambda=\frac{n\cdot (n-3)}{2}$ which is, by construction, coprime to both~$p$ and $q$, and we are done.

	Now, assume $2\in \{p,q\}$, and that $n\equiv 3 \bmod 4$. Without loss, we can assume $p=2$. So, we have $$n=1+2^{a_1}+\sum_{i=2}^t  2^{a_i}=3+\sum_{i=2}^t 2^{a_i}$$ with $a_1=1$ and $a_2\geq 2$. %
	In particular, $n$ is odd and the $2$-part of~$n-3$ is $(n-3)_2=2^{a_2}\geq 2^2$.
	As before, write $n=q^k \tilde m +1$ with $\tilde m\geq 1$ such that $ \tilde m$ is coprime to~$q$. Since~$n$ and $q$ are odd, $\tilde m$ must be even.

	\subparagraph*{\textit{Case 3.2:}  $p=2$ and $n\equiv 3 \bmod 4$ and $q\geq 5$} Consider the partition $\lambda=(1,2,n-3)$ which is non-self-conjugate (as~$n\geq 7$).
	Its degree $d^\lambda=\frac{n \cdot(n-2) \cdot(n-4)}{3}$ is odd as~$n$ is, and coprime to~$q$ as~$q\geq 5$ divides~$n-1$ but not $n\cdot(n-2)\cdot (n-4)$.

	Now, we are left with the situation $q=3$ (remember that we have $p=2$ and still assume $n\equiv 3 \bmod 4$).

	\subparagraph*{\textit{Case 3.3:} $p=2$ and $n\equiv 3 \bmod 4$ and $q=3$} As~$q=3$, we have $n=3^k \tilde m+1$.

	First require that $\tilde m-1$ is coprime to~$3$ if $k=1$. In this case, consider again the partition $\lambda=(1,2,n-3)$ which is non-self-conjugate (as~$n\geq 7$) and $\{2,3\}$-principal for~$\sym n$ by Lemma~\ref{lem:almosthookpartitionpprincipal}.
	Its degree $d^\lambda=\frac{n \cdot(n-2) \cdot(n-4)}{3}$ is odd.
	Note that the $3$-part of the factor $n-4=3\cdot (3^{k-1}\tilde  m-1)$ is always
	$3$, %
	by assumption on $\tilde m$ and $k$. %
	Thus, the degree is coprime to~$3$ as well, and we are done.%

	The remaining case to consider is
	$n=3\tilde m+1$ with $\tilde m-1$ a multiple of~$3$.
	Then, we have $\tilde m\geq 4$ (as~$\tilde m$ is even), and we deduce $n\geq 4\cdot 3+1=13$. %
	Now we look at $(n-3)_2=2^{a_2}\geq 2^2$ and distinguish whether $a_2=2$ or not.

	If $a_2=2$, we have $n=1+2+2^2+\sum_{i=3}^t  2^{a_i}=7+\sum_{i=3}^t  2^{a_i}$, so $(n-7)_2=2^{a_3}\geq 2^3$ holds.
	Consider the almost hook partition $\lambda=(1^{n-12},5,7)\vdash n$
	which is $\{2,3\}$-principal by Lemma~\ref{lem:almosthookpartitionpprincipal} using the fact that $n-4$ is a multiple of~$3$ and $n-7$ is even.
	The degree~$d^\lambda$ is $$d^\lambda= \frac{\prod\limits_{\substack{i=1\\i \neq 3}}^7 (n-7+i) \cdot \prod_{i=1}^4 (n-7-i)}{ \prod\limits_{\substack{i=1\\i\neq 3}}^7 i \cdot \prod_{i=1}^4 i}=\frac{\prod\limits_{\substack{i=0\\i\neq 4,7}}^{11} (n-i)}{\prod_{i=1}^8 i}.$$
	As the $2$-part of~$n-7$ is $(n-7)_2=2^{a_3}\geq 2^3$, we have $(n-7\pm i)_2=i_2$ for any $1\leq i\leq 7<2^3$, and we deduce that $d^\lambda$ is odd.
	The $3$-part of the numerator of~$d^\lambda$ is $((n-1)(n-10))_3=(3\tilde m \cdot 3(\tilde m-3))_3=3^2$ as~$3$ divides~$\tilde m-1$, and thus equals the $3$-part of the denominator. Hence, the degree is coprime to~$3$ and we are done.

	If $a_2\geq 3$, consider the partition $\lambda=(1^{n-8},4^2)$
	which is $\{2,3\}$-principal by Lemma~\ref{lem:almosthookpartitionpprincipal} using that $n-4$ is a multiple of~$3$ and $n-3$ is even. The degree $d^\lambda$ is
	\begin{align*}
		d^\lambda=\frac{\prod_{i=1}^{3} (n-3+i) \cdot \prod_{i=2}^{4} (n-3-i)}{ \prod_{i=1}^{3} i \cdot \prod_{i=2}^{4} i }=\frac{\prod\limits_{\substack{i=0 \\i\neq 3,4}}^{7} (n-i)}{\prod\limits_{\substack{i=1\\i\neq 5}}^6 i}.
	\end{align*}
	For any $1\leq i\leq 4=2^2$, we have $(n-3\pm i)_2=i_2$ as~$(n-3)_2=2^{a_2}\geq 2^3> i_2$. Hence, the degree of the partition $\lambda$ is odd.
	Since~$3$ divides~$\tilde m-1$, the $3$-part of the numerator of~$d^\lambda$ is $((n-1)(n-7))_3=(3\tilde m \cdot 3(\tilde m-2))_3=3^2$ and cancels with the $3$-part of the denominator. Hence, the degree is coprime to~$3$.

	To sum up, in all subcases the irreducible character $\res{\chi^\lambda}{\alt n}$ does the job, as explained before, and we are done with Case 3 ($r=s=1$), which is the last case. 

	All in all, we always find a non-trivial irreducible character of~$\alt n$ lying in both the principal $p$- and $q$-block of~$\alt n$ having degree $d^\lambda$ coprime to both $p$ and $q$.
\end{proof}

\begin{remark}
	As for the symmetric groups in Proposition~\ref{pqSn}, in Proposition~\ref{pqAn} we reprove \cite[Proposition~3.11]{NRS} for alternating groups and  $2\in\{p,q\}$ and extend 
 this result to arbitrary primes $p$,~$q$. %
	Remark that some cases are missing in the original proof of Navarro--Rizo--Schaeffer Fry. However, the addendum \cite{NRSAddendu} handles all of them except for Case~3 (i.e. $r=1=s$), using our notation. In this case, their original proof would still give the inappropriate sign character restricting trivially to the alternating group.\newline
	Our case distinction in Case 2 (i.e. $r=1=s-1$), another missing case fixed in the addendum, is inspired by \cite{NRSAddendu}, see Table \ref{tab:Anextracases1}.
\end{remark}

As a corollary, we conclude that Conjecture \rm{(NRS)} is true for all almost simple groups with socle an alternating group, thus proving Corollary~\ref{coro:B}.

\begin{proof}[Proof of Corollary \ref{coro:B}] Let $A$ be an almost simple group with socle (isomorphic to) $\alt n$.
	A~classical result of H\"older (\cite[Satz 5.13]{sambale}) states that
	$$\aut {\alt n}\cong \begin{cases}
			\sym n,             & \text{if }  n\neq 6, \\
			\sym 6 \rtimes C_2, & \text{if } n=6.      \\
		\end{cases}$$
	Hence, if $A\in \{\alt n, \sym n\}$ (up to isomorphism), the statement follows from Theorem~\ref{thm:A}, %
	and we are done if $n\neq 6$. If $n=6$, we additionally have to consider the full automorphism group $\aut {\alt 6}=\sym 6 \rtimes C_2$ and its two subgroups of index $2$ not isomorphic to~$\sym 6$. Let $A$ be one of these groups, i.e., we let $A\ \in \{\PGL{2}{9}, M_{10}, \aut {\alt 6}\}$ (up to isomorphism). In this case, the conjecture can be checked directly using the Character Table Library of \cite{GAP}. More precisely, there is always a non-trivial %
	character in \linebreak $\irrprime p {\blprinc{p}{A}} \cap \irrprime q {\blprinc{q}{A}}$ for all pairs of distinct primes~$p$, $q$ such that $pq$ divides~$\abs A$ (i.e., $p,q\in \{2,3,5\}$, using~$\abs{\aut {\sym 6}}=2 \cdot \abs{\sym 6} = 2^5 \cdot 3^2 \cdot 5$). In all cases, the assertion of Corollary~\ref{coro:B} follows by contraposition.
\end{proof}

\subsection{The small intersection case} \label{sec:smallintersection}

The main difficulty in the proof of Proposition~\ref{pqAn} occurs whenever $\irrprime p {\blprinc{p}{\sym n}} \cap \irrprime q {\blprinc{q}{\sym n}}=\{\mathbbm{1}_{\sym n}, \sgn\}$, i.e., the set of consideration consists only of the two linear characters of~$\sym n$ (see Example \ref{ex:trivialintersection9}).
The aim of this section is to characterize all tuples $(n,p,q)$ having the latter property and investigate
$\irrprime p {\blprinc{p}{\alt n}} \cap \irrprime q {\blprinc{q}{\alt n}}$ in this case, thus proving Theorem~\ref{thm:C}.

First, we study Fermat and Mersenne numbers.
\begin{definition} A \emph{Fermat} (resp. \emph{Mersenne}) \emph{number} is a positive integer of the form $2^k+1$ (resp. $2^k-1$) for some $k\geq 1$. An odd prime is called \emph{Fermat prime} (resp. \emph{Mersenne prime}) if it is a Fermat (resp. Mersenne) number.%
\end{definition}
The following elementary number-theoretical lemma says that a Mersenne number which is not prime is never a proper prime power. The same holds for Fermat numbers except~$9$.

\begin{lemma} \cite[Lemma 4.5]{BMO} \label{lem:fermatmersenne}
	Let $p$ be an odd prime and $a, k\geq 1$.
	\begin{enumerate}[\rm(i)]
		\item If $p^a=2^k+1$, then we have $a=1$ or $p^a=9$ \textup(i.e., $a=2$ and $p=2$\textup).
		\item If $p^a=2^k-1$ with $k\geq 2$, then we have $a=1$.
	\end{enumerate}
\end{lemma}

\begin{proof}
	In the following, we use the telescoping sums
	\begin{align*}\tag{a}\label{eq:telplus}
		(p-1)\sum_{i=0}^{a-1} p^i=\sum_{i=0}^{a-1}(p^{i+1}-p^i)=p^a-1
	\end{align*}
	and
	\begin{align*}\tag{b}\label{eq:telminus}
		(p+1)\sum_{i=0}^{a-1} (-p)^i & =-((-p)-1)\sum_{i=0}^{a-1} (-p)^i=-\sum_{i=0}^{a-1}((-p)^{i+1}-(-p)^i) \\
		                             & =-((-p)^a-1)=(-1)^{a+1} p^a +1.
	\end{align*}

	(i) Assume that $p^a=2^k+1$ and $a\geq 2$.
	Equation~\eqref{eq:telplus} yields
	$2^k=p^{a}-1=(p-1) \sum_{i=0}^{a-1} p^i $, and it follows that the factor $y\coloneqq \sum_{i=0}^{a-1} p^{i} $ is even. Since~$y$ is a sum of~$a$ odd numbers (as $p$ is odd), $a$ must be even.
	Now, we have $2^k={p^{2\frac{a}{2}}}-1=(p^{\frac{a}{2}}+1)(p^{\frac{a}{2}}-1)$. Write $ p^{\frac{a}{2}}\pm 1 = 2^{t_{\pm}}$ with integers $t_\pm\geq 1$ such that $t_+ + t_- = k$.
	Then, we have $$2^{t_+}=p^{\frac{a}{2}}+ 1=(p^{\frac{a}{2}}- 1) +2 =2^{t_-}+2=2(2^{t_{-}-1}+1).$$
	We conclude that $t_{-}=1$, $t_+=2$, and $p^a=2^k+1=2^{t_{+}+ t_{-}}+1=9$.

	(ii) Assume that $p^a=2^k-1$. Note that $a$ must be odd, otherwise we get \linebreak $1\equiv p^{2\frac{a}{2}}=2^k-1=4 \cdot 2^{k-2}-1\equiv 3 \bmod 4$ using $k\geq 2$, a contradiction.

	Using Equation~\eqref{eq:telminus} and that $a$ is odd, we then obtain the equation \linebreak $2^k=p^a+1=(p+1) \sum_{i=1}^{a-1}(-p)^i$. Set $x\coloneqq \sum_{i=1}^{a-1}(-p)^i$. If $a>1$, the last equation implies that $2$ divides~$x$, %
	but $x\geq 1$ is certainly odd (since~$p$ and $a$ are). Thus, $x=1$ and~$a=1$ follow.
\end{proof}

Next, we look at the set $\irrprime 2 {\blprinc{2}{\sym n}} \cap \irrprime p {\blprinc{p}{\sym n}}$ for a Fermat or Mersenne prime~$p$ and specific choices of~$n$. %

\begin{lemma}\label{lem:pqSnfermatprime}
	Let $p\geq 3$ be a Fermat prime.
	Then, we have $$\irrprime{2}{\blprinc{2}{\sym p}}\cap \irrprime{p}{\blprinc{p}{\sym p}}=\irrprime{\{2,p\}}{\sym p}=\{\mathbbm{1}_{\sym p}, \sgn\}.$$
	Moreover, if  $p=2^k+1$ with $k\geq 2$ \textup(i.e., $p\neq 3$\textup) and $\lambda=(1^{2^{k-1}},2^{k-1}+1)\vdash p$, we have $$\irrprime{2}{\blprinc{2}{\alt p}}\cap \irrprime{p}{\blprinc{p}{\alt p}}=\{\mathbbm{1}_{\alt p},\psi^\lambda_+, \psi^\lambda_-\}.$$
\end{lemma}

Remark that the second statement of the previous lemma is not true for the Fermat prime $p=3=2^1+1$.
Note the alternating group $\alt 3\cong C_3$ is a $3$-group (in particular of odd order). Thus, we have $\irrprime 2 {\blprinc{2}{\alt 3}} \cap \irrprime 3 {\blprinc{3}{\alt 3}}\subseteq \irr {\blprinc{2}{\alt 3}}=\{\mathbbm 1_{\alt 3}\}$.%

\begin{proof}[Proof of Lemma \ref{lem:pqSnfermatprime}]
	Set $n\coloneqq p$.
	Let $\lambda\vdash n$ be a partition of~$n$ of degree coprime to~$p$. The hook formula (alternatively Lemma~\ref{lem:degalmosthookpprimeor2}~(i)(a)) yields that $\lambda$ has a hook of length $p$, so $\lambda$ is a hook itself.
	However, using Lemma~\ref{lem:degalmosthookpprimeor2}~(i)(b) (with $p=2$ and $c=1$), we see that the degree of a non-trivial hook partition of~$2^k+1=n$ is odd if and only if $\lambda=(1^n)$. %
	Thus, we have~$\irrprime{\{2,p\}}{\sym n}=\{\mathbbm{1}_{\sym n}, \sgn\}$. Using that the sign character is $\{2,p\}$-principal by Lemma~\ref{lem:almosthookpartitionpprincipal} (as~$n$ is even and $n=p$), %
	we deduce $\irrprime{2}{\blprinc{2}{\sym n}}\cap \irrprime{p}{\blprinc{p}{\sym n}}=\{\mathbbm{1}_{\sym n}, \sgn\}$.

	Now, let $p=2^k+1$ with $k\geq 2$.
	The partition $\lambda=(1^{2^{k-1}},2^{k-1}+1)\vdash n$ is the only partition of~$n=2^k+1$
	fulfilling $(d^\lambda)_2=2$, by Lemma~\ref{lem:degalmosthookpprimeor2}~(i)(c).
	Then, we have that $\irrprime {\{2,p\}} {\alt n}=\{\mathbbm{1}_{\alt n},\psi^\lambda_+, \psi^\lambda_-\}$.
	Since~$k\geq 2$, %
	the partition $\lambda$ is $\{2,p\}$-principal (for $\sym n$) using Lemma~\ref{lem:almosthookpartitionpprincipal}. Thus,
	the characters $\psi^\lambda_+, \psi^\lambda_-\in \irr {\alt n}$ are also $\{2,p\}$-principal for $\alt n$.
	We conclude that $\irrprime 2 {\blprinc{2}{\alt n}} \cap \irrprime p {\blprinc{p}{\alt n}}=\irrprime {\{2,p\}} {\alt n} =\{\mathbbm{1}_{\alt n},\psi^\lambda_+, \psi^\lambda_-\}$ using that $\lambda$ is self-conjugate and Remark \ref{rem:pprincSnAn}.%
\end{proof}

\begin{lemma}\label{lem:pqSnfmersenneprime}
	Let $p\geq 3$ be a Mersenne prime.
	Then, we have $$\irrprime{2}{\blprinc{2}{\sym {p+1}}}\cap \irrprime{p}{\blprinc{p}{\sym {p+1}}}=\irrprime{\{2,p\}}{\sym {p+1}}=\{\mathbbm{1}_{\sym {p+1}}, \sgn\}.$$
	Moreover, if $p=2^k-1$ with $k\geq 2$ and $\lambda=(1^{2^{k-1}-2},2, 2^{k-1})\vdash p+1$, we have  $$\irrprime{2}{\blprinc{2}{\alt {p+1}}}\cap \irrprime{p}{\blprinc{p}{\alt {p+1}}}=\{\mathbbm{1}_{\alt {p+1}},\psi^\lambda_+, \psi^\lambda_-\}.$$
\end{lemma}

\begin{proof}
	Set $n\coloneqq p+1=2^k$.
	By a result of Macdonald \cite[Corollary 1.3]{MacD} there are exactly $2^k$ partitions of~$2^k=n$ of odd degree (i.e., $\abs{\irrprime{2}{\sym {2^k}}}=2^k$).
	These are the hook partitions of~$n$ %
	as their degree is indeed odd by Lemma~\ref{lem:degalmosthookpprimeor2}~(i)(a).
	However, the degree of a hook partition $\lambda\vdash n$ is always divisible by~$p$ unless $\lambda\in\{(n), (1^n)\}$.
	We conclude that
	$$\irrprime 2 {\blprinc{2}{\sym n}} \cap \irrprime p {\blprinc{n}{\sym n}}=\irrprime{\{2,n\}}{\sym n}=\{\mathbbm{1}_{\sym n}, \sgn\}$$
	as the sign character is $\{2,p\}$-principal by Lemma~\ref{lem:almosthookpartitionpprincipal} (as~$n$ is odd and $n-1=p$).

	Now, we examine the set $\irrprime{2}{\blprinc{2}{\alt {n}}}\cap \irrprime{p}{\blprinc{p}{\alt {n}}}$.
	Note that, by the hook formula, a partition $\lambda$ of~$n=p+1$ of degree not divisible by~$p$ must admit a $p$-hook. Thus, we have $\lambda \in\{(1^n), (n)\}$, or $\lambda$ is an almost hook partition of the form $\lambda=\lambda_\ell\coloneqq (1^{n-\ell-2},2,\ell)$ with $2\leq \ell\leq n-2$.
	According to Lemma~\ref{lem:degalmosthookpprimeor2}~(ii)(b), the degree $d^\lambda$ of such an almost hook partition is indeed coprime to~$p$ as~$\ell<p$ in all cases.
	Further, by Lemma~\ref{lem:degalmosthookpprimeor2}~(ii)(c), we have $(d^\lambda)_2=2$  if and only if $\ell=2^{k-1}$.
	Set $\mu=(1^{2^{k-1} -2},2,2^{k-1})\vdash n$.
	Observe that the partition $\mu$ is $\{2,p\}$-principal (for $\sym n$) by Lemma~\ref{lem:almosthookpartitionpprincipal}.
	We conclude that $$\irrprime 2 {\blprinc{2}{\alt n}} \cap \irrprime p {\blprinc{p}{\alt n}}=\{\mathbbm{1}_{\alt n}, \psi^\mu_+, \psi^\mu_-\}$$ using that $\lambda$ is self-conjugate and Remark \ref{rem:pprincSnAn}.
\end{proof}

Finally, we characterize the phenomenon $\irrprime p {\blprinc{p}{\sym n}} \cap \irrprime q {\blprinc{q}{\sym n}}=\{\mathbbm{1}_{\sym n}, \sgn\}$ using the previous lemmas.
In particular, this can only be true if $2\in \{p,q\}$, say $q=2$, %
and $p$ is a Fermat or Mersenne prime. %
For convenience, we require $q<p$ in Theorem~\ref{thm:C}, even though this is inconsistent with the notation used previously in the proofs of Proposition~\ref{pqSn} and Proposition~\ref{pqAn}.

\thmpqlinAnSn

\begin{proof}
	In the proofs of Proposition~\ref{pqSn} and Proposition~\ref{pqAn} we give a partition $\lambda \vdash n$ such that the corresponding irreducible character $\chi^\lambda$ lies in $\irrprime{p}{\blprinc{p}{\sym n}}\cap \irrprime{q}{\blprinc{q}{\sym n}}$ but $\chi^\lambda \notin \{\mathbbm{1}_{\sym n}, \sgn\}$ (i.e., $\lambda \notin \{(1^n),(n)\}$) unless $q=2$ and one of the following hold:
	\begin{itemize}
		\item $n=cp^a=2^k+1$ with positive integers $a\geq 1$, $k\geq 2$, %
		      and $1\leq c<p$, or
		\item $n=2^k=p^a+1$ with positive integers $a\geq 1$, $k\geq 2$.
	\end{itemize}

	First, let $n=cp^a=2^k+1$ with positive integers $a\geq 1$, $k\geq 2$, and $1< c<p$.
	Consider the almost hook partition $\lambda=(1^{n-p^a-2},2,p^a)\vdash n$ which is well-defined as~$c\neq 1$.
	Lemma~\ref{lem:almosthookpartitionpprincipal} (with $\ell=p^a$) says that $\lambda$ is $\{2,p\}$-principal (for $\sym n$) using that $n-1$ is odd and $p$ divides~$n-\ell=(c-1)p^a$. Moreover, the degree $d^\lambda$ is coprime to~$p$ and odd using Lemma~\ref{lem:degalmosthookpprimeor2}~(ii)(a) and~(ii)(b) (with $\ell=p^a$) noticing that $\ell_p=p^a$ and $\ell<2^k$, respectively.
	In total, we get $\irrprime 2 {\blprinc{2}{\sym n}} \cap \irrprime p {\blprinc{p}{\sym n}}\neq \{\mathbbm{1}_{\sym n}, \sgn\}.$

	We are left with the following three cases:
	If $n=p^a=2^k+1$ (i.e., $c=1$ in the previous case), Lemma~\ref{lem:fermatmersenne}~(i) shows that either $a=1$ (in this case $n=p$ is a Fermat prime), or $n=9$ and $p=3$.
	The last case is $n=2^k=p^a+1$ with $k\geq 2$, $a\geq 1$, as before. Lemma~\ref{lem:fermatmersenne}~(ii) shows that we even have $a=1$, i.e., $p=n-1=2^k-1$ is a Mersenne prime.
	Thanks to Example \ref{ex:trivialintersection9} and Lemma~\ref{lem:pqSnfermatprime} (in the first two cases) and Lemma~\ref{lem:pqSnfmersenneprime} (in the Mersenne case), we get $\irrprime p {\blprinc{p}{\sym n}} \cap \irrprime 2 {\blprinc{2}{\sym n}}=\irrprime{\{2,p\}}{\sym n}=\{\mathbbm{1}_{\sym n},\sgn\}$. Moreover, we have
	$\irrprime p {\blprinc{p}{\alt n}} \cap \irrprime 2 {\blprinc{2}{\alt n}}=\{\mathbbm{1}_{\alt n},\psi^\lambda_+, \psi^\lambda_{-}\}$
	where
	$$\lambda=\begin{cases} (1^{4}, 5),                  & n=9,       \\
              (1^{2^{k-1}}, 2^{k-1}+1),    & n=p=2^k+1, \\
              (1^{2^{k-1}-2}, 2, 2^{k-1}), & n=2^k=p+1.\end{cases}$$
	By definition, the irreducible characters $\psi^\lambda_+, \psi^\lambda_{-} \in \irr {\alt n}$ are the irreducible constituents of~$\res{\chi^\lambda}{\alt n}$. Equivalently, they both induce the irreducible character $\chi^\lambda\in \irr {\sym n}$. Setting~$\psi_{+}\coloneqq \psi^\lambda_{+}$ and $\psi_{-}\coloneqq \psi^\lambda_{-}$, the result follows.
\end{proof}

Remark that, by Theorem~\ref{thm:C}, the assertions $\irrprime {\{p,q\}} {\sym n}=\{\mathbbm{1}_{\sym n},\sgn\}$ and \linebreak $\irrprime p {\blprinc{p}{\sym n}} \cap \irrprime q {\blprinc{q}{\sym n}}=\{\mathbbm{1}_{\sym n},\sgn\}$ are equivalent.

\begin{remark} In 2019, Giannelli, Schaeffer Fry, and Vallejo study whether \linebreak $\irrprime {\{p,q\}} G \neq \{\mathbbm{1}_G\}$ for a finite group $G$ and distinct primes $p,q$,
	and show that this is true if and only if $G\neq \{1\}$ \cite[Theorem A]{GSV}.
	A priori, Proposition~\ref{pqSn} and Proposition~\ref{pqAn} imply their result for $G\in \{\sym n, \alt n\}$ and $n\geq 4$.
	However, for both groups and~$n\geq 5$, one can even show that the partition given in the proofs of \linebreak \cite[Theorem 2.8]{GSV} and \cite[Theorem 2.11]{GSV} labeling a non-trivial irreducible character of degree coprime to~$p$ and $q$  %
	is also $\{p,q\}$-principal.
	This gives another, but different, way to prove Proposition~\ref{pqSn} and Proposition~\ref{pqAn}. In contrast to our approach, their main tool is to compare the $p$- and $q$-adic expansion of~$n$. However, this idea does not generalize to projective special linear and unitary groups studied in the next chapter.\newline
	In  \cite[Theorem 2.8]{GSV}, Giannelli--Schaeffer~Fry--Vallejo also characterize tuples~$(n,p,q)$ with $\irrprime{\{p,q\}}{\sym n}=\{\mathbbm{1}_{\sym n},\sgn\}$ for $n\geq 5$, analogously to Theorem~\ref{thm:C}, without taking blocks into the picture.
	Using Lemma~\ref{lem:fermatmersenne} to refine their statement, we see that the lists of possible numbers $(n,p,q)$ in \cite[Theorem 2.8]{GSV} and Theorem~\ref{thm:C} agree.
	Notice that the ``only if'' direction of Theorem~\ref{thm:C} is stronger than the ``only if'' direction of their result. \newline 
	At the time of proving the main results of this section, %
	we were unaware of the cited paper, and we thank the authors for bringing their paper to our attention.
\end{remark}

\section{Classical groups of Lie type} \label{sec:2}
This section is devoted to the proof of Theorem A for finite (simple) classical groups of type~$\type{A}$, $\type{B}$ or $\type{C}$.
\subsection{Background}\label{sec:classicaltypes}
Here, we recollect some facts about groups of Lie type and their unipotent characters that we use throughout the rest of this article. For the statements recalled here, see the references \cite{MT}, \cite{GM}, and \cite{CE}, unless specified otherwise.

Let $\GG$ be a simple algebraic group of simply connected type, defined over $\overline{\F_{\QQ}}$ where $\QQ$ is a prime power. Consider a Frobenius endomorphism $F\colon \GG \to \GG$ endowing $\GG$ with an $\F_{\QQ}$-structure. Then, the group $G\coloneqq \GG^F$ is finite and, up to finitely many exceptions, the quotient $S\coloneqq G/\zent G$ of $G$ by its center $\zent G$ is a finite simple group \emph{of Lie type}. In fact, all finite simple groups of Lie type arise in this way (if we more generally allow $F$ to be a Steinberg endomorphism). %

By Lusztig's fundamental work, the set of irreducible characters of $G$ admits a partition $$\irr G=\bigcup_{s}\mathcal{E}(G,s)$$ into so-called \emph{Lusztig series} $\mathcal{E}(G,s)$, where $s$ runs through the semisimple elements of the dual finite group up to conjugation. The \emph{unipotent characters} are the irreducible characters in the distinguished set $\Uch(G)\coloneqq \mathcal{E}(G,1)$, e.g., the trivial character $\mathbbm{1}_G$ is unipotent. \newline
Since any unipotent character $\chi$ has the center $\zent G$ in its kernel, its deflates to an irreducible character $\bar \chi$ of $S$, and the corresponding map $\Uch(G) \to \irr S$, $\chi\mapsto \bar \chi$ is degree-preserving. Moreover, by a general group-theoretic argument, given a non-defining prime $p$ (i.e., $p$ not diving $\QQ$), two unipotent characters $\chi$, $\psi\in \Uch(G)$ lie in the same $p$-block of $G$ if and only if their deflations $\bar \chi, \bar \psi \in \irr S$ lie in the same $p$-block of $S$ \linebreak \cite[Lemma~17.2]{CE}.
\newline
Note that any unipotent character lies in \emph{the unipotent $p$-block} $$\mathcal{E}_{p}(G,1)\coloneqq \bigcup_{t} \mathcal{E}(G,t)$$ where $t$ runs through the semisimple $p$-elements of the dual group. %
The unipotent $p$-block is itself a union of $p$-blocks and each of them contains at least one unipotent character (whence the name) \cite[Theorem~9.12]{CE}).%

Coming back to our problem, we are interested in the set $\irrprime p {\blprinc p G} \cap \irrprime q {\blprinc q G}$ for two distinct non-defining primes $p$ and $q$. By the above compatibility of Lusztig series and (unipotent) blocks, we have $$\irr {\blprinc{p}{G}} \cap \irr{\blprinc{q}{G}} \subseteq \mathcal{E}_{p}(G,1) \cap \mathcal{E}_{q}(G,1)=\Uch(G)$$
(since any irreducible character in $\mathcal{E}_{p}(G,1)\cap \mathcal{E}_{q}(G,1)$ lies in an intersection $\mathcal{E}(G,x)\cap \mathcal{E}(G,y)$ of Lusztig series labeled by a $p$-element $x$ and a $q$-element $y$, thus $x=y=1$).

From now on, let $G$ be of classical type $\type{A}_{n-1}$, $\tw{2}{\type{A}_{n-1}}$, $\type{B}_n$ or $\type{C}_n$ (note that we do not treat the $\type{D}$-types in this paper). For a classical group $G$, unipotent characters are labeled by combinatorial objects (partitions or related objects called \emph{symbols}) which also determine the character degrees. %
They also encode whether two unipotent characters belong to the same $p$-block, for a non-defining prime $p$.
Note that any unipotent character of $G$ lies in the principal $2$-block, given an odd prime power $\QQ$, by \cite[Theorem 13]{CE93}.%

Now, the strategy to prove Theorem~\ref{thm:A} for $S$ is to show \linebreak $\irrprime p {\blprinc p G} \cap \irrprime q {\blprinc q G}\neq \{\mathbbm{1}_G$\} for any two non-defining distinct primes $p$, $q$ dividing~$\abs G$ (equivalently dividing $\abs S$) and, by what we said above, deduce that the same is true for $S$.

The aim of the next section is to develop tools tailored to the degree analysis of unipotent character degrees. %
The third section develops the number-theoretical framework used in the proof of Theorem~\ref{thm:A} for classical groups.
Both Section~\ref{sec:degreeanalysis} and Section~\ref{eppadicsetup} are of purely number-theoretical nature. In Section~\ref{sec:prooftypeA} and Section~\ref{sec:prooftypesBC}, we eventually come back to our setting of classical groups where we deal with $G$ of type~$\type{A}$ and types $\type{B}$ or $\type{C}$, respectively.
Remark that each type of group requires a separate (but often similar) case distinction.

\subsection{\texorpdfstring{$p$}{p}-parts of \texorpdfstring{$x^m\pm 1$}{xᵐ±1}} \label{sec:degreeanalysis}
In this section, let $p$ be a prime, $x$ an integer coprime to~$p$ and fix a positive integer $f\geq 1$.
For any positive integer $m\geq 1$ denote the order of~$x$ modulo~$m$ by~$\ord_m(x)$,
and the polynomials $\Psi_m^{-}\coloneqq X^m-1\in \Q[X]$ and $\Psi_m^{+}\coloneqq X^m+1\in \Q[X].$  %
Recall the factorization $\Psi_m^{-}=X^m-1=\prod_{d\mid m} \Phi_d \in \Q[X]$, where $\Phi_d$ is the \emph{$d$th cyclotomic polynomial} over $\Q$.

The goal of this section is to give both necessary and sufficient conditions for divisibility of the polynomial values $\Psi_f^{\pm}(x)$ by~$p$, and in the divisibility case compare $p$-parts of~$\Psi_f^{\pm}(x)$ with $p$-parts of~$\Psi_{f'}^{\pm}(x)$, for $f'\geq 1$. %
Most of the work concerns the case of an odd prime $p$.

In the first two lemmas, we give a well-known elementary condition when the polynomial values $\Psi_f^{\pm}(x)$ (resp. $\Phi_f(x)$) are divisible by~$p$.

\begin{lemma} \label{lem:divpsi}
	\begin{enumerate}[\rm(i)]
		\item If $p$ divides~$\Psi_f^{-}(x)$ \textup(resp. $\Psi_f^{+}(x)$\textup), then $\ord_p(x)$ \textup(resp. $\ord_p(x^2)$\textup) divides~$f$.

		      In both cases,  $\ord_p(x^2)$ divides~$f$.

		\item The prime $p$ divides~$\Psi_{2f}^{-}(x)$ if and only if $\ord_p(x^2)$ divides~$f$. %
	\end{enumerate}
\end{lemma}

\begin{proof}
	(i)	If $p$ divides~$\Psi_f^{-}(x)$, then we have $x^f\equiv 1\bmod p$, hence $\ord_p(x)$ divides~$f$. \\
	If $p$ divides~$\Psi_f^{+}(x)$ or $\Psi_f^{-}(x)$, using $\Psi_f^{-}(x)\Psi_f^{+}(x)=\Psi_{2f}^{-}(x)=\Psi_{f}^{-}(x^2)$, we get $\ord_p(x^2)$ divides~$f$ by the same argument as before.

	(ii) %
	The ``only if'' direction follows directly from~(i), again using $\Psi_{2f}^{-}(x)=\Psi_f^{-}(x^2)$.
	If~$\ord_p(x^2)$ divides~$f$, there is an integer $m\geq 1$ with $f=\ord_p(x^2)m$, and we have \linebreak $\Psi_{2f}^{-}(x)=x^{2\ord_p(x^2)m}-1\equiv 0\bmod p$, so we are done.
\end{proof}
The first part of the next lemma says that
$p$ divides the value $\Phi_f(x)$ of the $d$th cyclotomic polynomial evaluated at $x$ (coprime to~$p$) if and only if $f$ has $p'$-part $f_{p'}=\ord_p(x)$. Then we compute $p$-parts of~$\Phi_f(x)$ for such $f$.

\begin{lemma}  \label{lem:cyclo}
	\begin{enumerate}[\rm(i)]
		\item  \cite[Lemma 5.2]{MalleH0}
		      The prime $p$ divides~$\Phi_f(x)$ if and only if \linebreak $f=p^k \ord_p(x)$ for some non-negative integer $k\geq 0$.%
		\item \cite[Lemma 5.2]{MalleH0}
		      If $p^2$ divides~$\Phi_f(x)$ and $p$ is odd, then $f=\ord_p(x)$. Moreover, if $p=2$, then $4$ divides~$\Phi_f(x)$ if and only if $f=\ord_4(x)$.

		\item If $p$ is odd, $\Phi_{p^k\ord_p(x)}(x)$ has $p$-part $(\Phi_{p^k\ord_p(x)}(x))_p=p$ for any $k\geq 1$.
		\item %
		      The value $\Phi_{2^k}(x)$ has $2$-part $(\Phi_{2^k}(x))_2=2$, for any $k\geq 2$, and $\Phi_1(x)\Phi_2(x)$ has $2$-part $2^3$.
	\end{enumerate}
\end{lemma}

\begin{proof}
	If $p$ divides~$\Phi_f(x)$, then $p$ divides~$\Psi_f^{-}(x)$, hence $\ord_p(x)$ divides~$f$ by Lemma~\ref{lem:divpsi}~(i).
	The rest of the proof of~(i) and~(ii) can be found in \cite[Lemma~5.2]{MalleH0}.
	The assertion in~(iii) follows directly from~(i) and~(ii).

	(iv) Assume now that $p=2$. Recall that $\Phi_1(x)=x-1$ and $\Phi_{2^k}(x)=x^{2^{k-1}}+1$ for any~$k\geq 1$. Since square numbers are always congruent to~$1$ modulo $4$, the $2$-part of~$\Phi_{2^k}(x)$, for any $k\geq 2$, is $(\Phi_{2^k}(x))_2=2$.
	If $\ord_4(x)=1$ (i.e., $x\equiv 1 \bmod 4$), then we have $(\Phi_1(x))_2=4$ and $(\Phi_2(x))_2=2$, while for $\ord_4(x)=2$ (i.e., $x\equiv 3 \bmod 4$), we have $(\Phi_2(x))_2=4$ and $(\Phi_1(x))_2=2$. In total, we get $(\Phi_1(x))_2=2^{3-\ord_4(x)}$ and $(\Phi_2(x))_2=2^{\ord_4(x)}$, thus $(\Phi_1(x)\Phi_2(x))_2=2^3$. This also reproves~(ii) for $p=2$.
\end{proof}

As a corollary, we decide whether the $p$-parts of~$\Psi_n^{\pm}(x)$ and $\Psi_m^{\pm}(x)$ are equal, given positive integers $n,m\geq 1$ divisible by~$\ord_p(x)$ (later in this section weakened to~$\ord_p(x^2)$), using knowledge on the $p$-parts $n_p$ and~$m_p$.

\begin{corollary} \label{coro:psipparts}
	Let $m,n\geq 1$ be positive integers such that \fbox{$\ord_p(x)$} divides both $n$ and~$m$.
	\begin{enumerate}[\rm(i)]
		\item The $p$-parts of~$\Psi_m^{-}(x)$ and $\Psi_n^{-}(x)$ are equal if and only if $m_p=n_p$.
		\item If $m_p=n_p$, then the $p$-part of~$\Psi_m^{+}(x)$ equals the $p$-part of~$\Psi_n^{+}(x)$. For odd $p$ dividing~$\Psi^{+}_m(x)$, the converse also holds.
	\end{enumerate}
\end{corollary}

\begin{proof}
	Let $n_p=p^a$ with $a\geq 0$. To determine the $p$-parts of~$\Psi_n(x)=\prod_{d\mid n}\Phi_d(x)$ we have to take into account only the $p$-parts of the factors $\Phi_d(x)$ where $d$ is of the form $d=\ord_p(x)p^k$ with $0\leq k\leq a$ by Lemma~\ref{lem:cyclo}~(i).

	First assume that $p$ is odd.
	Setting $d\coloneqq (\Phi_{\ord_p(x)}(x))_p$ and applying Lemma~\ref{lem:cyclo}~(iii),
	we see that the $p$-part of~$\Psi_n^{-}(x)$ is
	\begin{align*}\tag{a}\label{eq:oddp}
		(\Psi_n^{-}(x))_p
		=\prod_{k=0}^a (\Phi_{p^k \ord_p(x)}(x))_p=d \cdot \prod_{k=1}^a (\Phi_{p^k \ord_p(x)}(x))_p
		= d \cdot \prod_{k=1}^a p
		=p^a \cdot d
	\end{align*}
	since the values~$\Phi_{p^k \ord_p(x)}(x)$ all have $p$-part $p$, for running $k\geq 1$.

	For $p=2$, similarly to before, using Lemma~\ref{lem:cyclo}~(i), we have $(\Psi_n^{-}(x))_2=\prod_{k=0}^a (\Phi_{2^k}(x))_2$, thus $(\Psi_n^{-}(x))_2=(\Phi_1(x))_2$ if $a=0$. Otherwise, we have $a\neq 0$. Setting $d'\coloneqq (\Phi_1(x)\Phi_2(x))_2$, Lemma~\ref{lem:cyclo}~(iv) gives
	\begin{align*}\tag{b}\label{eq:evenp}
		(\Psi_n^{-}(x))_2
		=\prod_{k=0}^a (\Phi_{2^k}(x))_2
		=d' \cdot \prod_{k=2}^{a} (\Phi_{2^k}(x))_2
		=2^3 \cdot \prod_{k=2}^{a} 2
		=2^3 \cdot 2^{a-1}=2^{a+2}.
	\end{align*}

	(i) First, assume that $m_p=n_p=p^a$. Now, $(\Psi_n^{-}(x))_p=(\Psi_m^{-}(x))_p$ is immediate from Equation~\eqref{eq:oddp} and Equation \eqref{eq:evenp}.

	Conversely, assume that $(\Psi_n^{-}(x))_p=(\Psi_m^{-}(x))_p$.  Let $n_p=p^a$ as before, and $m_p=p^b$ with $b\geq 0$. For odd $p$, Equation~\eqref{eq:oddp} yields $(\Psi_n^{-}(x))_p= p^a \cdot (\Phi_{\ord_p(x)}(x))_p$, and we get $(\Psi_m^{-}(x))_p=p^b \cdot (\Phi_{\ord_p(x)}(x))_p$. Hence, $a=b$ follows.\\
	For $p=2$ and $a\neq 0$, the reasoning is similar using Equation~\eqref{eq:evenp} and the same conclusion as before holds. Note that if $a=0$ and we assume $b\neq 0$, then we get the contradiction $2^2<2^{b+2}=(\Psi_m^{-}(x))_2=(\Psi_n^{-}(x))_2=(\Phi_1(x))_2\leq 2^2$. Hence, $a=b$ holds in all cases, i.e., the integers $n$ and $m$ have the same $p$-part.

	(ii)  Assuming that $m_p=n_p$,~(i) gives $(\Psi_n^{-}(x))_p=(\Psi_m^{-}(x))_p$ and $(\Psi_{2n}^{-}(x))_p=(\Psi_{2m}^{-}(x))_p$ (since we have $(2n)_p=(2m)_p$). The assertion $(\Psi_n^{+}(x))_p=(\Psi_m^{+}(x))_p$ follows directly using the factorization $\Psi_{2f}^{-}(x)=\Psi_f^{-}(x) \Psi_f^{+}(x)$ for $f\in\{n,m\}$.

	Conversely, if $1\neq (\Psi_n^{+}(x))_p=(\Psi_m^{+}(x))_p$ and $p$ is odd, %
	then $p$ divides neither $ \Psi_n^{-}(x)$ nor $\Psi_m^{-}(x)$ (otherwise $-1\equiv x^n\equiv 1 \bmod p$, contradicting $p\neq 2$). Using the same trick as before, we get $(\Psi_{2n}^{-}(x))_p=(\Psi_{2m}^{-}(x))_p$, i.e., $(2n)_p=(2m)_p$ by~(i), hence $n_p=m_p$ follows.
\end{proof}

Be aware that, if $p$ does not divide~$\Psi_n^{+}(x)$ nor $\Psi_m^{+}(x)$, i.e., $1=(\Psi_n^{+}(x))_p=(\Psi_m^{+}(x))_p$, %
a priori, we cannot say anything about the relation of the $p$-parts $n_p$ and $m_p$.
\begin{remark}
	Note that for $p=2$ the assumption of the corollary is always fulfilled (as~$\ord_2(x)=1$).\\
	When we want to apply Corollary~\ref{coro:psipparts}, we might be given $\ord_p(x^2)$ (see below) instead of~$\ord_p(x)$. However, note that for any $n\geq 1$, the assumption $\ord_p(x)$ dividing~$2n$ is equivalent to~$\ord_p(x^2)$ dividing~$n$ using $\ord_p(x^2)=\frac{\ord_p(x)}{\gcd(\ord_p(x),2)}$.
\end{remark}
The next goal is to give sufficient conditions guaranteeing divisibility of the polynomial values $\Psi_f^{\pm}(x)$ by~$p$ whenever $p$ is odd. Here, $\ord_p(x^2)$ comes into the picture. We finally decide whether the $p$-parts of~$\Psi\in \{\Psi^{+}_{f}, \Psi_f^{-}\}$ and $\Psi'\in \{\Psi_{f'}^{+}, \Psi_{f'}^{-}\}$ are equal, for $f'\geq 1$ and $f,f'$ both divisible by~$\ord_p(x^2)$.

Recall the following observations we used so far: Whenever $p$ divides~$\Psi_f^{-}(x)$ or $\Psi_f^{+}(x)$, we automatically get that $p$ divides the product $\Psi_f^{-}(x)\Psi^{+}_f(x)=\Psi_{2f}^{-}(x)=\Psi_f(x^2)$, hence $\ord_p(x^2)$ divides~$f$, as seen in Lemma~\ref{lem:divpsi}.
Conversely, if we assume that $p$ is odd and $\ord_p(x^2)$ divides~$f$, then $p$ divides either $\Psi_f^{-}(x)$ or $\Psi_f^{+}(x)$,
but not both, otherwise we get the contradiction $1\equiv x^f \equiv -1\bmod p$.

\begin{notation} For the rest of the section, the prime $p$ is assumed to be odd.
\end{notation}

Note that $p$ divides~$\Psi_{\ord_p(x^2)}^{-}(x^2)=\Psi_{2\ord_p(x^2)}^{-}(x)=\Psi_{2f}^{-}(x)$, so $p$ divides either $\Psi_{\ord_p(x^2)}^{-}(x)$ or $\Psi_{\ord_p(x^2)}^{+}(x)$ by what we said previously. %
This leads to the following notion which is due to Fong--Srinivasan \cite[5.]{FSclassicalgroups} who defined it in the context of~$x$ being a prime power, say $\QQ$, in the setting of classical groups.

\begin{definition} \label{def:FSunitarylinear}
	Let $p$ be odd. %
	We call $p$ \emph{linear for $x$} if $p$ divides~$\Psi_{\ord_p(x^2)}^{-}(x)$ and \emph{unitary for $x$} if $p$ divides~$\Psi_{\ord_p(x^2)}^{+}(x)$. For short, we might simply say that $p$ is linear or unitary.
\end{definition}

\begin{remark} We rephrase the definition of linear and unitary primes in terms of~$\ord_p(x)$.
	Note that
	$\ord_p(x^2)=\frac{\ord_p(x)}{\gcd(\ord_p(x),2)},$
	thus the prime $p$ is linear or unitary for $x$ if and only if $\ord_p(x)$ is odd or even, respectively. Equivalently, $p$ is linear (resp. unitary) if and only if $\ord_p(x)=\ord_p(x^2)$ (resp.$\ord_p(x)=2 \ord_p(x^2)$).
\end{remark}

Now, we state necessary and sufficient criteria deciding whether $p$ divides~$\Psi_f^{\pm}(x)$ bringing $\ord_p(x^2)$ to the picture.

\begin{lemma} %
	\label{lem:divpsiep}
	Let $p$ be odd.
	\begin{enumerate}[\rm(i)]
		\item %
		      Suppose that $p$ is linear for $x$.
		      Then, $p$ does not divide~$\Psi_f^{+}(x)$, and $p$ divides~$\Psi_f^{-}(x)$ if and only if $\ord_p(x^2)$ divides~$f$.
		\item %
		      Suppose that $p$ is unitary for $x$.
		      Then, $p$ divides~$\Psi_f^{-}(x)$ or $\Psi_f^{+}(x)$ if and only if $f$ is an even or odd multiple of~$\ord_p(x^2)$, respectively.

	\end{enumerate}
\end{lemma}

\begin{proof}

	(i) Since~$p$ is linear, we have $\ord_p(x)=\ord_p(x^2)$.
	Note first that $p$ does not divide~$\Psi_f^{+}(x)=x^f+1$, otherwise $\ord_p(x^2)=\ord_p(x)$ divides~$f$ by Lemma~\ref{lem:divpsi}~(i), in particular $p$ divides $\Psi_f^{-}(x)$. In other words, $p$ divides both $\Psi_f^{+}(x)$ and $\Psi_f^{-}(x)$, contradicting $p$ being odd.

	If $p$ divides~$\Psi_f^{-}(x)$, then $\ord_p(x^2)$ divides~$f$ by Lemma~\ref{lem:divpsi}~(ii).
	Conversely, if $\ord_p(x^2)$ divides~$f$, i.e., $p$ divides~$\Psi_{f}^{-}(x^2)=\Psi_{2f}^{-}(x)=\Psi_f^{-}\Psi_f^{+}$, then $p$ must divide~$\Psi_f^{-}(x)$ since~$p$ does not divide~$\Psi_f^{+}(x)$ as $p$ is odd.

	(ii) As $p$ is unitary, we have $\ord_p(x)=2\ord_p(x^2)$.
	Without loss of generality suppose that $p$ divides either $\Psi_f^{-}(x)$ or $\Psi_f^{+}(x)$, i.e., $p$ divides~$\Psi_{2f}^{-}(x)$, which is equivalent to~$\ord_p(x^2)$ dividing~$f$ by Lemma~\ref{lem:divpsi}~(ii). In the opposite case, the assertion in~(ii) is trivial, again by the same lemma. Henceforth, we assume that $f =\ord_p(x^2)m$ with $m\geq 1$.

	If $m$ is even, then we have $x^f=x^{\ord_p(x^2) m}={x^2}^{\ord_p(x^2)\frac{m}{2}}\equiv 1 \bmod p$, thus $p$ divides~$\Psi_f^{-}(x)$. Conversely, if $p$ divides~$\Psi_f^{-}(x)$, then we have $f=\ord_p(x)a=2\ord_p(x^2)a$ for some $a\geq 1$ by Lemma~\ref{lem:divpsi}~(i), hence %
	$m=2a$ must be even.\newline
	The second equivalence follows by contraposition.
\end{proof}

Let us reformulate the previous lemma from the point of view of~$\Psi_f^{\pm}(x)$: The polynomial value $\Psi_f^{+}(x)$ is divisible by~$p$ if and only if $p$ is unitary for $x$ and $f$ is an odd multiple of~$\ord_p(x^2)$.
Moreover, $\Psi_f^{-}(x)$ is divisible by~$p$  if and only if either $p$ is linear and $\ord_p(x^2)$ divides~$f$, or $p$ is unitary and $f$ is an even multiple (non-zero) of~$\ord_p(x^2)$.
With this in mind we prove a version of Corollary~\ref{coro:psipparts} taking into account linear and unitary primes. %
Two integers are said to have the same \emph{parity} if they are both even or both odd.

\begin{corollary} \label{coro:psieppparts}
	Let $p$ be odd and $m,n\geq 1$ positive integers such that \fbox{$\ord_p(x^2)$} divides both $n$ and $m$, and assume $m_p=n_p$.
	\begin{enumerate}[\rm(i)]
		\item If $p$ is linear for $x$, then the $p$-parts of~$\Psi_m^{-}(x)$ and $\Psi_n^{-}(x)$ as well as the $p$-parts of~$\Psi_m^{+}(x)$ and $\Psi_n^{+}(x)$ are equal.
		\item Assume that $p$ is unitary for $x$.
		      \begin{enumerate}[\rm(a)]
			      \item If $\frac{m}{\ord_p(x^2)}$ and $\frac{n}{\ord_p(x^2)}$ have the same parity, then the same conclusion as in~{\rm(i)} holds.
			      \item  If $\frac{m}{\ord_p(x^2)}$ and $\frac{n}{\ord_p(x^2)}$ have different parity, then the $p$-parts of~$\Psi_m^{-}(x)$ and $\Psi_n^{+}(x)$ as well as the $p$-parts of~$\Psi_m^{+}(x)$ and $\Psi_n^{-}(x)$ are equal.
		      \end{enumerate}
	\end{enumerate}
\end{corollary}

\begin{proof}
	(i) Since~$p$ is linear for $x$, we have $\ord_p(x^2)=\ord_p(x)$ and refer to Corollary~\ref{coro:psipparts}. Note that we even have $(\Psi_m^{-}(x))_p=(\Psi_n^{-}(x))_p\neq 1$ and $(\Psi_m^{+}(x))_p=(\Psi_n^{+}(x))_p= 1$,	by Lemma~\ref{lem:divpsiep}~(i).

	(ii) Since~$p$ is unitary, we have $2\ord_p(x^2)=\ord_p(x)$. %
	By assumption, $\ord_p(x^2)$ divides both $n$ and $m$, so $\ord_p(x)$ divides both $2m$ and $2n$. Using this and $(2n)_p=(2m)_p$ (by assumption), we have $(\Psi_{2n}^{-}(x))_p=(\Psi_{2m}^{-}(x))_p$ by Corollary~\ref{coro:psipparts}. In particular, we have
	\begin{equation*}%
		\begin{aligned}
			(\Psi_n^{-}(x))_p\cdot (\Psi_n^{+}(x))_p
			 & =(\Psi_n^{-}(x) \cdot \Psi_n^{+}(x))_p
			=(\Psi_{2n}^{-}(x))_p                                                                                    \\
			 & =(\Psi_{2m}^{-}(x))_p=(\Psi_m^{-}(x) \cdot \Psi_m^{+}(x))_p=(\Psi_m^{-}(x))_p\cdot (\Psi_m^{+}(x))_p.
		\end{aligned}
	\end{equation*}

	(a) If $\frac{m}{\ord_p(x^2)}$ and $\frac{n}{\ord_p(x^2)}$ have the same parity, then $p$ divides either both $\Psi_m^{-}(x)$ and $\Psi_n^{-}(x)$ or both $\Psi_m^{+}(x)$ and $\Psi_n^{+}(x)$ by Lemma~\ref{lem:divpsiep}~(ii) using that $p$ is unitary. Since~$p$ can not divide both $\Psi_f^{-}(x)$ and $\Psi_f^{+}(x)$ for $f\in \{n,m\}$ (as~$p$ is odd), the above equation yields $(\Psi_m^{-}(x))_p=(\Psi_n^{-}(x))_p$ and $(\Psi_m^{+}(x))_p=(\Psi_n^{+}(x))_p$.

	(b) If $\frac{m}{\ord_p(x^2)}$ and $\frac{n}{\ord_p(x^2)}$ have different parity, then either $p$ divides both $\Psi_m^{-}(x)$ and $\Psi_n^{+}(x)$ or both $\Psi_m^{+}(x)$ and $\Psi_n^{-}(x)$, again by Lemma~\ref{lem:divpsiep}~(ii). %
	By the same argument as in~(a), the above equation yields $(\Psi_m^{-}(x))_p=(\Psi_n^{+}(x))_p$ and $(\Psi_m^{+}(x))_p=(\Psi_n^{-}(x))_p$.
\end{proof}

In the next section we develop the number-theoretical framework used in the proof of Theorem~\ref{thm:A} for classical groups, modifying the $p$-adic setup introduced in Section~\ref{padicsetup} by bringing $\ord_p(x)$ into the picture.

\subsection{Chopped \texorpdfstring{$e_p$-$p$}{e\textpinferior -p}-adic expansion}\label{eppadicsetup}

In this section, fix a positive integer $n\geq 2$ and an integer $x$. For any prime $p$ not dividing~$x$, we set $e_p\coloneqq \ord_p(x)$. Given such a prime $p$, we first define an expansion of~$n$ similar to its $p$-adic expansion but taking $e_p$ into account. The content of this section is very similar to Section \ref{padicsetup} but neither do the results in this section imply the results in Section~\ref{padicsetup} nor the other way around.

\begin{definition}[$e_p$-$p$-adic expansion]\label{def:eppadic} Let $p$ be a prime not dividing $x$.
	Assume that \fbox{$e_p\leq n$}. 
	Write $n=w e_p+s$ with non-negative integers $w\geq 1$ and $0\leq s<e_p<p$.
	Let 
	$$w=\sum_{i=0}^t c_i p^{a_i}=c_0+\sum_{i=1}^t c_i p^{a_i}$$ be the $p$-adic expansion of~$w$,
	where $t\geq 0$ and $0=a_0 < a_1< a_2< \ldots< a_t$ and $0\leq c_0\leq p-1$, and $1\leq c_i\leq p-1$ for all $1\leq i\leq t$.
	Then, the \emph{$e_p$-$p$-adic expansion} of~$n$ is the expansion
	$$n=s+w e_p=s+\sum_{i=0}^t c_i e_p p^{a_i}.$$
	Note that $s$ is the residue of~$n$ modulo $e_p$, and $c_0$ the residue of~$w$ modulo $p$.
\end{definition}

\begin{convention} Remark that, against our previous convention regarding the $p$-adic expansion of~$w=\frac{n-s}{e_p}$ in Section~\ref{padicsetup}, we intentionally include a \textsl{possibly zero} coefficient $c_0$ (being the residue of~$w$ modulo $p$) in the sum and do not write it as a separate summand. %
\end{convention}

\begin{remark}
	Notice that the $e_p$-$p$-adic expansion of~$n$ is usually not the same as its $p$-adic expansion since the coefficients $c_i e_p$ can be bigger or equal to~$p$. %
	However, for $e_p=1$ (e.g., for $p=2$), we recover the $p$-adic expansion of~$n$ and $c_0$ is just the residue of~$n$ modulo $p$. Hence, note that even for $e_p=1$, $s$ is defined differently than in Section~\ref{padicsetup} (for $e_p=1$, we have $s=0$).
\end{remark}

From now on, fix the following notation.
\begin{notation} Let $n\geq 2$ and $x$ an integer and $p$ and $q$ distinct primes not dividing~$x$ such that \fbox{$e_p\leq n$} and \fbox{$e_q\leq n$}.
	Fix the $e_p$-$p$-adic expansion $n=s+\sum_{i=0}^t c_i e_p p^{a_i}$ of~$n$ according to Definition \ref{def:eppadic}. We do \textsl{not} require $t\neq 0$, i.e., $p\leq w=\sum_{i=0}^t c_i p^{a_i}$.\newline
	Write $n=m e_q+r$ with integers $m\geq 1$ %
	and $0\leq r<e_q<q$. Note that $r$ is the residue of~$n$ modulo $e_q$.
\end{notation}

Similarly as in Section~\ref{padicsetup}, now depending on $r$, the residue of~$n$ modulo $e_q$, we define a decomposition of the $e_p$-$p$-adic expansion of~$n$, and examine properties of the parts of~$n$ obtained in this way. Note, however, that the quantities $b$, $T$ defined below now have a different meaning than before. %

\begin{lemma}\label{lem:nepeqquantities}
	Let $n\geq 2$ and assume that $e_p, e_q \leq n$.
	Then, the quantities
	\renewcommand{\arraystretch}{1.5}
	\begin{center}
		\begin{tabular}{|ll|}
			\hline
			$t_0$ & $\coloneqq \min\{i\in\{0,\dots, t\}  \mid  r< c_i e_p p^{a_i}\}$, \\
			$b$   & $\coloneqq s+\sum_{i=0}^{t_0-1} c_i e_p p^{a_i}$,                 \\
			$T$   & $\coloneqq \sum_{i=t_0}^t c_i e_p p^{a_i}$                        \\
			\hline
		\end{tabular}
	\end{center}
	\renewcommand{\arraystretch}{1}
	are well-defined, and we have $n=b+T$ with $T\neq 0$.
	Moreover, the following hold:
	\begin{enumerate}[\rm(i)]
		\item If $k\leq b$ is a positive integer divisible by~$e_p$, then we have $(T+k)_p=k_p$.
		\item If $k\leq r$ is a positive integer divisible by~$e_p$, then we have $(T-k)_p=k_p$.
		\item If \fbox{$r\geq s$ and $r\neq 0$}, then we have $b<2r$. If further $b\neq r$, $T$ is not divisible by~$e_q$. %
	\end{enumerate}
\end{lemma}

\begin{proof} For $0\leq k\leq t$, let $S_k\coloneqq s+e_p(\sum_{i=0}^k c_i p^{a_i})$ be the $k$th partial sum of the $e_p$-$p$-adic expansion of~$n$.
	First assume that $t\neq 0$. For all $0\leq k< t$, the geometric sum formula gives $ \sum_{i=0}^k c_i p^{a_i}\leq \sum_{i=0}^{k} (p-1)p^{a_i}\leq \sum_{i=0}^{a_k} (p-1)p^{i}=p^{a_k+1} -1$.
	This yields
	\begin{align*}\label{eq:partialsum} \tag{a} 
		S_k< e_p + e_p(p^{a_k+1} - 1)= e_pp^{a_k+1}\leq e_p p^{a_{(k+1)}},
	\end{align*}
	hence $S_k<  e_p p^{a_{k+1}}$. %
	In particular, we get 
	$n=S_{t-1}+c_t e_p p^{a_t}< e_p p^{a_t}+c_t e_p p^{a_t}\leq 2 c_t e_p p^{a_t}$.\newline
	If $t=0$, we have $n=s+c_0e_p=s+c_0e_pp^0$ with $c_0\neq 0$ since~$s<e_p\leq n$. %
	Then we have $n=s+c_0e_pp^0<2c_0e_p p^0=2c_t e_p p^{a_t}$.  %
	In total, we get
	\begin{align*}\tag{b} \label{eq:naivebound}
		r<\frac{r+me_q}{2}=\frac{n}{2}< c_t e_p p^{a_t}.
	\end{align*}
	Thus, $t_0\coloneqq \min\{i\in\{0,\dots, t\}  \mid  r< c_i e_p p^{a_i}\}$ is well-defined, and so are \linebreak	$b=s+\sum_{i=0}^{t_0-1} c_i e_p p^{a_i}$ and $T= \sum_{i=t_0}^t c_i e_p p^{a_i}$.
	Notice that $T>r$ by definition, in particular we have $T\neq 0$. %

	By Equation~\eqref{eq:partialsum}, we have the estimate $b=S_{t_0-1}< e_p p^{a_{t_0}}$ setting $S_{-1}\coloneqq s$.
	Remark that the $p$-part of~$T$ is $T_p=p^{a_{t_0}}\geq 1 $ using that $c_{t_0}\neq 0$.

	For the proof of (i) and (ii), let $k\geq 1$ be divisible by~$e_p$. This is equivalent to~$e_p$ dividing the $p'$-part $k_{p'}$ noticing $k=k_p k_{p'}$ and $e_p<p$. %

	(i) 	%
	Let $k\leq b$.
	As $b<e_p p^{a_{t_0}}<p^{a_{t_0}+1}$, the $p$-part of~$k\leq b$ satisfies $k_p\leq p^{a_{t_0}}=T_p$, and since~$e_p$ divides~$k_{p'}$, we obtain $k_p<T_p$. %
	Now, we conclude analogously to Lemma~\ref{lem:npqquantities}~(i).

	(ii) %
	Let $k\leq r$. Note that $T-k>0$, as~$r<T$.
	As~$k\leq r<c_{t_0} e_p p^{a_{t_0}}<e_p p^{a_{t_0} +1}$, by Equation~\eqref{eq:naivebound}, and $k_{p'}$ is divisible by~$e_p$, we have $k_p\leq p^{a_{t_0}}=T_p$.
	With a similar argument as in~(i), $k_p<T_p$ directly implies $(T-k)_p=k_p$.

	If $k_p=T_p=p^{a_{t_0}}$, observe that $p^{a_{t_0}} k_{p'}=k \leq r < c_{t_0} e_p p^{a_{t_0}}$, hence $k_{p'}=de_p$ with $1\leq d<c_{t_0}$ as~$e_p$ divides~$k_{p'}$.
	Note that $T_{p'}$ is of the form $e_p( c_{t_0}+py)$ for some $y\geq 0$. Hence, we get that $T_{p'}-k_{p'}=e_pc_{t_0}-k_{p'}+e_p p y=e_p(c_{t_0}-d+py)$ is not divisible by~$p$ using that $1 \leq d< c_{t_0}<p$ and $e_p<p$. %

	We conclude that $T-k=k_p(T_{p'}-k_{p'})$ has $p$-part $k_p$.

	(iii) Assume that $r\geq s$ and $r\neq 0$.
	Note that for $r=s$, we have $t_0\leq 1$, and we have $t_0=1$ if and only if $c_0=0$. %
	In particular, we get $b=s$ if $t_0=0$ and $b=s+c_0e_p=s$ if $t_0=1$, i.e., in both cases we have $b=s=r$.

	Now, we show that $b<2r$.
	Indeed,
	if $t_0=0$, then we have $b=s\leq r<2r$ using $s\leq r$ and $r\neq 0$. \\
	If $t_0=1$, then we have $c_0 e_p\leq r$ by definition of~$t_0$ and thus $b=s+c_0 e_p<2r$ using that $c_0=0$ if $r=s$, as observed in the previous paragraph.
	For $t_0\geq 2$, we have \linebreak $b=S_{t_0-1}=S_{t_0-2}+c_{t_0-1}e_p p^{a_{t_0-1}}<2c_{t_0-1} e_p p^{a_{t_0-1}}<2r$ using $S_{t_0-2}< e_p p^{a_{t_0-1}}$, by Equation~\eqref{eq:partialsum}.
	 Thus, we have $b<2r$ in all cases.

	This bound also implies that $T=(n-r)+(r-b)=me_q+(r-b)$ is not divisible by~$e_q$ whenever $r\neq b$ taking into account that $r<e_q$.
\end{proof}

Remark that the divisibility condition in Lemma~\ref{lem:nepeqquantities}~(i)--(ii) is crucial, otherwise we would not have enough control on the $p$-part of~$T\pm k$ in the case $k_p=T_p$.

\begin{remark}\label{rem:eppadic}
	In contrast to the decomposition of the $p$-adic expansion of~$n$ defined in Lemma~\ref{lem:npqquantities} (assuming $p\leq n$) we now have a splitting $n=b+T$ with %
	slightly different properties. As before, we have $r<T$ and the $p$-part of~$T$ is $T_p=p^{a_{t_0}}$, but $T_p=1$ is possible since we have no assumption on $p$ relative to~$n$ but only require $e_p\leq n$.
	Since~$e_p$ is involved in the $e_p$-$p$-adic expansion of~$n$, we get that $b\equiv s \bmod e_p$ and the \textsl{additional} feature that $T$ is divisible by~$e_p$.

	Finally, note that, even if $e_p=1$ (e.g., for $p=2$), in general the decomposition $n=b+T$ is different from the one in Lemma~\ref{lem:npqquantities} since the question here is where $r$, the residue of~$n$ modulo $e_q$ (and not $q$), fits into the $1$-$p$-adic (i.e., the $p$-adic) expansion of~$n$.
\end{remark}

Next, we modify the decomposition $n=b+T$ of the $e_p$-$p$-adic decomposition of~$n$, similarly to the process in Section~\ref{padicsetup}. %

\begin{lemma}\label{lem:nepeqquantities2}
	We keep the assumptions of Lemma~\ref{lem:nepeqquantities}. %
	Let $r'\coloneqq e_q+r$, and \fbox{$n\geq 2r'$}, and $t_0$, $b$ be as defined in Lemma~\ref{lem:nepeqquantities}.
	Then, the quantities
	\renewcommand{\arraystretch}{1.5}
	\begin{center}
		\begin{tabular}{|ll|}
			\hline
			$t_0'$ & $\coloneqq \min\{i\in\{0,\dots, t\}  \mid  r'< c_i e_p p^{a_i}\}$, \\
			$b'$   & $\coloneqq s+\sum_{i=0}^{t_0'-1} c_i e_p p^{a_i}$                  \\
			$T'$   & $\coloneqq \sum_{i=t_0'}^t c_i e_p p^{a_i}$                        \\
			\hline
		\end{tabular}
	\end{center}
	\renewcommand{\arraystretch}{1}
	are well-defined,
	and we have $n=b'+T'$ with $T'\neq 0$.
	Moreover, the following hold:
	\begin{enumerate}[\rm(i)]
		\item If $k\leq b'$  is a positive integer divisible by~$e_p$, then we have $(T'+k)_p=k_p$.
		\item If $k\leq r'$  is a positive integer divisible by~$e_p$, then we have $(T'- k)_p=k_p$.
		\item If \fbox{$r\geq s$}, then we have $b'<2r'$.
	\end{enumerate}
	If we assume \fbox{$b\geq e_q$} and \fbox{$r\geq s$}, then we have $ t_0'\in\{t_0, t_0+1\}$ and $b'\neq r'$, and the following hold:
	\begin{enumerate}
		\item[\rm(iv)] If $t_0'=t_0$, then we have $b'=b<r'$. In this case, $T'$ is not divisible by~$e_q$.
		\item[\rm(v)]  If $t_0'=t_0+1$, then we have $r'<b'<e_q+3r$.
	\end{enumerate}
\end{lemma}

\begin{proof}
	Using the assumption $2r'\leq n$ and the estimate $n< 2 c_t e_p p^{a_t}$ (obtained in the proof of Lemma~\ref{lem:nepeqquantities} right after Equation~\eqref{eq:partialsum}), we get that $r'\leq \frac{n}{2} < c_t e_p p^{a_t}$. Thus $t_0'$ is well-defined and so are $b'= s+\sum_{i=0}^{t_0'-1} c_i e_p p^{a_i}$ and $T'=\sum_{i=t_0'}^t c_i e_p p^{a_i}$.
	In particular, we have $n=b'+T'$ with $T'>r'>0$ and $b'\equiv s \equiv b \bmod e_p$.

	(i)--(iii) The proof is analogous to the proof of Lemma~\ref{lem:nepeqquantities}~(i)--(iii) replacing $t_0$, $T$, $b$ by~$t_0'$, $T'$, $b'$, respectively and $r$ by~$r'$ (except in the assumption $r>s$).

	(iv)--(v) The proof goes along the lines of the proof of Lemma~\ref{lem:npqquantities2}~(iv)--(v) substituting $q$ by~$e_q$. Here, in contrast to Section~\ref{padicsetup}, we use the $e_p$-$p$-adic instead of the $p$-adic expansion of~$n$ to define $b$ and $T$. %
	Thus, we occasionally have to take into account the additional factor~$e_p$ whenever terms of the $p$-adic expansion occur in the original proof. Following the proof of Lemma~\ref{lem:npqquantities2}~(iv)--(v),
	we heavily use the assumption $b\geq e_q\geq r\geq s$, first to observe that $r\neq 0$, to see that $t_0'\leq t_0+1$ (using only $b\geq e_q$), and then to make use of the bound $b<2r$ obtained in Lemma~\ref{lem:nepeqquantities}~(iii).
\end{proof}
Unfortunately, in Lemma~\ref{lem:nepeqquantities2}~(v), we can not decide whether $T$ is divisible by~$e_q$ or not; the best bound we have is $0<b'-r'<2r<2 e_q$.

\subsection{Proof of Theorem \texorpdfstring{\ref{thm:A}}{A} in type A}\label{sec:prooftypeA}

In this section, let $G={\GG}_{sc}^F$ with $\GG_{sc}$ a simple simply connected group of type $\type{A}_{n-1}$, for $n\geq 2$, defined over $\overline{ \F_\QQ}$, %
\linebreak that is $G\in\{\SL n \QQ, \SU n \QQ\}$. Writing $\SL n{-\QQ}\coloneqq \SU n \QQ$ and $\PSL n {-\QQ}\coloneqq \PSU n \QQ$, we have $G=\SL  n {\eps \QQ}$ and $G/\zent G=\PSL n {\eps \QQ}$, with $\eps\in \{\pm 1\}$. The aim of this section is to prove Conjecture (NRS) for $G$ and $G/\zent G$, thus proving Theorem~\ref{thm:A} for~$G/\zent G=\PSL n {\eps \QQ}$.

The unipotent characters of $G$ are parametrized by partitions of $n$ where the trivial partition $(n)$ corresponds to the trivial character $\mathbbm{1}_G$ and the partition $(1^n)$ to the Steinberg character, see \cite[Chapter 4.3]{GM}. As for the symmetric group, let $\chi^\lambda$ denote the unipotent character of $G$ corresponding to a partition $\lambda\vdash n$. We say that $\lambda$ has \linebreak \emph{degree} $d^\lambda\coloneqq \chi^\lambda(1)$.
There is a $\QQ$-analog of the hook formula (for $\sym n$) which computes $d^\lambda$, see \cite[Proposition 4.3.1, Proposition 4.3.5]{GM}.\newline
Given a prime $p$ coprime to $\QQ$, henceforth called a \emph{non-defining prime}, let $e_p\coloneqq \ord_p(\eps \QQ)$. By a classical result, two characters $\chi^\lambda$ and $\chi^\mu$ lie in the same $p$-block of $G$ if and only if the $e_p$-cores of the partitions $\lambda, \mu\vdash n$ agree. 
For the original statements of Fong--Srinivasan and Brou\'{e} treating $\GL n {\eps\QQ}$, see \cite[(5D), (7A)]{FStypeA} (for $p\neq 2$) and \linebreak \cite[(3.9) Théorème]{BroueBlocksTypeA} (for any $p$), respectively, and finally combine with \cite[Proposition 11]{CE93} to transfer results from $\GL n {\eps\QQ}$ to $G$.\newline
In particular, the character $\chi^\lambda$ lies in the principal $p$-block of $G$ if and only if the $e_p$-core of $\lambda$ equals $(s)$ given $n=we_p+s$ with $0\leq s<e_p$.

\begin{definition}
	Let $n\geq 2$ and $p$ a non-defining prime. %
	We say that a partition $\lambda\vdash n$ is \emph{$p$-principal} (in type $\type{A}_n$ or $\tw{2}{\type{A}_n}$)
	if the corresponding unipotent character $\chi^\lambda$ lies in the principal $p$-block of~$G$. \newline%
	If $q$ is a prime distinct from $p$, we say that $\lambda$ is \emph{$\{p,q\}$-principal} if $\lambda$ is $p$-principal as well as $q$-principal.
\end{definition}
Observe that any partition is $2$-principal as $e_2=1$. This also follows from \linebreak \cite[Theorem 13]{CE93}.%

We first state two lemmas about almost hook partitions (see Definition~\ref{def:almhooks}) which are basically analogs of the corresponding statements for the symmetric group. To do so, we introduce the following notation.

\begin{notation}
	If $b\coloneqq p^m$ is a prime power and $k\geq 1$ a positive integer, define $k_{b'}\coloneqq k_{p'}=\frac{k}{k_p}$, the \emph{part of~$k$ coprime to $b$}, in short \emph{$b'$-part} of~$k$.
\end{notation}

\begin{notation}
	In this section, let $\Psi_f\coloneqq \Psi_f(\eps\QQ)\coloneqq\Psi_f^{-}(\eps\QQ)=(\eps\QQ)^f-1$ for any $f\geq 1$.
\end{notation}
Notice that the $\QQ'$-part of $\Psi_f$ equals $\Psi_f$.

\begin{lemma}\label{lem:degalmosthookpartitiontypeA2A}
	Let $n\geq 2$. %
	Up to sign, the $\QQ'$-part of the degree of the almost hook partition
	$(1^{n-k-\ell-1},k+1,\ell)$, where  $0\leq k< \ell< n-k\leq n$, %
	is	%
	$$\frac{\prod_{i=1}^{k} \Psi_{n-k+i} \cdot \prod\limits_{\substack{i=1\\ i\neq \ell-k}}^{\ell} \Psi_{n-k-i} }{\prod_{i=1}^{k} \Psi_i \cdot \prod\limits_{\substack{i=1 \\ i\neq \ell-k}}^\ell \Psi_i }= \frac{\prod\limits_{\substack{i=1\\i\neq \ell-k}}^{\ell} \Psi_{n-\ell+i} \cdot \prod_{i=1}^{k} \Psi_{n-\ell-i} }{\prod\limits_{\substack{i=1 \\ i\neq \ell-k}}^\ell \Psi_i \cdot \prod_{i=1}^{k} \Psi_i  }.$$
\end{lemma}

\begin{proof}
	By \cite[Proposition 4.3.1, Proposition 4.3.3]{GM}, %
	the unipotent character $\chi^\mu$ \linebreak
	parametrized by a partition $\mu\vdash n$ has degree $
		d^\mu=E \cdot \QQ^{a} \cdot \frac{\prod_{k=1}^n \Psi_k}{\prod_{h \in \mathcal{H}} \Psi_{h}},
	$ where $E\in\{\pm 1\}$ is a sign, $a\geq 0$ a non-negative integer (both depending on $\mu$), and $h$ runs through the multiset $\mathcal{H}$ of hook lengths of~$\mu$.
	In particular, the $\QQ'$-part of $d^\lambda$ equals $\frac{\prod_{k=1}^n \Psi_k}{\prod_{h\in \mathcal{H}} \Psi_{h}}$.
	Now, the proof is analogous to the proof of Lemma~\ref{lem:degalmosthookpartition}.%
\end{proof}

\begin{lemma} \label{lem:almosthookpartitionpprincipaltypeA} Let $n\geq 2$, and $p$ a non-defining prime. %
	The almost hook partition \linebreak $(1^{n-k-\ell-1},k+1,\ell)$, where $0\leq k<  \ell< n-k\leq n$, is $p$-principal if $n-\ell$ or $n-k$ are divisible by~$e_p$.
\end{lemma}

\begin{proof} The proof is similar to the one of Lemma~\ref{lem:almosthookpartitionpprincipal} ($\sym n$-case) and uses Lemma~\ref{lem:partitionecoreinsteps}, now with $e_p$ playing the role of~$e$.
\end{proof}

Now, we are able to prove Conjecture (NRS) for $G$ using the $e_p$-$p$-adic setup introduced in the previous section (with $x$ chosen as $x=\QQ$). Then, the partitions and cases are constructed the same way as for symmetric groups up to some minor modifications. In contrast to the symmetric groups case, the analysis of partition degrees is a little more involved and requires some basic tools from Section~\ref{sec:degreeanalysis}.

\begin{proposition}\label{pqtypeA}
	Let $G=\SL{n}{\eps \QQ}$ with $n\geq 2$ and $\eps\in\{\pm 1\}$. Let $p$ and $q$ be distinct primes such that $p$ and $q$ do not divide~$\QQ$, and $p$ and $q$ divide~$\abs G$.
	Then, we have $$\irrprime p {\blprinc{p}{G}} \cap \irrprime q {\blprinc{q}{G}}\neq \{\mathbbm{1}_G\}.$$
\end{proposition}

\begin{proof}%
	Let $e_p= \ord_p(\eps\QQ)$ and $e_q=\ord_q(\eps \QQ)$ as defined before. Note that by the order formula (\cite[Table 24.1]{MT}) the condition $pq$ divides the group order $\abs G$ is equivalent to~$e_p\leq n$ and $e_q\leq n$.
	Thus, it is possible to write $n=w e_p+s=m e_q+r$ with positive integers $w,m\geq 1$ and non-negative integers $r,s$ such that $0\leq r<e_q<q$ and $0\leq s<e_p<p$. %
	Let $R=me_q$ and $R'=(m-1)e_q$ (if $m\neq 1$).

	Remember that we have to find a partition $\lambda \neq (n)$ of~$n$ which is $\{p,q\}$-principal such that its degree $d^\lambda$ is coprime to both $p$ and $q$. Then, the partition $\lambda$ gives rise to a non-trivial unipotent character $\chi^\lambda$ of~$G$ lying in $\irrprime p {\blprinc{p}{G}} \cap \irrprime q {\blprinc{q}{G}}$.

	Without loss of generality, we may assume {$r\geq s$} (or interchange $p$ and $q$).
	Further, we can assume that {$r\geq 2$} since for $r\leq 1$ (in particular $s\leq 1$), the partition $(1^n)$ (labeling the Steinberg character) is $\{p,q\}$-principal by Lemma~\ref{lem:almosthookpartitionpprincipaltypeA}. Further, its degree is just a power of the defining characteristic of~$G$, by \cite[Proposition 3.4.10]{GM}, so indeed coprime to the non-defining primes $p$ and $q$.

	Henceforth, we have \fbox{$r\geq \max(2,s)$}. Note that the assumption $r\geq 2$ forces $q\geq 5$ since~$r<e_q<q$.

	We proceed with a case-by-case analysis with cases listed in Table~\ref{tab:casestypeA}, and then give a suitable partition in each of these cases.
	The cases are given in terms of the quantities \linebreak $b$, $T$, $R$, $r$ and $b'$, $T'$, $R'$, $r'$ defined in Section~\ref{eppadicsetup} (taking $x=\QQ$).

	\begin{table}[h] \label{tab:casestypeA}
			\renewcommand{\arraystretch}{1.5}
		\begin{center}
			\begin{tabular}{|p{2cm}l|}
				\hline
				Case I:    & $r=b$.                                    \\
				Case IIa:  & $r>b$.                                    \\
				Case IIb:  & $r<b<e_q$.                                \\
				Case IIbb: & $r<e_q\leq b$ and $q\nmid (m-1)$.         \\
				Case IIIa: & $r<e_q\leq b$, $q\mid (m-1)$ and $r'>b'$. \\
				Case IIIb: & $r<e_q\leq b$, $q\mid (m-1)$ and $r'<b'$. \\
				\hline
			\end{tabular}
		\end{center}
		\renewcommand{\arraystretch}{1}
		\caption{Case distinction for $G=\SL n {\eps \QQ}$ and $r\geq\max(2,s)$.}
	\end{table}

	We first check that the quantities $b'$, $T'$ are defined in Cases~IIIa--b. In these cases we have $m\neq 1$ (since~$n=b+T>e_q+r$ as $T>r$, and $b\geq e_q$ by definition of Cases~IIIa--b), thus $R'=(m-1)e_q\neq 0$. In particular, the condition that $m-1$ is a multiple of $q$ implies that $m\geq 3$, thus $n=me_q+r\geq 3e_q+r> 2(e_q+r)=2r'$. Hence, the assumption of Lemma~\ref{lem:nepeqquantities2} is met, and $b'$ and $T'$ are indeed well-defined in Cases~IIIa--b.\newline\newline
	Moreover, the given cases cover all possible cases since~$b'=r'$ is impossible, again by Lemma~\ref{lem:nepeqquantities2}. 
	Note that in Case I, we have $T=n-b=n-r=me_q=R$, whereas in Case~IIa $T>R$ and in the remaining four cases $T<R$ holds (as well as~$T'>R'$ in Case~IIIa and $T'<R'$ in Case~IIIb).

	Now, we give a partition in each case, see Table \ref{tab:partitionsfortypeA} below. %
	Remark that $$R-b=n-r-b=T-r>0$$ (as~$T>r$), thus we have $R>b$. Similarly, we have $R'>b'$.
	These facts (together with the definition of each case) imply that the partitions are well-defined in each case, and of size $n$.
	\begin{table}[h]
		\begin{center}
			\begin{tabular}{lcl}
				\toprule
				\multirow{2}{*}{Case} & Main condition & \multirow{2}{*}{Partition} \\
				                      & in this case   &                            \\
				\midrule
				I                     & $r=b$          & $(1^{R},r)$                \\
				IIa                   & $r>b$          & $(1^{R-b-1}, b+1,r)$       \\
				IIb--bb               & $r<b$          & $(1^{R-b-1}, r+1,b)$       \\
				IIIa                  & $r'>b'$        & $(1^{R'-b'-1},b'+1,r')$    \\
				IIIb                  & $r'<b'$        & $(1^{R'-b'-1}, r'+1, b')$  \\
				\bottomrule
			\end{tabular}
		\end{center}%
		\caption{Partitions for non-trivial unipotent characters in \linebreak
			$\irrprime p {\blprinc p G} \cap \irrprime q {\blprinc q G}$ where $G=\SL n{\eps \QQ}$ %
			and $\gcd(pq,\QQ)=1$ with $p\neq q$ and $e_p, e_q\leq n$,
			assuming $r\geq \max(2,s)$.} %
		\label{tab:partitionsfortypeA}
	\end{table}

	\noindent Finally, note that the case distinction and partitions are essentially the same as for the symmetric group $\sym n$, refer to proof of Proposition~\ref{pqSn}, but with $q$ replaced by~$e_q$ in the defining inequalities %
	and all quantities, like $m$, $r$, $b$, and $r'$, $b'$, defined slightly different. As before, the partitions are all almost hooks.

	Now, let $\lambda\vdash n$ be one of the partitions given in Table \ref{tab:partitionsfortypeA}, and we claim that it has the desired properties.

	\paragraph*{\textbf{Block analysis}}
	We first explain why the given partition is $\{p,q\}$-principal.
	By definition of~$T$ and $R$, $e_p$ divides~$n-b=T$ and $e_q$ divides~$n-r=R$. Similarly, we have $e_p$ dividing~$n-b'$ and $e_q$ dividing~$n-r'$. Thus, Lemma~\ref{lem:almosthookpartitionpprincipaltypeA} yields that $\lambda$ is $p$- as well as~$q$-principal in each case. In Case~I, we additionally use the condition $r=b$.

	\paragraph*{\textbf{Degree analysis}}
	To determine the $p$- and $q$-part of the degree $d^\lambda$ we use Lemma~\ref{lem:degalmosthookpartitiontypeA2A}.
	Note that, since~$p$ and $q$ are not equal to the defining characteristic of~$G$, we can focus on the part of~$d^\lambda$ coprime to~$\QQ$.
	Table \ref{tab:pmdegtypeA2A} lists the $\QQ'$-part of the degree $d^\lambda$ (up to a possible sign if $\eps=-1$), obtained from Lemma~\ref{lem:degalmosthookpartitiontypeA2A}, in all cases.
	\begin{table}[t]
		\begin{center}
			\begin{tabular}{lccc}
				\toprule
				Case    & \multicolumn{3}{c}{$\pm(d^\lambda)_{\QQ'}$}                                                                                                          \\
				\cmidrule(lr){1-1} \cmidrule(lr){2-4}
				\addlinespace[0.5 em]
				I%
				        & $\frac{\prod_{k=1}^{b-1} \Psi_{T+k} }{\prod\limits_{k=1}^{b-1} \Psi_{k}}$     & $=$ & $\frac{ \prod\limits_{k=1}^{r-1} \Psi_{R+k} }{ \prod\limits_{k=1}^{r-1} \Psi_{k} }$ \\
				\addlinespace[0.5 em]
				IIa     & $\frac{\prod\limits_{k=1}^{b} \Psi_{T+k} \cdot \prod\limits_{\substack{k=1                                                                                  \\ k\neq r-b}}^{r} \Psi_{T-k}}{\prod\limits_{k=1}^{b} \Psi_{k} \cdot \prod\limits_{\substack{k=1 \\ k\neq r-b}}^{r} \Psi_{k}}$ &
				$=$     & $ \frac{\prod\limits_{\substack{k=1                                                                                                                  \\k\neq r-b}}^{r} \Psi_{R+k} \cdot \prod\limits_{k=1}^{b} \Psi_{R-k} }{ \prod\limits_{\substack{k=1\\k\neq r-b}}^{r} \Psi_{k} \cdot \prod\limits_{k=1}^b \Psi_{k}}$\\
				\addlinespace[0.5 em]
				IIb--bb &
				$\frac{\prod\limits_{\substack{k=1                                                                                                                             \\ k\neq b-r}}^{b} \Psi_{T+k} \cdot \prod\limits_{k=1}^{r} \Psi_{T-k}  }{\prod\limits_{\substack{k=1 \\ k\neq b-r}}^{b} \Psi_{k} \cdot \prod\limits_{k=1}^{r} \Psi_{k}}$ &
				$=$     & $\frac{\prod\limits_{k=1}^{r} \Psi_{R+k} \cdot \prod\limits_{\substack{k=1                                                                                  \\ k \neq b-r}}^{b} \Psi_{R-k}}{ \prod\limits_{k=1}^{r} \Psi_{k} \cdot \prod\limits_{\substack{k=1 \\ k\neq b-r}}^b \Psi_{k}}$ \\
				\addlinespace[0.5 em]
				IIIa    & $\frac{\prod\limits_{k=1}^{b'} \Psi_{T'+k}\cdot\prod\limits_{\substack{k=1                                                                                  \\k\neq r'-b'}}^{r'} \Psi_{T'-k} }{\prod\limits_{k=1}^{b'} \Psi_{k} \cdot \prod\limits_{\substack{k=1 \\k\neq r'-b'}}^{r'} \Psi_{k}}$ &
				$=$     & $\frac{\prod\limits_{\substack{k=1                                                                                                                   \\k\neq r'-b'}}^{r'} \Psi_{R'+k}  \cdot \prod\limits_{k=1}^{b'} \Psi_{R'-k} }{\prod\limits_{\substack{k=1 \\k\neq r'-b'}}^{r'} \Psi_{k} \cdot \prod\limits_{k=1}^{b'} \Psi_{k} }
				$                                                                                                                                                              \\
				\addlinespace[0.5 em]
				IIIb    & $\frac{\prod\limits_{\substack{k=1                                                                                                                   \\k\neq b'-r'}}^{b'} \Psi_{T'+k} \cdot \prod\limits_{k=1}^{r'} \Psi_{T'-k} }{\prod\limits_{\substack{k=1 \\k\neq b'-r'}}^{b'} \Psi_{k} \cdot \prod\limits_{k=1}^{r'} \Psi_k} $ &
				$=$     & $\frac{ \prod\limits_{k=1}^{r'} \Psi_{R'+k} \cdot \prod\limits_{\substack{k=1                                                                               \\k\neq b'-r'}}^{b'} \Psi_{R'-k}  }{\prod\limits_{k=1}^{r'} \Psi_{k} \cdot \prod\limits_{\substack{k=1 \\k\neq b'-r'}}^{b'} \Psi_{k}}$ \\
				\addlinespace[0.5 em]
				\bottomrule
			\end{tabular}
		\end{center}
		\caption{$\QQ'$-part $(d^\lambda)_{\QQ'}$ of the degree $d^\lambda$ up to sign.}
		\label{tab:pmdegtypeA2A}
	\end{table}

	First note that, by Lemma~\ref{lem:divpsi}~(i), only factors $\Psi_{\ell}$ with $\ell$ divisible by~$e_p$ (resp. $e_q$) can contribute non-trivially to the $p$-part (resp. $q$-part) of~$\pm (d^\lambda)_{\PP'}$.
	Let $D\in\{T,T',R,R'\}$. Consequently, we only have to compare $p$-parts (resp. $q$-parts) of a factor $\Psi_{D\pm k}$ (in the numerator of~$\pm (d^\lambda)_{\PP'}$) and $\Psi_k$ (in the denominator) for fixed $D\geq 1$ and distinct $k$. If we have $(D\pm k)_p=k_p$ (resp. $(D\pm k)_q=k_q$), Corollary~\ref{coro:psieppparts}~(i) yields that $d^\lambda$ is coprime to~$p$ (resp. $q$). Again, only factors $\Psi_k$ with $k$ divisible by~$e_p$ (resp. $e_q$) matter remembering that $e_p$ divides~$T$ and $T'$ (resp. $q$ divides~$R$ and in Cases~IIIa--b also $R'$).

	\subparagraph*{\textit{$p$-part of the degree~$d^\lambda$}}
	Here, we simply consider the left-hand side expression for $\pm (d^\lambda)_{\PP'}$ given in Table \ref{tab:partitionsfortypeA}, and use the observations~(i) and~(ii) in Lemma~\ref{lem:nepeqquantities} (resp. Lemma~\ref{lem:nepeqquantities2}). %
	By Lemma~\ref{lem:nepeqquantities}~(i)--(ii), we have $(T+k)_p=k_p$, for all $1\leq k\leq b$ divisible by~$e_p$, and $(T-k)_p=k_p$ for all $1\leq k\leq r$ divisible by~$e_p$. This implies that the degree $d^\lambda$ is not divisible by~$p$ in Cases~I--IIbb using Corollary~\ref{coro:psieppparts}.
	In Cases~IIIa--b we use analogous observations in Lemma~\ref{lem:npqquantities2}~(i)--(ii) to conclude that the degree $d^\lambda$ is not divisible by~$p$.

	\subparagraph*{\textit{$q$-part of the degree~$d^\lambda$}}
	For the determination of the $q$-part, consider the right-hand side expression in Table \ref{tab:partitionsfortypeA}. Now, the defining condition of the particular case becomes relevant.
	As said before, we only need to look at indices $k$ divisible by~$e_q$.

	In Cases~I--IIa, there is no such index $k$ %
	since we have $r,b< e_q$, so we are done by Lemma~\ref{lem:divpsi}~(i). %

	In Case~IIbb, we have $r<e_q\leq b$. Recall that we have $b<2r<2e_q$ by Lemma~\ref{lem:nepeqquantities}~(iii). %
	Hence, there is no relevant index $k$ (divisible by~$e_q$) unless $k=b=e_q$. However, $R-b=(m-1)e_q$ and $b=e_q$ both have trivial %
	$q$-part as~$q$ does not divide~$m-1$ by assumption. %
	Thus, Corollary~\ref{coro:psipparts}~(i) yields that the factors $\Psi_{R-b}$ and $\Psi_b$ have the same $q$-part, and we conclude that the degree $d^\lambda$ is coprime to~$q$.

	In Cases~IIIa and IIIb, $R'=(m-1)e_q$ is divisible by~$q$ (since~$m-1$ is by assumption). Thus, the prime $q$ divides an index $k$ if and only if $q$ divides~$R'\pm k$.
	So, by Corollary~\ref{coro:psipparts}~(i), only indices $k$ divisible by ($e_q$ and) $q$ matter.
	Again, there is no relevant index $k$ (divisible by~$qe_q$) unless possibly in Case~IIIb since in Case~IIIa we have $b'<r'=e_q+r<qe_q$ by Lemma~\ref{lem:nepeqquantities2}~(iv).
	In Case~IIIb we have $r'<b'<e_q+3r$ by Lemma~\ref{lem:nepeqquantities2}~(v) and thus $b'<e_q+3r<qe_q$ using $r<e_q<q$ and $q\geq 5$.%
	Hence, there is no index $k\leq b'$ %
	divisible by~$q e_q$, and we conclude that $d^\lambda$ is coprime to~$q$ in all cases.

	All in all, the partition $\lambda\vdash n$ labels a non-trivial unipotent character $\chi^\lambda\in \irr G$ lying in the intersection $\irrprime p {\blprinc{p}{G}} \cap \irrprime q {\blprinc{q}{G}}$.
\end{proof}

We conclude that Theorem~\ref{thm:A} is true for any finite simple group of type $\type{A}_{n-1}$ or $\tw{2}{\type{A}_{n-1}}$.
\begin{corollary}\label{coro:pqtypeA}
	Let $S=\PSL{n}{\eps \QQ}$ with $n\geq 2$ and $\eps\in\{\pm 1\}$. Let $p$ and $q$ be distinct primes such that $p$ and $q$ do not divide~$\QQ$, and $pq$ divides~$\abs S$.
	Then, we have $$\irrprime p {B_p(S)} \cap \irrprime q {B_q(S)}\neq \{\mathbbm{1}_S\}.$$
\end{corollary}

\begin{proof}
	Note that $S=G/\zent G$ with $G=\SL n {\eps \QQ}$ and that the assumption implies that $p$ and $q$ divide $\abs S$.
	By Proposition~\ref{pqtypeA}, we have $\irrprime p {\blprinc p G} \cap \irrprime q {\blprinc q G}\neq \{\mathbbm{1}_G\}$ and the analogous result for $S$ follows by the discussion at the beginning of Section~\ref{sec:classicaltypes}.
	Note that the whole reasoning there does not require the group $S$ to be simple, just of the form $S=G/\zent G$.
\end{proof}
In conclusion, in Corollary~\ref{coro:B}, we reprove \cite[Proposition 3.7]{NRS} for $\PSL n {\eps \QQ}$ and  $2\in \{p,q\}$ and extend this result to odd primes.

\subsection{Proof of Theorem~\texorpdfstring{\ref{thm:A}}{A} in types B and C}\label{sec:prooftypesBC}
Let $G={\GG}_{sc}^F$ where $\GG_{sc}$ is a simple simply connected group defined over $\overline{\F_\QQ}$, of type $\type{B}_n$ or $\type{C}_n$ for some $n\geq 2$, and $G$ is defined over $\F_\QQ$.
The aim of this section is to prove Conjecture (NRS) for $G$ and $G/\zent G$, thus proving Theorem~\ref{thm:A} for $G/\zent G$.

By work of Lusztig \cite[8.2. Theorem]{LusztigSymbols}, the unipotent characters of~$G$ are parame\-trized by so called \emph{Lusztig symbols} (of odd defect and rank $n$) which also encode the character degree and determine if two unipotent characters belong to the same $p$-block for a non-defining prime $p$. Let $e_p\coloneqq \ord_p(\QQ^2)$ and recall the notation 
$$\Psi_f^{\pm}\coloneqq \Psi_f^{\pm}(\QQ)\coloneqq \QQ^f\pm1 \text{ for any }f\geq 1,$$ introduced in Section~\ref{sec:degreeanalysis}.

We first recall the definition of symbols and associated objects, see, e.g., \cite[Section~5]{OlssonCRT} or \cite[Chapter 4.4]{GM}.

A \emph{symbol} $\Lambda={X \choose Y} \coloneqq {{x_1,\dots,x_k}\choose{y_1,\dots,y_l}}$ consists of two ordered integer sequences $X=\{x_1,\dots,x_k\}$ and $Y=\{y_1,\dots,y_l\}$ such that $0\leq x_1<\ldots < x_k$ and $0\leq y_1< \ldots < y_l$.
The \emph{defect} of~$\Lambda$ is $d(\Lambda)\coloneqq \abs{k-l}\geq 0$, while its \emph{rank} is $\mathrm{rk}({\Lambda})\coloneqq \sum_{i=1}^k x_i +\sum_{i=1}^l y_l - \floor{\left(\frac{k+l-1}{2}\right)^2}\geq 0,$
where $\floor{x}\coloneqq \max_{n\in \Z}\{n\leq x\}$ denotes the value of the floor function at a real number $x\in \R$.\newline
We make the following identifications:
We identify $\Lambda={{x_1,\dots,x_k}\choose{y_1,\dots,y_l}}$ with the symbol ${y_1,\dots,y_l} \choose {x_1,\dots,x_k}$. Moreover, if $x_1=y_1=0$, we identify $\Lambda$ with ${{x_2-1,\dots,x_k-1}\choose{y_2-1,\dots,y_l-1}}$. Henceforth, we implicitly work with equivalence classes of symbols instead of symbols themselves. Note that rank and defect are constant on such classes. We usually use the conventions $(x_1,y_1)\neq (0,0)$ and $k\geq l$ for fixing a representative of a class of symbols. If $l=0$, we often write $\Lambda={{x_1,\dots,x_k}\choose\emptyset}$.

Given a positive integer $e\geq 1$ and a symbol $\Lambda={X\choose Y}$, there is a unique symbol associated to $\Lambda$ without so-called \emph{$e$-hooks} and one without \emph{$e$-cohooks} (each of them unique under the conventions made before).%
\newline
An \emph{$e$-hook} is an entry $x\in X$ (or $y\in Y$) such that $x\geq e$ and $x-e\notin X$ (or $y\geq e$ and $y-e\notin Y$). We also call an $e$-hook \emph{a hook of length $e$}.
Removing the $e$-hook $x\in X$ (resp. $y\in Y$) we obtain a new symbol defined by~$(X\setminus\{x\})\cup\{x-e\}$ and $Y$ (resp. $X$ and $(Y\setminus \{y\}) \cup \{y-e\}$).
Successively removing as many as possible $e$-hooks from $\Lambda$ yields a unique symbol without $e$-hooks, %
called the \emph{$e$-core} of~$\Lambda$. Notice that the resulting symbol is independent of the order of removal operations.%
\newline
Similarly, an \emph{$e$-cohook} (or cohook of length $e$) is an entry $x\in X$ (or $y\in Y$) such that $x\geq e$ and $x-e\notin Y$ (or $y\geq e$ and $y-e\notin X$) and removing the $e$-hook $x\in X$ (resp. $y\in Y$) yields the symbol defined by~$(X\setminus\{x\})$ and $Y\cup\{x-e\}$ (resp. $X\cup\{y-e\}$ and $(Y\setminus \{y\})$).
Then, the \emph{$e$-cocore} of~$\Lambda$ is the unique symbol obtained from $\Lambda$ by successive removal of all $e$-cohooks which is again well-defined and unique.

Coming back to our group $G$, the unipotent characters of~$G$ are parametrized by (classes of) symbols of odd defect and rank $n$,  \cite[Theorem 4.4.13]{GM}, where the trivial character~$\mathbbm{1}_G\in \irr G$ corresponds to the symbol~${n\choose \emptyset}$, henceforth called \emph{trivial symbol}. %
For any such symbol $\Lambda$, let $\chi^\Lambda\in \Uch(G)$ denote the corresponding unipotent character. We say that $\Lambda$ has \emph{degree} $d^\Lambda\coloneqq \chi^\Lambda(1)$. The hook and cohook lengths of~$\Lambda$ determine the degree of~$\chi^\Lambda$ (thus $d^\Lambda$), see, e.g., \cite[Proposition 4.4.17]{GM}.\\
Given a non-defining prime~$p$, let $e_p\coloneqq \ord_p(\QQ^2)$. The $p$-block of the character $\chi^\Lambda$ is determined by the symbol $\Lambda$ as follows: 
If $p=2$ (then $\QQ$ is odd), every unipotent character lies in the principal $2$-block by \cite[Theorem 13]{CE93}.\newline
If $p$ is odd, $p$ is either linear for $\QQ$ (i.e., $p$ divides $\Psi_{e_p}^{-}$) or unitary for $\QQ$ (i.e.,~$p$ divides~$\Psi_{e_p}^{+}$). If $p$ is linear (resp. unitary) for $\QQ$, the $e_p$-core (resp. $e_p$-cocore) of~$\Lambda$ decides to which \linebreak $p$-block the character $\chi^{\Lambda}$ belongs.
If $\Lambda'$ is another symbol of odd defect and rank $n$, then the characters $\chi^\Lambda$ and $\chi^{\Lambda'}$ lie in the same $p$-block of~$G$ if and only if their $e_p$-cores (resp. $e_p$-cocores) agree and $p$ is linear (resp. unitary). This follows from \cite[Theorem 4.4]{CE94} together with Brou\'{e}--Malle--Michel's classification of unipotent $e$-cuspidal pairs in \cite{BMM}, for large $p$ see also \cite[5.24 Theorem]{BMM}. %
Originally, for $p\neq 2$ and $\QQ$ being odd, the labeling of all the $p$-blocks of particular ``types'' of groups related to $G$ (i.e., conformal symplectic groups and orthogonal groups in even dimension) is due to Fong--Srinivasan, see \cite[(10B), (12A), (11E), (13B)]{FSclassicalgroups}. Combined with \cite[Proposition 11]{CE93}, their results can be transferred to $G$.\newline
From the above characterization, we see that the character $\chi^{\Lambda}$ lies in the principal~$p$-block of~$G$ if the $e_p$-cores (resp. $e_p$-cocores) of~$\Lambda$ and the trivial symbol agree if $p$ is linear (resp. unitary).

Since the $p$-blocks of unipotent characters are fully determined by the underlying combinatorics, we make the following definition.

\begin{definition} Let $p$ be a non-defining prime for $G$.
	A symbol $\Lambda$ of odd defect (and suitable rank) is called \emph{$p$-principal} if the corresponding unipotent character $\chi^\Lambda\in \irr G$ lies in the principal $p$-block of~$G$. \newline %
	If $q$ is a prime distinct from $p$, we say that $\Lambda$ is \emph{$\{p,q\}$-principal} if $\Lambda$ is $p$- as well as $q$-principal.
\end{definition}

Rephrasing what we said before, any symbol $\Lambda$ of odd defect and rank $n$ is $2$-principal. %
If $p$ is odd and $n=we_p+s$ with $0\leq s<e_p$, the symbol $\Lambda$ is $p$-principal if and only if it has the same $e_p$-core (if $p$ is linear) or $e_p$-cocore (if $p$ is unitary) as the trivial symbol~$n \choose \emptyset$, i.e., $e_p$-core or $e_p$-cocore equal to~${{s}\choose \emptyset}$. For example, the symbol  ${0 , 1, \dots, n}\choose{\phantom{0,}1, \dots, n}$ (labeling the Steinberg character) is $p$-principal if~$s=0$.

As a further preparation to prove Conjecture (NRS) for $G$, we investigate special symbols $X\choose Y$ where $X$ and $Y$ correspond to almost hook partitions introduced in Section~\ref{sec:almhooks} (more precisely to the so-called $\beta$-sets of such partitions).

	\begin{table}[h!]
	\begin{center}
		\begin{tabular}{cccc}
			\toprule
			$\Lambda_1$ & $\Lambda_2$                                                                                                           & $\Lambda_3$ & $\Lambda_4$ \\
			\midrule
			${{0,1, 2, \ldots, n-k-\ell-1, n-k}\choose{\phantom{0, }1, 2, \ldots, n-k-\ell-1, n-\ell}}$
			& ${{0, 1, 2, \ldots, n-k-\ell-1, n-\ell}\choose{\phantom{0, }1, 2, \ldots, n-k-\ell-1, n-k}}$
			& ${{\phantom{0, }1, 2, \ldots, n-k-\ell-1, n-\ell, n-k}\choose{0, 1, 2, \ldots, n-k-\ell-1\phantom{, n-k, n-\ell}}}$
			& ${{0, 1, 2, \ldots, n-k-\ell-1, n-\ell, n-k}\choose{\phantom{0, } 1, 2, \ldots, n-k-\ell-1 \phantom{, n-\ell, n-k}}}$                             \\
			\bottomrule
		\end{tabular}
	\end{center}
	\caption{Symbols $\Lambda_1$, $\Lambda_2$ for $k\leq \ell$, and $\Lambda_3$, $\Lambda_4$ for $k<\ell$.}
	\label{tab:specialsymbols}
\end{table}

For the next lemma, the notion of $\QQ'$-part of an integer is extended to the rationals in the obvious way, i.e., for $a$, $b\in \Z$, the $\QQ'$-part of $\frac{a}{b}$ is $(\frac{a}{b})_{\QQ'}\coloneqq \frac{a_{\QQ'}}{b_{\QQ'}}\in \Q$.
\begin{lemma}\label{lem:degspecialsymbolstypesBC}
	Let $n\geq 2$ and $0\leq k\leq \ell < n$, and $\Lambda\in\{\Lambda_j\,|\, 1\leq j \leq 4\}$ be one of the symbols in Table~\ref{tab:specialsymbols}. Then, the symbol $\Lambda$ has rank $n$ and odd defect.
	Moreover, the $\QQ'$-part of the degree $d^{\Lambda})$ of~$\Lambda$ is of the form $(d^\Lambda)_{\QQ'}=C \cdot E_{\QQ'}$, where %
	$$C=\frac{\prod\limits_{i=1}^{k} \Psi_{2(n-k+i)}^{-} \cdot \prod\limits_{\substack{i=1\\ i\neq \ell-k}}^{\ell} \Psi_{2(n-k-i)}^{-} }{\prod\limits_{i=1}^{k} \Psi_{2i}^{-} \cdot \prod\limits_{\substack{i=1 \\ i\neq \ell-k}}^\ell \Psi_{2i}^{-} }= \frac{\prod\limits_{\substack{i=1\\i\neq \ell-k}}^{\ell} \Psi_{2(n-\ell+i)}^{-} \cdot \prod\limits_{i=1}^{k} \Psi_{2(n-\ell-i)}^{-} }{\prod\limits_{\substack{i=1 \\ i\neq \ell-k}}^\ell \Psi_{2i}^{-} \cdot \prod\limits_{i=1}^{k} \Psi_{2i}^{-} }.$$
	If $k=\ell$ \textup(hence $\Lambda=\Lambda_1=\Lambda_2$\textup), we have $E=1$, and if $k<\ell$, $E$ is given in Table \ref{tab:exdeg}.
\end{lemma}

	\begin{table}[h]
	\begin{center}
		\begin{tabular}{cccc}
			\toprule
			 $\Lambda_1$  & $\Lambda_2$ & $\Lambda_3$ & $\Lambda_4$ \\
			\midrule
			 $\frac{\Psi_{K}^{-} \cdot \Psi_{L}^{+}}{ \Psi_{{\ell-k}}^{-} \cdot 2}$ &
			 $\frac{\Psi_{K}^{+} \cdot \Psi_{L}^{-}}{ \Psi_{{\ell-k}}^{-} \cdot 2}$&
			 $\frac{\Psi_{K}^{+} \cdot \Psi_{L}^{+}}{ \Psi_{{\ell-k}}^{+} \cdot 2}$&
			 $\frac{\Psi_{K}^{-} \cdot \Psi_{L}^{-}}{ \Psi_{{\ell-k}}^{+} \cdot 2}$                                           \\
			\bottomrule
		\end{tabular}
	\end{center}
	\caption{Exceptional part $E$ of the degree $d^{\Lambda_i}$ of symbols $\Lambda_i$, $1\leq i\leq 4$, assuming $k<\ell$.}
	\label{tab:exdeg}
\end{table}

We refer to the rational numbers $C$ and $E$ in the previous lemma as \emph{canonical} and \emph{exceptional} part of the degree $d^\Lambda$. %

\begin{proof}[Proof of Lemma~\ref{lem:degspecialsymbolstypesBC}]
	By direct computation, we see that the rank of~$\Lambda$ is indeed $n$ and its defect $d$ is odd (indeed either $d=1$ or $d=3$).

	If $\Lambda = {{X}\choose{Y}}$ is an arbitrary symbol of rank $n$, the unipotent character (of~$G=\GG^F$ in type $\type{B}_n$ or $\type{C}_n$) labeled by~$\Lambda$ has degree
	\begin{align*}\tag{a} \label{eq:deg}
		d^\Lambda=\QQ^{a} \cdot \frac{1}{2^{b'(\Lambda)}} \cdot \frac{\prod\limits_{i=1}^{n} \Psi_{2i}^{-} }{ \prod\limits_{h\in \mathcal{H}}\Psi_{h}^{-} \prod\limits_{c\in \mathcal{C}} \Psi_{c}^{+}}
		=
		\QQ^a \cdot \frac{\prod\limits_{i=1}^{n} \Psi_{2i}^{-}}{ \prod\limits_{x\in \mathcal{H}\cap  \mathcal{C}}\Psi_{2x}^{-} \cdot 2^{b'(\Lambda)} \cdot \prod\limits_{h\in \mathcal{H}\setminus\mathcal{C}} \Psi_{h}^{-} \cdot \prod\limits_{c\in \mathcal{C}\setminus\mathcal{H}} \Psi_{c}^{+}},
	\end{align*}
	for some non-negative integer $a\geq 0$, and $b'(\Lambda)=\floor{\frac{\abs{X}+\abs{Y}-1}{2}}-\abs{X\cap Y}$ \linebreak (see \cite[Proposition 4.4.17]{GM} together with \cite[Table 1.3]{GM}). Here, the first and the second product in the denominator run over the multisets $\mathcal H$ and $\mathcal{C}$ of hook and cohook lengths of~$\Lambda$, respectively.

	Now, let $\Lambda\in \{\Lambda_j\,|\, 1\leq j\leq 4\}$ and set $L\coloneqq n-\ell$ and $K\coloneqq n-k$. By Equation~\eqref{eq:deg}, the $\QQ'$-part of the degree $d^\Lambda$ is of the form $(d^\Lambda)_{\QQ'}=C\cdot E_{\QQ'}$ with possibly rational 
	\begin{align*}\tag{b} \label{eq:CandE}
		C\coloneqq
		\frac{
			\prod\limits_{\substack{k=1\\k\notin \mathcal{I} } }^n \Psi_{2k}^{-}}{ \prod\limits_{x\in \mathcal{H}\cap\mathcal{C}} \Psi_{2x}^{-} },\,\,\,  E \coloneqq \frac{\prod\limits_{i\in \mathcal{I}} \Psi_{2i}^{-} }{ 2^{b'(\Lambda)} \cdot \prod\limits_{h\in \mathcal{H}\setminus\mathcal{C}} \Psi_{h}^{-} \cdot \prod\limits_{c\in \mathcal{C}\setminus\mathcal{H}} \Psi_{c}^{+}}\in \Q_{\geq 1},
	\end{align*}
	\noindent where $\mathcal{I}=\{K,L\}$ if $k\neq \ell$ and $\mathcal{I}=\emptyset$ otherwise. Note that the $\QQ'$-part of $C$ equals $C$ whereas

	First, assume $k\neq \ell$. Then, we have $b'(\Lambda)=1$ and $\mathcal{H}\neq \mathcal{C}$. Table~\ref{tab:hookcohooklengths} lists the elements of the three multisets $\mathcal{H}\cap \mathcal{C}$, $\mathcal{C}\setminus\mathcal{H}$, $\mathcal{H}\setminus \mathcal{C}$ which in turn determine $\mathcal{H}$ and $\mathcal{C}$ completely. Here, $\widehat{i}$ means that the number $i$ is omitted.
	A sequence $1,\dots, i$ is interpreted as the empty sequence if $i=0$.
	If a multiset has no elements, we display $\emptyset$.
	
	\begin{table}[t]
		\begin{center}
			\begin{tabular}{ccccc}
				\toprule
				                                               & \multicolumn{4}{c}{Symbol}                                                                                      \\
				\cmidrule{2-5}
				                                               & $\Lambda_1$                                                  & $\Lambda_2$   & $\Lambda_3$ & $\Lambda_4$        \\
				\midrule
				\multirow{3}{*}{$\mathcal{H}\cap \mathcal{C}$} & \multicolumn{4}{c}{ $1,\dots,k,$}                                                                               \\
				                                               & \multicolumn{4}{c}{ $1,\dots,\widehat{\ell-k},\dots, \ell;$}                                                    \\
				                                               & \multicolumn{4}{c}{$1,\dots, L-k-1$}                                                                            \\
				\addlinespace[0.5 em]
				$\mathcal{C}\setminus \mathcal{H}$             & $K$                                                          & $L$           & $\ell-k$    & $\ell-k$, $K$, $L$ \\
				\addlinespace[0.5 em]
				$\mathcal{H}\setminus\mathcal{C}$              & $\ell-k$, $L$                                                & $\ell-k$, $K$ & $K$, $L$    &                 $\emptyset$   \\
				\bottomrule
			\end{tabular}
		\end{center}
		\caption{Hook and cohook lengths of symbols $\Lambda_i$, $1\leq i\leq 4$, for $k<\ell$.}
		\label{tab:hookcohooklengths}
	\end{table}
	\noindent Using the data in Table~\ref{tab:hookcohooklengths} and Equation~\eqref{eq:CandE}, we get
	\begin{align*}\tag{c} \label{eq:degC}
		C=\frac{ \prod\limits_{\substack{i=1  \\i \neq K,L }}^n \Psi_{2i}^{-} }{\prod\limits_{x\in \mathcal{H}\cap\mathcal{C}} \Psi_{2x}^{-} }
		=\frac{ \prod\limits_{\substack{i=1   \\i \neq K,L }}^n \Psi_{2i}^{-} }{ \prod\limits_{i=1}^{L-k-1} \Psi_{2i}^{-} \cdot \prod\limits_{i=1}^{k} \Psi_{2i}^{-} \cdot \prod\limits_{\substack{i=1\\i \neq \ell-k }}^\ell \Psi_{2i}^{-} }
		= \frac{\prod\limits_{\substack{i=L-k \\i\neq K, L}}^{n} \Psi_{2i}^{-} }{\prod\limits_{i=1}^{k} \Psi_{2i}^{-} \cdot \prod\limits_{\substack{i=1\\i \neq \ell-k }}^\ell \Psi_{2i}^{-}}.
	\end{align*}
	After splitting the product in the numerator of the rightmost equation in Equation~\eqref{eq:degC} at the omitted number $L$ and using that $n=L+\ell$ (this is the same trick as in Lemma~\ref{lem:degalmosthookpartitiontypeA2A} and Lemma~\ref{lem:degalmosthookpartition}) and suitable index shifts we obtain %
	the desired right-hand side expression for $C$.
	Splitting at the omitted number $K$ instead of~$L$ and taking into account that $L-k=n-\ell-k=K-\ell$, we get the desired left-hand side expression.%
	
	\noindent Now, we determine $E=\frac{ \Psi_{2K}^{-} \Psi_{2L}^{-}}{ 2 \cdot \prod_{h\in \mathcal{H}\setminus\mathcal{C}} \Psi_{h}^{-} \cdot \prod_{c\in \mathcal{C}\setminus\mathcal{H}} \Psi_{c}^{+}}$.
	Using that $\Psi_{2f}^{-}=\Psi_f^{+}\cdot \Psi_{f}^{-}$ for any $f\geq 1$, we obtain the expressions claimed in Table~\ref{tab:exdeg}. This concludes the case $k<\ell$.\\
	
	If $k=\ell$ and $\Lambda=\Lambda_1=\Lambda_2$, we have $b'(\Lambda)=0$ and $$\mathcal{H}=\{1,\dots,\ell, 1,\dots, k, 1,\dots, L-k-1, K\}=\mathcal{C} $$ (using $K=L$). Observe that $\mathcal{H}\cap \mathcal{C}$ is the same multiset as for $k<\ell$, and we have $\mathcal{H}\setminus\mathcal{C}=\emptyset=\mathcal{C}\setminus\mathcal{H}$. Thus, we get $E=1$, and we obtain the same formulae for $C$ as in the previous case.
\end{proof}

Next, we examine whether the symbols defined in Table~\ref{tab:specialsymbols} are $p$-principal for an odd prime $p$.
The next lemma is an analog of Lemma~\ref{lem:almosthookpartitionpprincipal}, now for symbols instead of almost hook partitions, but only for odd primes $p$. The proof follows easily considering the (possibly intertwined) $e_p$-abaci $A_X$, $A_Y$ corresponding to $\beta$-sets $X$, $Y$, defining a symbol~$\Lambda={X \choose Y}$.

\begin{lemma} \label{lem:specialsymbolspprincipaltypeBC} Let $n\geq 2$, and $p$ an odd prime coprime to~$\QQ$, and $\Lambda_j$, $1\leq j\leq 4$ be the symbols defined in Table~\ref{tab:specialsymbols}.
	Assume that $e_p=\ord(\QQ^2)$ divides either $L\coloneqq n-\ell$ or $K\coloneqq n-k$. Then, the following hold:
	\begin{enumerate}[\rm(i)]
		\item If $k=\ell$, then $\Lambda=\Lambda_1=\Lambda_2$ is $p$-principal.
		\item If $p$ is linear for $\QQ$, and $e_p$ divides $L$ \textup(resp. $K$\textup), then $\Lambda_1$ \textup(resp. $\Lambda_2$\textup) is $p$-principal.
		\item If $p$ is unitary for $\QQ$ and $K$ \textup(resp. $L$\textup) is an odd multiple of~$e_p$, then $\Lambda_1$ \textup(resp. $\Lambda_2$\textup) is $p$-principal.
		\item If $p$ is linear for $\QQ$, then $\Lambda_3$ is $p$-principal.
		\item If $p$ is unitary for $\QQ$ and $K$ or $L$ are even \textup(resp. odd\textup) multiples of~$e_p$, then $\Lambda_3$ \textup(resp. $\Lambda_4$\textup) is $p$-principal.

	\end{enumerate}
\end{lemma}

Now, we are able to prove Conjecture (NRS) for $G$ using the $e_p$-$p$-adic setup introduced in Section~\ref{eppadicsetup} (with $x$ chosen as $x=\QQ^2$, not $x=\QQ$ as in type $\type{A}$) and the symbols defined above, in Table~\ref{tab:specialsymbols}. We start with a case distinction similar to the one in the proof of Proposition~\ref{pqtypeA} (type $\type{A}$). In contrast to this case, we also need a subcase distinction. The analysis of unipotent character degrees requires Section~\ref{sec:degreeanalysis}.

\begin{proposition}\label{pqtypeBC}
	Let $\GG$ be a simple simply connected linear algebraic group of type $\type{B}_n$ with $n\geq 3$ or $\type{C}_n$ with $n\geq 2$, defined over $\overline{ \F_\QQ}$, and $F\colon \GG \to \GG$ a Frobenius endomorphism defining an $\F_\QQ$-structure on $\GG$. Let $G\coloneqq {\GG}^F$ and $p$ and $q$ be distinct primes such that $p$ and $q$ do not divide~$\QQ$ and $pq$ divides~$\abs G$. %
	Then, we have $$\irrprime p {B_p(G)} \cap \irrprime q {B_q(G)}\neq\{\mathbbm{1}_G\}.$$
\end{proposition}

\begin{proof}
	Let $e_p\coloneqq \ord_p(\QQ^2)$ and $e_q\coloneqq \ord_q(\QQ^2)$ as in the introductory paragraph of this section. By the degree formula \cite[Table 1.3]{GM} the condition $pq$ divides the group order~$\abs G$ is equivalent to~$e_p\leq n$ and $e_q\leq n$.
	Write $n=\omega e_p+s=m e_q+r$ with $\omega, m\geq 1$ and $0\leq r<e_q<q$ and $0\leq s<e_p<p$. This is possible since~$e_p, e_q \leq n$.

	Using the setup fixed so far, we have to find a symbol $\Lambda \neq {{n}\choose\emptyset}$ of rank $n$ and odd defect (thus parametrizing an unipotent character of~$G$) which is
	$\{p,q\}$-principal such that its degree $d^\Lambda$ is coprime to both $p$ and $q$.

	Without loss of generality, we can assume $r\geq s$ (or interchange $p$ and $q$).
	Moreover, we can assume $r\neq 0$, as for $r=s=0$ the Steinberg character, labeled by the symbol ${0 , 1, \dots, n}\choose{\phantom{0,}1, \dots, n}$, is $\{p,q\}$-principal by the remark after the definition of~$p$-principality. Furthermore, its degree is a power of~$\QQ$, by \cite[Proposition 3.4.10]{GM}, so coprime to both $p$ and $q$ as~$p$ and $q$ are non-defining primes.

	Henceforth, we have \fbox{$r\geq \max(s,1)$}. Note that $r\geq 1$ implies %
	$q\geq 3$ since~$r<e_q<q$. %
	In the following, we distinguish the following main cases given in Table~\ref{tab:casestypeBC}. Here, we use the quantities $b$, $T$, $R$, $r$ and $b'$, $T'$, $R'$, $r'$ defined and examined in Section~\ref{eppadicsetup} (taking $x=\QQ^2$ in the definition).
	\begin{table}[h]
		\begin{center}
	\renewcommand{\arraystretch}{1.5}
	\begin{center}
		\begin{tabular}{|p{2cm}l|}
			\hline
			Case I:    & $r=b$.                                    \\
			Case IIa:  & $r>b$.                                    \\
			Case IIb:  & $r<b<e_q$.                                \\
			Case IIbb: & $r<e_q\leq b$ and $q\nmid (m-1)$.         \\
			Case IIIa: & $r<e_q\leq b$, $q\mid (m-1)$ and $r'>b'$. \\
			Case IIIb: & $r<e_q\leq b$, $q\mid (m-1)$ and $r'<b'$. \\
			\hline
		\end{tabular}
	\end{center}
	\renewcommand{\arraystretch}{1}
	\end{center}
	\caption{Case distinction for a group $G$ of type $\type{B}_n$ or $\type{C}_n$, where $r\geq \max(s,1)$.} \label{tab:casestypeBC}
\end{table}

	With the exact same arguments as in the proof of Proposition~\ref{pqtypeA} (type A) we see that the quantities $b'$ and $T'$ are well-defined in Cases~IIIa--b (as $m\geq 3$). Moreover, the Cases~I--IIIb exhaust all possible cases (since~$b'=r'$ is impossible by Lemma~\ref{lem:nepeqquantities2}).

	Recall that $R\coloneqq m e_q$ and $R'\coloneqq  (m-1) e_q$ (if $b'$ is defined) so that $n=R+r$ and $n=R'+r'$.
	Note that in Case~I, we have $T=n-b=n-r=me_q=R$, in Case~IIa $T$ is strictly bigger than $R$ and in the remaining four cases $T$ is strictly smaller than $R$. Similarly, we see that $T'>R'$ in Case~IIIa and $T'<R'$ in Case~IIIb.

	A further subcase distinction is needed before we give a suitable symbol %
	in each of the out-coming cases. Then we check the block and degree conditions separately.

	To tackle Cases~I--IIbb, we need to distinguish further cases coming from mutually exclusive conditions on each of the pairs $(q,m)$ and $(p, \tilde T )$
	where $\tilde T \coloneqq \frac{T}{e_p}$.
	Here we use the notion of linear and unitary primes recalled in the introductory paragraph of this section.

	The conditions on $q$ and $m$ are as follows:
	Since~$q$ is odd, $q$ is either linear (i.e., $q$ divides~$\Psi_{e_q}^{-}$) or unitary for $Q$ (i.e, $q$ divides~$\Psi_{e_q}^{+}$). If $q$ is unitary for $Q$, we distinguish whether $m$ is even or odd. This yields three mutually exclusive conditions on $(q,m)$ one of which is always fulfilled.

	The conditions on $(p, \tilde T)$ are similar with $\tilde T = \frac{T}{e_p}$ playing the role of~$m=\frac{R}{e_q}$. Either $p=2$ or $p$ is odd, then $p$ is either linear (i.e., $p$ divides~$\Psi_{e_p}^-$) or unitary for $\QQ$ (i.e, $p$ divides~$\Psi_{e_p}^{+}$). In the unitary case, we distinguish whether $\tilde T$ is even or odd, in total yielding four mutually exclusive conditions on $(p, \tilde T)$.
	Grouping the various conditions on both of the pairs $(q,m)$ and $(p, \tilde T )$ together, we define the subcases $\boldsymbol \alpha$, $\boldsymbol \beta$, $\boldsymbol \gamma$, $\boldsymbol \delta$ as indicated in Table~\ref{tab:subcases}.

	In Cases~IIIa--IIIb, a subcase distinction analogously to the one for Cases~I--IIbb is used. Similarly to before, we can pose mutually exclusive conditions on the pair $(q,m-1)$ and similar ones on the pair $(p, \tilde T' )$, where $\tilde T' \coloneqq \frac{T'}{e_p}$, which in turn define the subcases $\boldsymbol \alpha'$, $\boldsymbol \beta'$, $\boldsymbol \gamma'$, $\boldsymbol \delta'$, see Table~\ref{tab:subcasesprime}.
	\begin{table}[h]\small
		\begin{tabular}{cccc}
			\toprule
			                                                                                                &                                           & \multicolumn{2}{c}{\bf Condition on $q$}                                        \\
			\cmidrule(lr){3-4}
			                                                                                                &                                           & $q$ linear for $\QQ$ or                  & $q$ unitary for $\QQ$                \\
			                                                                                                &                                           & $q$ unitary for $\QQ$ and $m$ even       & and $m$ odd                          \\
			\cmidrule{2-4}
			\multirow{4.5}{*}{\rotatebox[origin=c]{90}{{\parbox[c]{1.8cm}{\centering \bf Condition on $p$}}}} & $p=2$ or                                  & \multirow{3}{*}{$\boldsymbol \alpha$}    & \multirow{3}{*}{$\boldsymbol \beta$} \\
			                                                                                                & $p$ linear for $\QQ$ or                   &                                          &                                      \\
			                                                                                                & $p$ unitary for $\QQ$ and $\tilde T$ even &                                          &                                      \\
			\cmidrule{2-4}
			                                                                                                & $p$ unitary for $\QQ$ and $\tilde T$ odd  & $\boldsymbol \gamma$                     & $\boldsymbol \delta$                 \\
			                                                                                                \addlinespace
			\bottomrule
		\end{tabular}
		\caption{Definition of subcases~$\boldsymbol \alpha$, $\boldsymbol \beta$, $\boldsymbol \gamma$, $\boldsymbol \delta$ for Cases~I--IIbb.}
		\label{tab:subcases}
	\end{table}

	\begin{table}[h]\small
		\begin{tabular}{cccc}
			\toprule
			                                                                                                &                                            & \multicolumn{2}{c}{\bf Condition on $q$}                                         \\
			\cmidrule{3-4}
			                                                                                                &                                            & $q$ linear for $\QQ$ or                  & $q$ unitary for $\QQ$                 \\
			                                                                                                &                                            & $q$ unitary for $\QQ$ and $m$ odd        & and $m$ even                          \\
			\cmidrule{2-4}
			\multirow{4.5}{*}{\rotatebox[origin=c]{90}{{\parbox[c]{1.8cm}{\centering \bf Condition on $p$}}}} & $p=2$ or                                   & \multirow{3}{*}{$\boldsymbol \alpha'$}   & \multirow{3}{*}{$\boldsymbol \beta'$} \\
			                                                                                                & $p$ linear for $\QQ$ or                    &                                          &                                       \\
			                                                                                                & $p$ unitary for $\QQ$ and $\tilde T'$ even &                                          &                                       \\
			\cmidrule{2-4}
			                                                                                                & $p$ unitary for $\QQ$ and $\tilde T'$ odd  & $\boldsymbol \gamma'$                    & $\boldsymbol \delta'$                 \\
			                                                                                                \addlinespace
			\bottomrule
		\end{tabular}
		\caption{Definition of subcases $\boldsymbol \alpha'$, $\boldsymbol \beta'$, $\boldsymbol\gamma'$, $\boldsymbol \delta'$ for Cases IIIa--IIIb.}
		\label{tab:subcasesprime}
	\end{table}

	In Cases~I--IIbb, for each of the (sub)cases just defined, Table~\ref{tab:typeBCCasesI-IIbb} lists a symbol $\Lambda$, and we claim that it has the desired properties. %
	Remark that we have $R-b=n-r-b=T-r>0$ (as~$T>r$), thus $R>b$ and $R-b\leq T$. This implies, together with the condition on $R$ and $T$ imposed in Cases~IIa--b, that the symbols in Table~\ref{tab:typeBCCasesI-IIbb} are all well-defined.

	\begin{table}[h!]\small
		\begin{center}
			\begin{tabular}{ccccc}
				\toprule
				     \multirow{2.5}{*}{Case}& \multicolumn{4}{c}{Subcase}                                                                                                                                       \\
				\cmidrule(lr){2-5}
				 & $\boldsymbol \alpha$                                                                          & $\boldsymbol \beta$ & $\boldsymbol \gamma$ & $\boldsymbol \delta$ \\
				\midrule \addlinespace
				I
				     & \multicolumn{4}{c}{${0, 1, 2, \ldots, R-b-1, R}\choose{\phantom{0,} 1, 2, \ldots, R-b-1, R}$}                                                                     \\ \addlinespace \addlinespace
				IIa
				     & ${\phantom{0,}1, 2, \ldots, R-b-1, R, T}\choose{0, 1, 2, \ldots, R-b-1\phantom{, R, T}}$
				     & ${0,1, 2, \ldots, R-b-1, R}\choose{\phantom{0,}1, 2, \ldots, R-b-1, T}$
				     & ${0, 1, 2, \ldots, R-b-1, T}\choose{\phantom{0,}1, 2, \ldots, R-b-1, R}$
				     & ${0, 1, 2, \ldots, R-b-1, R, T}\choose{\phantom{0,} 1, 2, \ldots, R-b-1 \phantom{, R, T}}$                                                                        \\ \addlinespace \addlinespace
				IIb--bb
				     & {${\phantom{0,} 1, 2, \ldots, R-b-1, T, R}\choose{0, 1, 2, \ldots, R-b-1 \phantom{, T, R}}$}
				     & {${0,1, 2, \ldots, R-b-1, R}\choose{\phantom{0, }1, 2, \ldots, R-b-1, T}$}
				     & {${0, 1, 2, \ldots, R-b-1, T}\choose{\phantom{0, }1, 2, \ldots, R-b-1, R}$ }
				     & { ${0, 1, 2, \ldots, R-b-1, T, R}\choose{\phantom{0,} 1, 2, \ldots, R-b-1\phantom{, T, R}}$}                                                                      \\ \addlinespace
				\bottomrule
			\end{tabular}
		\end{center}
		\caption{Symbols for non-trivial unipotent characters in \linebreak $\irrprime p {\blprinc{p}{G}} \cap \irrprime q {\blprinc{q}{G}}$ for a group $G$ of type $\type{B}_n$ or $\type{C}_n$, $n\geq 2$, Cases~I--IIbb.}
		\label{tab:typeBCCasesI-IIbb}
	\end{table}

	Note that, whenever $r = s = 0$ (then $b=r$), the character in Case~I which would correspond to the symbol given in Table~\ref{tab:typeBCCasesI-IIbb} is the Steinberg character, and we recover the character chosen before.

	In Cases~IIIa--b, we consider similar symbols as in Table \ref{tab:typeBCCasesI-IIbb}, see Table~\ref{tab:typeBCCasesIII}. Note that these symbols can be obtained from Cases~IIa--b substituting $b$, $T$, $R$, $r$ by $b'$, $T'$, $R'$, $r'$ and the subcases $\boldsymbol \alpha$, $\boldsymbol \beta$, $\boldsymbol \gamma$, $\boldsymbol \delta$ by their pendants $\boldsymbol \alpha'$, $\boldsymbol \beta'$, $\boldsymbol\gamma'$, $\boldsymbol \delta'$.
	Analogously to Cases~IIa--b, we see that the symbols in Table~\ref{tab:typeBCCasesIII} are well-defined, first observing $R'>b'$ and then using $R'-b'=T'-r'<T'$ together with the condition on $R'$ and $T'$ imposed in Cases~IIIa--b.

	\begin{table}[h!]\small
		\begin{center}
			\begin{tabular}{ccccc}
				\toprule
				      \multirow{2.5}{*}{Case}&  \multicolumn{4}{c}{Subcase}                                                                                                                                               \\
				\cmidrule(lr){2-5}
				& $\boldsymbol \alpha'$                                                                              & $\boldsymbol \beta'$ & $\boldsymbol \gamma'$ & $\boldsymbol \delta'$ \\
				\midrule \addlinespace
				IIIa
				     & ${\phantom{0, }1, 2, \ldots, R'-b'-1, R', T'}\choose{0, 1, 2, \ldots, R'-b'-1 \phantom{, R', T'}}$
				     & ${0,1, 2, \ldots, R'-b'-1, R'}\choose{\phantom{0,}1, 2, \ldots, R'-b'-1, T'}$
				     & ${0, 1, 2, \ldots, R'-b'-1, T'}\choose{\phantom{0,}1, 2, \ldots, R'-b'-1, R'}$
				     & ${0, 1, 2, \ldots, R'-b'-1, R', T'}\choose{\phantom{0,}1, 2, \ldots, R'-b'-1\phantom{, R', T'}}$                                                                          \\ \addlinespace \addlinespace
				IIIb
				     & ${\phantom{0, }1, 2, \ldots, R'-b'-1, T', R'}\choose{0, 1, 2, \ldots, R'-b'-1 \phantom{, T', R'}}$
				     & ${0,1, 2, \ldots, R'-b'-1, R'}\choose{\phantom{0,}1, 2, \ldots, R'-b'-1, T'}$
				     & ${0, 1, 2, \ldots, R'-b'-1, T'}\choose{\phantom{0,}1, 2, \ldots, R'-b'-1, R'}$
				     & ${0, 1, 2, \ldots, R'-b'-1, T', R'}\choose{\phantom{0,}1, 2, \ldots, R'-b'-1\phantom{, T', R'}}$                                                                          \\ \addlinespace
				\bottomrule
			\end{tabular}
		\end{center}
		\caption{Symbols for non-trivial unipotent characters in \linebreak $\irrprime p {B_p(G)} \cap \irrprime q {B_q(G)}$ for a group $G$ of type $\type{B}_n$ or $\type{C}_n$, $n\geq 2$, Cases~IIIa--b.}
		\label{tab:typeBCCasesIII}
	\end{table}

Note that the symbols in Table~\ref{tab:typeBCCasesI-IIbb} and Table~\ref{tab:typeBCCasesIII} are of the same form as the symbols considered in Lemma 
	~\ref{lem:degspecialsymbolstypesBC} and Lemma~\ref{lem:specialsymbolspprincipaltypeBC}, see also Table~\ref{tab:specialsymbols}.
	Now, let $\Lambda$ be one of the symbols in Table~\ref{tab:typeBCCasesI-IIbb} or Table~\ref{tab:typeBCCasesIII}.
	By Lemma~\ref{lem:degspecialsymbolstypesBC}, we get that the rank of~$\Lambda$ is indeed $n$ and its defect $d$ is odd. (In Cases~I--IIbb take $\{k,\ell\}=\{r,b\}$ such that $k\leq \ell$ and identify $\Lambda$ with the corresponding symbol in Table \ref{tab:specialsymbols}; in Cases~IIIa--b do the same with $\{k,\ell\}=\{r',b'\}$.) 
	Next, we first check that $\Lambda$ fulfills the block conditions and then examine its degree.

	\paragraph*{\textbf{Block analysis}} Since every symbol is $2$-principal (see the remark on $p$-principal symbols at the beginning of this section), we can assume that $p$ is odd.
	By definition of~$T$ and $T'$ (resp. $R$ and $R'$), $e_p$ divides~$n-b=T$ and $n-b'=T'$ (resp. $e_q$ divides~$n-r=R$ and $n-r'=R'$). Thus, Lemma~\ref{lem:specialsymbolspprincipaltypeBC} applies. Taking into account whether $p$ (resp. $q$) is linear or unitary as well as the prescribed parity of~$\frac{T}{e_p}$ and $\frac{T'}{e_p}$ (resp. $\frac{R}{e_q}$ and $\frac{R'}{e_q}$) in each subcase, we see that $\Lambda$ is $p$-principal (resp.~$q$-principal) in each case.

	In total, $\Lambda$ is always $\{p,q\}$-principal and we turn our attention to its degree $d^\Lambda$.

	\paragraph*{\textbf{Degree analysis}} Since $p$ and $q$ are non-defining (i.e., do not divide $\QQ$), we are only interested in the part of the degree $d^\Lambda$ coprime to $\QQ$, denoted $(d^\Lambda)_{\QQ'}$. By Lemma \ref{lem:degspecialsymbolstypesBC}, we have a factorization $(d^\Lambda)_{\QQ'}=C\cdot E_{\QQ'}$ into a canonical part $C\in \Q_{\geq 1}$ and an exceptional part $E\in \Q_{\geq 1}$ which we analyze separately.

	\subparagraph*{\textit{Canonical part $C$ of the degree $d^\Lambda$}}
	We first list the canonical part $C$ of the degree $d^\Lambda$ (written in two equivalent ways) obtained from the previous lemma, see Table~\ref{tab:candegtypeB}. For brevity, we use the notation $\Psi_k\coloneqq \Psi_k^{-}$ for any $k\geq 1$, but only in this paragraph. Note that $C$ does not depend on the subcases.
	The strategy to analyze $C$ is the same as for the degree expressions occurring in the proof of Proposition~\ref{pqtypeA}.
	\begin{table}[t]
		\begin{center}
			\begin{tabular}{lccc}
				\toprule
				Case    & \multicolumn{3}{c}{$C$}                                                                                                                                                                                                                                                       \\
				\cmidrule(lr){1-1} \cmidrule(lr){2-4}
				\addlinespace[0.5 em]
				I%
				        & $\frac{\prod\limits_{k=1}^{b} \Psi_{2(T+k)} \cdot \prod\limits_{k=1}^{r} \Psi_{2(T-k)}}{\prod\limits_{k=1}^{b} \Psi_{2k} \cdot \prod\limits_{k=1}^{r} \Psi_{2k}}$ & $=$ & $\frac{ \prod\limits_{k=1}^{r} \Psi_{2(R+k)} \cdot \prod\limits_{k=1}^{b} \Psi_{2(R-k)} }{ \prod\limits_{k=1}^{r} \Psi_{2k} \cdot \prod\limits_{k=1}^b \Psi_{2k}}$ \\
				\addlinespace[0.5 em]
				IIa     & $\frac{\prod\limits_{k=1}^{b} \Psi_{2(T+k)} \cdot \prod\limits_{\substack{k=1                                                                                                                                                                                                               \\ k\neq r-b}}^{r} \Psi_{2(T-k)}}{\prod\limits_{k=1}^{b} \Psi_{2k} \cdot \prod\limits_{\substack{k=1 \\ k\neq r-b}}^{r} \Psi_{2k}}$ &
				$=$     & $ \frac{\prod\limits_{\substack{k=1                                                                                                                                                                                                                                                  \\k\neq r-b}}^{r} \Psi_{2(R+k)} \cdot \prod\limits_{k=1}^{b} \Psi_{2(R-k)} }{ \prod\limits_{\substack{k=1\\k\neq r-b}}^{r} \Psi_{2k} \cdot \prod\limits_{k=1}^b \Psi_{2k}}$\\
				\addlinespace[0.5 em]
				IIb--bb &
				$\frac{\prod\limits_{\substack{k=1                                                                                                                                                                                                                                                             \\ k\neq b-r}}^{b} \Psi_{2(T+k)} \cdot \prod\limits_{k=1}^{r} \Psi_{2(T-k)}  }{\prod\limits_{\substack{k=1 \\ k\neq b-r}}^{b} \Psi_{2k} \cdot \prod\limits_{k=1}^{r} \Psi_{2k}}$ &
				$=$     & $\frac{\prod\limits_{k=1}^{r} \Psi_{2(R+k)} \cdot \prod\limits_{\substack{k=1                                                                                                                                                                                                               \\ k \neq b-r}}^{b} \Psi_{2(R-k) }}{ \prod\limits_{k=1}^{r} \Psi_{2k} \cdot \prod\limits_{\substack{k=1 \\ k\neq b-r}}^b \Psi_{2k}}$ \\
				\addlinespace[0.5 em]
				IIIa    & $\frac{\prod\limits_{k=1}^{b'} \Psi_{2(T'+k)}\cdot\prod\limits_{\substack{k=1                                                                                                                                                                                                               \\k\neq r'-b'}}^{r'} \Psi_{2{(T'-k)}} }{\prod\limits_{k=1}^{b'} \Psi_{2k} \cdot \prod\limits_{\substack{k=1 \\k\neq r'-b'}}^{r'} \Psi_{2k}}$ &
				$=$     & $\frac{\prod\limits_{\substack{k=1                                                                                                                                                                                                                                                   \\k\neq r'-b'}}^{r'} \Psi_{2{(R'+k)}}  \cdot \prod\limits_{k=1}^{b'} \Psi_{2{(R'-k)}} }{\prod\limits_{\substack{k=1 \\k\neq r'-b'}}^{r'} \Psi_{2k} \cdot \prod\limits_{k=1}^{b'} \Psi_{2k} }
				$                                                                                                                                                                                                                                                                                              \\
				\addlinespace[0.5 em]
				IIIb    & $\frac{\prod\limits_{\substack{k=1                                                                                                                                                                                                                                                   \\k\neq b'-r'}}^{b'} \Psi_{2{(T'+k)}} \cdot \prod\limits_{k=1}^{r'} \Psi_{2{(T'-k)}} }{\prod\limits_{\substack{k=1 \\k\neq b'-r'}}^{b'} \Psi_{2k} \cdot \prod\limits_{k=1}^{r'} \Psi_k} $ &
				$=$     & $\frac{ \prod\limits_{k=1}^{r'} \Psi_{2{(R'+k)}} \cdot \prod\limits_{\substack{k=1                                                                                                                                                                                                          \\k\neq b'-r'}}^{b'} \Psi_{2{(R'-k)}}  }{\prod\limits_{k=1}^{r'} \Psi_{2k} \cdot \prod\limits_{\substack{k=1 \\k\neq b'-r'}}^{b'} \Psi_{2k}}$ \\
				\addlinespace[0.5 em]
				\bottomrule
			\end{tabular}
		\end{center}
		\caption{Canonical part $C$ of the degree $d^\Lambda$.}
		\label{tab:candegtypeB}
	\end{table}

	\subparagraph*{\textit{$p$-part of~$C$}} Here, we examine the $p$-part of~$C$ using the left-hand side expression in Table~\ref{tab:candegtypeB}, e.g., $C=\frac{\prod_{k=1}^{b} \Psi_{2(T+k)} \cdot \prod_{k=1}^{r} \Psi_{2(T-k)}}{\prod_{k=1}^{b} \Psi_{2k} \cdot \prod_{k=1}^{r} \Psi_{2k}}$ in Case~I.%

	First note that, for any $k\geq 1$, the prime $p$ divides~$\Psi_{2k}$ if and only if $e_p=\ord_p(\QQ^2)$ divides~$k$ by Lemma~\ref{lem:divpsi}~(ii). (Equivalently, $\ord_p(\QQ)$ divides~$2k$ as~$\ord_p(\QQ^2)=\frac{\ord_p(\QQ)}{\gcd(2,\ord_p(\QQ))}$.)
	Thus, only factors $\Psi_{2k}$ with $e_p$ diving $k$ can contribute non-trivially to the $p$-part of~$C$. We call such an index $k$ a \textit{relevant index} for the $p$-part of~$C$. Also notice that $e_p$ divides both $T$ and $T'$. In particular, $e_p$ divides~$k$ if and only if $e_p$ divides~$T\pm k$ (resp. $e_p$ divides $T'\pm k$).  So, the strategy is to compare $p$-parts of pairs of factors $\Psi_{2(T\pm k)}$ and $\Psi_{2k}$ (resp. $\Psi_{2(T'\pm k)}$ and $\Psi_{2k}$) in Cases I--IIbb (resp. Cases IIIa--b) for the relevant indices $k$.
	If we have equality of~$p$-parts $(T\pm k)_p=k_p$ (in particular $(2(T\pm k))_p=(2k)_p$) for any relevant~$k$, Corollary~\ref{coro:psipparts}~(i) yields that $\Psi_{2(T\pm k)}$ and $\Psi_{2k}$ have the same $p$-part (here we use again that $\ord_p(\QQ)$ divides~$2k$) and we are done in Cases~I--IIbb. In Cases~IIIa--b, the same argument can be used interchangeably with $T'$ in place of~$T$.

	Now, turn to Cases~I--IIbb. For the relevant indices $k$ (in the given range of~$k$ according to the degree formulas in Table~\ref{tab:candegtypeB}), Lemma~\ref{lem:nepeqquantities}~(i)--(ii) yields that $(T\pm k)_p=k_p$,
	and we conclude that the canonical part $C$ is coprime to $p$ by what we said in the previous paragraph.
	In Cases~IIIa--b, the exact same reasoning works replacing $T$ by~$T'$ and using Lemma~\ref{lem:nepeqquantities2}~(i)--(ii) instead of Lemma~\ref{lem:nepeqquantities}~(i)--(ii).

	\subparagraph*{\textit{$q$-part of~$C$}}
	The analysis of the $q$-part of~$C$ works essentially the same as for the $p$-part. Here, we use the right-hand side expression for $C$, e.g., $C=\frac{ \prod_{k=1}^{r} \Psi_{2(R+k)} \cdot \prod_{k=1}^{b} \Psi_{2(R-k)} }{ \prod_{k=1}^{r} \Psi_{2k} \cdot \prod_{k=1}^b \Psi_{2k}}$ in Case~I. %
	A similar strategy as for the analysis of the $p$-part of~$C$ can be phrased for the prime $q$ replacing $T$ and $T'$ by~$R$ and $R'$, respectively and noticing that $q$ divides $R$ as well as $R'$ (by definition of Cases~IIIa--b).
	Again, an index $k$ in the degree formulas in Table~\ref{tab:candegtypeB} with $e_q$ dividing~$k$ is called \emph{relevant}. \\
	If we have equality of~$q$-parts $(R\pm k)_p=k_p$ (in particular $(2(R\pm k))_q=(2k)_p$) for any relevant index $k$ in Cases~I--IIbb, Corollary~\ref{coro:psipparts}~(i) yields that $\Psi_{2(R\pm k)}$ and $\Psi_{2k}$ have the same $q$-part (here we use again that $\ord_q(\QQ)$ divides~$2k$). Then, we conclude that $C$ is coprime to $q$. In Cases~IIIa--b, the same argument is used interchangeably with $R'$ in place of~$R$.

	To compare the $q$-parts of~$R\pm k$ and $k$ (or $R'\pm k$ and $k$) for the relevant indices $k$ %
	we take into account the defining relation (in terms of~$r$, $b$ and $e_q$) of each particular case separately. If we talk about indices we implicitly mean indices in the range of~$k$ according to the expressions in Table~\ref{tab:candegtypeB}.

	In Cases~I--IIb, there is no relevant index $k$ divisible by~$e_q$ since~$r,b< e_q$, so we are done. In Case~IIbb, there is no relevant index unless $b=e_q$. However, we then have $R-b=R-e_q=(m-1)e_q$, and $R-b$ and $b$ have trivial $q$-part (as~$q$ does not divide~$m-1$), so we are done. %

	We are left with Cases~IIIa--b. Since~$q$ divides~$m-1$ by assumption, we get that $q$ divides~$R'=(m-1)e_q$. In particular, $q$ divides an integer $k$ if and only if $q$ divides~$R'\pm k$. So, Corollary~\ref{coro:psipparts}~(i) tells us that only indices $k$ divisible by~$q$ (and additionally by~$e_q$ by Lemma~\ref{lem:divpsi}~(ii)) matter for the degree analysis. In Case~IIIa, there is no index $k$ divisible by~$qe_q$ as~$b'<r'=e_q+r<q e_q$. In Case~IIIb, we have $r'<b'<e_q+3r$ by Lemma~\ref{lem:nepeqquantities2}~(v) and thus $b'<e_q+3r<qe_q$ since~$r<e_q<q$ and $q\geq 3$.
	Hence, there is no index $k$ divisible by~$q e_q$, and we conclude that $C$ is coprime to $q$ in all cases.

	We continue with the analysis of the exceptional part $E$ of~$d^\Lambda$. We again use the notation $\Psi^{-}_k$ and $\Psi^{+}_k$, for $k\geq 1$, and no longer write $\Psi_k=\Psi^{-}_k$.

	\subparagraph*{\textit{Exceptional part $E$ of the degree $d^\Lambda$}}
	As for the canonical part, we use Lemma \ref{lem:degspecialsymbolstypesBC} to determine the exceptional part $E$. Note that $E=1$ in Case~I, so we are done in this case.
	In Cases~IIa--bb, $E$ is given as in Table~\ref{tab:exdegtypesBCCasesIIabb}. %
	Observe that the value of~$E$ depends on the subcase only.

	\begin{table}[h!]
		\begin{center}
			\begin{tabular}{ccccc}
				\toprule
				    & $\boldsymbol \alpha$                                                    & $\boldsymbol \beta$ & $\boldsymbol \gamma$ & $\boldsymbol \delta$ \\
				\midrule
				$E$ & $\frac{\Psi_{R}^{+} \cdot \Psi_{T}^{+}}{ \Psi_{\abs{r-b}}^{+} \cdot 2}$
				    & $\frac{\Psi_{R}^{-} \cdot \Psi_{T}^{+}}{ \Psi_{\abs{r-b}}^{-} \cdot 2}$
				    & $\frac{\Psi_{R}^{+} \cdot \Psi_{T}^{-}}{ \Psi_{\abs{r-b}}^{-} \cdot 2}$
				    & $\frac{\Psi_{R}^{-} \cdot \Psi_{T}^{-}}{ \Psi_{\abs{r-b}}^{+} \cdot 2}$                                                                     \\
				\bottomrule
			\end{tabular}
		\end{center}
		\caption{Exceptional part $E$ of the degree $d^\Lambda$ in subcases $\boldsymbol \alpha$, $\boldsymbol \beta$, $\boldsymbol \gamma$, $\boldsymbol \delta$ of Cases~IIa--bb.}
		\label{tab:exdegtypesBCCasesIIabb}
	\end{table}

	In Cases~IIIa--b we get expressions similar to Cases~IIa--bb (depending only on the subcases $\boldsymbol \alpha'$, $\boldsymbol \beta'$, $\boldsymbol \gamma'$, $\boldsymbol \delta'$) by replacing $b$, $T$, $R$, $r$ by $b'$, $T'$, $R'$, $r'$.  %
	Now we analyze the $p$- and the $q$-part of~$E$ in the remaining Cases~IIa--IIIb.

	\subparagraph*{\textit{$p$-part of~$E$}}
	We first focus on Cases~IIa--bb. Note that $E$ involves just expressions $\Psi_k^{\pm}$ with $k\in\{R,T, \abs{r-b}\}$. By Lemma~\ref{lem:divpsi}~(ii), %
	only factors $\Psi_k^{\pm}$ with $k$ divisible by~$e_p$ can contribute non-trivially to the $p$-part of~$E$. Henceforth, unless said or contradicted otherwise, we assume that  %
	 $k\in\{R,T, \abs{r-b}\}$ is divisible by $e_p$.
	For $T$ this is always the case. %
	Note that $$R=me_q=n-r=(n-b)+(b-r)=T+(b-r)=T\pm \abs{r-b}$$ and either $\abs{r-b}=r-b<r$ if $r> b$ (i.e., in Case~IIa) or $\abs{r-b}=b-r<b$ if $r< b$ (i.e., in Cases~IIb--bb). Thus, Lemma~\ref{lem:nepeqquantities}~(i) (in Case~IIa) or~(ii) (in Cases~IIab-bb)  yield equality of~$p$-parts $R_p=(\abs{r-b})_p$. %

	First focus on the subcases $\boldsymbol{\alpha}$ and $\boldsymbol{\beta}$, i.e., $E=\frac{\Psi_{T}^{+}}{2} \cdot \frac{\Psi_{R}^{\pm}}{\Psi_{\abs{b-r}}^{\pm}}$. By definition of~$\boldsymbol{\alpha}$ and $\boldsymbol{\beta}$, we either have $p=2$, or $p$ is linear, or $p$ is unitary for $\QQ$ and $\tilde T$ is even.

	If $p=2$, then $T$ is even (as~$T_2=2^{a_{t_0}}$ with $t_0>0$ using $r>0$), %
	thus $$\Psi_{T}^{+}=\QQ^{T}+1=(\QQ^2)^{\frac{T}{2}}+1\equiv 2 \bmod 4$$ and $\Psi_{T}^{+}$ has $2$-part $(\Psi_{T}^{+})_2=2$ canceling with the factor $2$ in the denominator. Moreover, $R_2=(\abs{r-b})_2$ (as shown above) together with Corollary~\ref{coro:psipparts} %
	yield equality of~$2$-parts $(\Psi_{R}^{\pm})_2=(\Psi_{\abs{r-b}}^{\pm})_2$, and we conclude that $E$ is odd.

	If $p$ is linear for $\QQ$ (in particular $p$ is odd), then the factor $\Psi_{T}^+$ is not divisible by~$p$ by Lemma~\ref{lem:divpsiep}~(i). Further, using $R_p=(\abs{r-b})_p$ and Lemma~\ref{lem:divpsi}~(i) together with Corollary~\ref{coro:psieppparts}~(i), we get equality of~$p$-parts $(\Psi_{R}^{\pm})_p=(\Psi_{\abs{r-b}}^{\pm})_p$, and conclude.

	We are left with $p$ being unitary for $\QQ$ and $\tilde T=\frac{T}{e_p}$ being even. Then Lemma~\ref{lem:divpsiep}~(ii) tells us that  $\Psi_{T}^+$ is not divisible by~$p$, so $$E_p=\left( \frac{\Psi_{R}^{\pm}}{\Psi_{\abs{b-r}}^{\pm}}\right)_p.$$ Moreover, using $R_p=(\abs{r-b})_p$ and that $R=T+\abs{b-r}$ and the fact that $\tilde T$ is even, %
	Lemma~\ref{lem:divpsi}~(i) together with Corollary~\ref{coro:psieppparts}~(ii)(a) yields equality of~$p$-parts \linebreak $(\Psi_{R}^{\pm})_p=(\Psi_{\abs{b-r}}^{\pm})_p$, and we are done.

	In subcases $\boldsymbol{\gamma}$ and $\boldsymbol{\delta}$, we have $E=\frac{\Psi_{T}^{-}}{2} \cdot \frac{\Psi_{R}^{\pm}}{\Psi_{\abs{b-r}}^{\mp}}$, and $p$ is unitary for $\QQ$ and $\tilde T$ is odd. Thus, the first factor $\Psi_{T}^{-}$ is not divisible by~$p$ by Lemma~\ref{lem:divpsiep}~(ii). Then we are done using $R_p=(\abs{r-b})_p$ and the fact that $\tilde T$ is odd, Lemma~\ref{lem:divpsi}~(i) and Corollary~\ref{coro:psieppparts}~(ii)(b).

	We are left to deal with Cases~IIIa--b where the arguments are essentially the same as before. Noticing that $R'=T'+b'-r'$ and assuming that $e_p$ divides~$R'$, we see that $R'$ and $\abs{r'-b'}$ have the same $p$-part  by Lemma~\ref{lem:nepeqquantities2}~(i)--(ii). Then, the reasoning goes along the lines of Cases~IIa--bb using the usual substitutions.

	\subparagraph*{\textit{$q$-part of~$E$}}
	As before, only factors $\Psi_k^{\pm}$ with $e_q$ diving $k$ can contribute non-trivially to the $q$-part of~$E$.

	Let us focus on Cases~IIa--bb first. By Lemma~\ref{lem:nepeqquantities}~(iii),
	$T$ is not divisible by~$e_q$ thus $\Psi_{T}^{\pm}$ are neither by Lemma~\ref{lem:divpsi}~(i).
	Thus, we have
	$$E_q=\begin{cases}
			\left(\frac{\Psi_{R}^{+}}{\Psi_{\abs{b-r}}^{\pm}}\right)_q, & \text{ in subcases } \boldsymbol{\alpha},\, \boldsymbol{\gamma}, \\
			\left(\frac{\Psi_{R}^{-}}{\Psi_{\abs{b-r}}^{\pm}}\right)_q, & \text{ in subcases } \boldsymbol{\delta},\, \boldsymbol{\beta}.
		\end{cases}$$ For brevity, we write $\pm$ in the exponent, where $\pm$ means that $+$ and $-$ are attained in the subcase listed first and second, respectively. 
	Recall that $R=me_q$ and that $q$ is odd.

	\noindent Now, let us look at subcases $\boldsymbol{\alpha}$ and $\boldsymbol{\gamma}$. Since~$q$ is linear for $\QQ$, or $q$ is unitary for $\QQ$ and $m$ is even, we get that $\Psi_{R}^{+}$ is not divisible by~$q$ by Lemma~\ref{lem:divpsiep}~(i) or~(ii), respectively, and we are done.

	\noindent In subcases $\boldsymbol{\beta}$ and $\boldsymbol{\delta}$ the prime $q$ is unitary for $\QQ$ and $m$ is odd. Again, Lemma~\ref{lem:divpsiep}~(ii) yields that $\Psi_{R}^{-}$ is not divisible by~$q$, hence $E$ is neither.

	In Cases~IIIa, we argue exactly the same way as before using the usual substitutions (with $m$ replaced by~$m-1$ in our reasoning since~$R=me_q$ is replaced by~$R'=(m-1)e_q$) and that $T'$ is not divisible by~$e_q$ according to Lemma~\ref{lem:nepeqquantities2}~(iv).

	However, in Case~IIIb, $T'$ could still be divisible by~$e_q$. Then we argue as follows: First, we can still argue the same way as before that $\Psi_{R'}^{+}$ in subcases $\boldsymbol{\alpha'}$ and $\boldsymbol{\gamma'}$, and  $\Psi_{R'}^{-}$ in subcases $\boldsymbol{\beta'}$ and $\boldsymbol{\delta'}$ is not divisible by~$q$. Thus, we get that
	$$E_q=\begin{cases}
			\left(\frac{\Psi_{T'}^{\pm}}{\Psi_{\abs{b'-r'}}^{\pm}}\right)_q, & \text{ in subcases }\boldsymbol{\alpha'},\, \boldsymbol{\gamma'}, \\
			\left(\frac{\Psi_{T'}^{\pm}}{\Psi_{\abs{b'-r'}}^{\pm}}\right)_q, & \text{ in subcases } \boldsymbol{\delta'}, \, \boldsymbol{\beta'}.
		\end{cases} $$
	Note that $T'=R'-(b'-r')$. We assume that $e_q$ divides~$T'$, otherwise we are done by Lemma~\ref{lem:divpsi}~(i). Then $e_q$ automatically divides~$B'\coloneqq b'-r'$ since~$R'=(m-1) e_q$.
	Now, we have to compare $q$-parts of~$T'$ and $B'$ to apply Corollary~\ref{coro:psieppparts} and conclude that $E$ is not divisible by~$q$. Recall that $q$ divides~$m-1$ by definition of Case~IIIb. Thus, $q$ divides~$T'$ if and only if $q$ divides~$B'$.
	We have $B'= b'-r'<2r'-r'=r'=e_q+r$ by Lemma~\ref{lem:nepeqquantities2}~(iii), thus $B<2 e_q<qe_q$, and moreover $e_q$ divides~$B$ by the above assumption.
	Thus, $B'$ is not divisible by~$q$ and $T'$ is neither. In particular, $T'$ and $B'$ have the same $q$-part. Now, we use Corollary~\ref{coro:psieppparts} to deduce that $E_q=1$; in subcases $\boldsymbol{\alpha'}$ and $\boldsymbol{\gamma'}$ we use~(i) if $q$ is linear and (ii)(a) otherwise, and in the remaining subcases~(ii)(b) gives the result.
	In total, we conclude that the degree $d^\Lambda$ is not divisible by~$p$ nor $q$.

	All in all, the non-trivial unipotent character $\chi^\Lambda$ lies in $\irrprime p {\blprinc{p}{G}} \cap \irrprime q {\blprinc{q}{G}}$. %
\end{proof}

We conclude that Theorem~\ref{thm:A} holds for any finite simple group of type $\type{B}_n$ or $\type{C}_n$.
\begin{corollary} \label{coro:pqtypesBC}
	Let $S$ be a finite simple group of type $\type{B}_n$ %
	or $\type{C}_n$ with $n\geq 2$, %
	defined over a finite field $\F_\QQ$.
	Let $p$ and $q$ be distinct primes such that $p$ and $q$ do not divide~$\QQ$ and $pq$ divides~$\abs S$. %
	Then, we have $$\irrprime p {\blprinc{p}{S}} \cap \irrprime q {\blprinc{q}{S}}\neq\{\mathbbm{1}_S\}.$$
\end{corollary}

\begin{proof} Note that $S$ is of the form $S=G/\zent G$ where $G=\GG_{sc}^F$ is the fixed point set of a simple simply connected group $\GG_{sc}$ of type $\type{B}_n$ or $\type{C}_n$, defined over $\overline{ \F_\QQ}$, under a suitable Frobenius endomorphism $F\colon \GG_{sc} \to \GG_{sc}$ defining an $\F_\QQ$-structure on $\GG_{sc}$. %
	By Proposition~\ref{pqtypeBC}, we have $\irrprime p {\blprinc p G} \cap \irrprime q {\blprinc q G}\neq \{\mathbbm{1}_G\}$ and the analogous result for $S$ follows by the discussion at the beginning of Section~\ref{sec:classicaltypes}. Note that the whole reasoning there does not require the group $S$ to be simple, just of the form $S=G/\zent G$.
\end{proof}

\begin{remark}
	In conclusion, in Corollary~\ref{coro:pqtypesBC} we reprove \cite[Proposition 3.7]{NRS} for groups of type $\type{B}_n$ or $\type{C}_n$ and $2\in \{p,q\}$ and extend this result to odd primes.
	For~$p=2$, our proof simplifies considerably and the only subcases appearing in the proof of Proposition~\ref{pqtypeBC} are $\boldsymbol{\alpha}$, $\boldsymbol{\beta}$ and $\boldsymbol{\alpha'}$, $\boldsymbol{\beta'}$, respectively (according to Table~\ref{tab:subcases} and Table~\ref{tab:subcasesprime}).
\end{remark}

\section*{Conclusion and open questions}
We eventually prove Theorem~\ref{thm:A}.
\begin{proof}[Proof of Theorem~\ref{thm:A}]
	This follows directly from Proposition \ref{pqSn}, Proposition \ref{pqAn}, Corollary \ref{coro:pqtypeA}, and Corollary \ref{pqtypeBC}.
\end{proof}

Finally, note that Conjecture (NRS) still remains open for finite simple groups of type~$\type{D}_n$ and $\tw{2}{\type{D}_n}$ (and distinct non-defining odd primes). Many cases are already settled by the author and a subsequent paper is planned.\newline
As a next step we aim to move from simple to almost simple groups and investigate the conjecture for almost simple groups with socle a classical group of Lie type.

\thispagestyle{empty}
\phantomsection %

\bibliographystyle{alpha-abbrvsort}  %
\bibliography{bib_pq}
\end{document}